\documentclass[10pt]{article}

\usepackage[latin1]{inputenc}
\usepackage{latexsym}
\usepackage{amsmath,amssymb,amsfonts,amsthm}
\usepackage{mathrsfs}
\usepackage{booktabs}
\usepackage{dsfont}

\usepackage{tikz}
\usepackage{pgfplots}
\pgfplotsset{compat=1.12}
\usetikzlibrary{patterns}

\usepackage{enumitem}

\usepackage[toc,page]{appendix}

\usepackage[numbers,comma,sort]{natbib}
\usepackage{stmaryrd}
\usepackage{tikz}
\definecolor{db}{RGB}{0, 0, 130}
\usepackage[colorlinks=true,citecolor=db,linkcolor=black,urlcolor=black,pdfstartview=FitH]{hyperref}

\usepackage{pdfsync}
\usepackage{breqn}

\numberwithin{equation}{section}

\newcommand{\R}{\mathbb{R}}

\newcommand{\N}{\mathbb{N}}
\newcommand{\EE}{\mathbb{E}}

\newcommand{\PP}{\mathbb{P}}

\newcommand{\diff}{\, d}

\newcommand{\B}{\mathcal{B}}

\newcommand{\C}{\mathbf{C}}
\newcommand{\M}{\mathbf{M}}

\def \limbashaut#1#2#3{\mathrel{\mathop{\kern 0pt#1}\limits_{#2}^{#3}}}

\DeclareMathOperator{\var}{Var} 

\makeatletter
\def\namedlabel#1#2{\begingroup
    #2%
    \def\@currentlabel{#2}%
    \phantomsection\label{#1}\endgroup
}
\makeatother

\makeatletter
\newcommand{\myitem}[1]{%
\item[#1]\protected@edef\@currentlabel{#1}%
}
\makeatother

\newtheorem{definition}{Definition}[section]
\newtheorem{theorem}[definition]{Theorem}

\newtheorem{corollary}[definition]{Corollary}

\newtheorem{lemma}[definition]{Lemma}

\newtheorem{proposition}[definition]{Proposition}
\newtheorem{remark}[definition]{Remark}

\author{Ludovic Gouden\`ege\footnote{F\'ed\'eration de Math\'ematiques de CentraleSup\'elec, CNRS FR-3487; Universit\'e Paris-Saclay, France, \texttt{goudenege@math.cnrs.fr}.}  \and
El Mehdi Haress\footnote{Universit\'e Paris-Saclay, CentraleSup\'elec, MICS, France, \texttt{el-mehdi.haress@centralesupelec.fr}. EH acknowledges the support of the Labex de Math\'ematiques Hadamard.} \and 
Alexandre Richard\footnote{Universit\'e Paris-Saclay, CentraleSup\'elec and CNRS FR-3487, France, \texttt{alexandre.richard@centralesupelec.fr.}\newline
This work is supported by the SIMALIN project ANR-19-CE40-0016 and the SDAIM project ANR-22-CE40-0015 from the French National Research Agency.
}} 

\medskip

\date{\empty}

\title{ \Large{\textbf{Numerical approximation of SDEs with fractional noise and distributional drift}}} 

\begin{document}

\maketitle

\begin{abstract}
We study the numerical approximation of SDEs with singular drifts (including distributions) driven by a fractional Brownian motion.
Under the Catellier-Gubinelli condition that imposes the regularity of the drift to be strictly greater than $1-1/(2H)$, we obtain an explicit rate of convergence of a tamed Euler scheme towards the SDE, extending results for bounded drifts. Beyond this regime, when the regularity of the drift is $1-1/(2H)$, we derive a non-explicit rate. As a byproduct, strong well-posedness for these equations is recovered.
Proofs use new regularising properties of discrete-time fBm and a new critical Gr\"onwall-type lemma. We present examples and simulations.
\end{abstract}


\medskip

\noindent\textbf{MSC2020 subject classification: }60H10, 65C30, 60G22, 60H50, 34A06.

\section{Introduction}\label{secintrosde}
We are interested in the numerical approximation of the following $d$-dimensional SDE:
\begin{align}\label{eqSDE}
 X_t = X_0 + \int_0^t b(X_s) \diff s + B_t  , \quad t\in [0,1],
\end{align}
where $X_0 \in \mathbb{R}^d$, $b$ is a distribution in some nonhomogeneous Besov space $\B_\infty^\gamma:= \mathcal{B}_{\infty,\infty}^\gamma(\R^d,\R^d)$ and $B$ is an $\mathbb{R}^d$-fractional Brownian motion (fBm) with Hurst parameter $H$. When $B$ is a standard Brownian motion ($H=1/2$), this equation received a lot of attention when the drift is irregular, see for instance \cite{Zvonkin,Veretennikov} for bounded measurable drift or \cite{KrylovRockner} under some integrability condition. Strong well-posedness was obtained in those cases, which contrasts with the non-uniqueness and sometimes non-existence that can happen for the corresponding equations without noise. Distributional drifts are even allowed when the SDE is either $1$-dimensional or if weak solutions are considered \cite{BassChen,FRW,DelarueDiel,FIR}.
In case $B$ is a fractional Brownian motion, several recent results were proven by \citet{NualartOuknine} for H\"older continuous drifts, then by \citet{Banos}, \citet{CatellierGubinelli}, \citet{GHM}, \citet{anzeletti2021regularisation} and \citet{GaleatiGerencser}, which show in particular that strong well-posedness holds    for some classes of distributional drifts when the Hurst parameter is smaller than $1/2$.

The simplest approximation scheme for \eqref{eqSDE} is given by the Euler scheme with a time-step~$h$
\begin{align*}
X_t^{h} = X_0+ \int_0^t b(X^{h}_{r_h}) \diff r + B_t  , \quad t\in [0,1],
\end{align*}
where $r_h = h \lfloor \frac{r}{h} \rfloor$. 
For the numerical analysis of Brownian SDEs with smooth coefficients, including the previous scheme and higher-order approximations, we point to a few classical works by Pardoux, Talay and Tubaro~\cite{pardoux1985discretization,talay1990expansion}, see also \cite{kloeden1992stochastic}. The strong error $ \|X_{t} - X_{t}^h\|_{L^m(\Omega)}$ is known to be of order $h$ (and $h^{1/2}$ when the noise is multiplicative). 
When the coefficients are irregular, \citet{dareiotis2021quantifying} obtained recently a strong error with the optimal rate of order $1/2$ for merely bounded measurable drifts, even if the noise is multiplicative. This was extended to integrable drifts with a Krylov-R\"ockner condition by \citet{le2021taming}. We also refer to the review \cite{szolgyenyi2021stochastic} and references therein for discontinuous coefficients, and the recent weak error analysis of \citet{JourdainMenozzi} for integrable drifts. Besides, we mention that when the drift is a distribution in a Bessel potential space with negative regularity,  \citet{de2019numerical} have obtained a rate of convergence for the so-called virtual solutions of a (Brownian) SDE, using a $2$-step mollification procedure of the drift.

Let us now recall briefly what is known when $B$ is a fractional Brownian motion. First, \citet{NN} considered one-dimensional equations with $H>1/2$, smooth coefficients and multiplicative noise, \emph{i.e.} the more general case with $B$ replaced by a symmetric Russo-Vallois~\cite{RussoVallois} integral $\int_{0}^t \sigma(X_{s}) \, d^oB_{s}$ in \eqref{eqSDE}. They proved that the rate of convergence for the strong error is exactly of order $2H-1$. 
 Then \citet{HuLiuNualart} introduced a modified Euler scheme to obtain an improved convergence rate of order $2H-1/2$, still in the multiplicative case. They also derived an interesting weak error rate of convergence. 
Recently, \citet{butkovsky2021approximation} considered \eqref{eqSDE} with any Hurst parameter $H\in (0,1)$ and drifts which are H\"older continuous functions in $\mathcal{C}^\gamma$, for $\gamma \in [0,1]$. They obtained the strong error convergence rate $h^{(1/2+\gamma H) \wedge 1 - \varepsilon}$, which holds whenever $\gamma\geq 0$ and $\gamma>1-1/(2H)$. The latter condition is optimal in the sense that it corresponds to the existence and uniqueness result for \eqref{eqSDE} established in \cite{CatellierGubinelli}. Our main contribution in this paper is an extension of their result to distributional drifts, \emph{i.e.} to negative values of $\gamma$, including the threshold $\gamma=1-1/(2H)$. For $\gamma$ to be negative under the condition $\gamma\geq 1-1/(2H)$, we will need to assume that $H<1/2$.

~

First, we state that if $b$ is in the Besov space $\mathcal{B}^\gamma_{\infty}$ and that $\gamma > 1/2 - 1/(2H)$, then there exists a weak solution $(X,B)$ to \eqref{eqSDE} which has some H\"older regularity. This result is a direct extension of  \cite[Theorem 2.8]{anzeletti2021regularisation} to any dimension $d\geq1$ and was also recently extended to time-dependent drifts in \cite{GaleatiGerencser}. The condition $\gamma > 1/2 - 1/(2H)$ allows negative values of $\gamma$ and therefore $b$ can be a genuine distribution. Solutions to \eqref{eqSDE} are then understood as processes of the form $X_{t} = X_{0} + K_{t} + B_{t}$, where $K_{t}$ is the limit in probability of $\int_{0}^t b^k(X_{s})\, ds$, for any approximating sequence $(b^k)_{k\in \N}$.

Second, to approximate numerically \eqref{eqSDE}, we consider the  Catellier-Gubinelli \cite{CatellierGubinelli} regime, namely that $b\in \mathcal{B}^\gamma_{\infty}$ with $\gamma > 1-1/(2H)$. For a time-step $h$ and a sequence $(b^k)_{k\in \N}$ that converges to $b$ in a Besov sense, we consider the following tamed Euler scheme defined on the same probability space and with the same fBm $B$ as $X$ by
\begin{align}\label{defEulerSDE}
X_t^{h,k} = X_0+ \int_0^t b^k(X^{h,k}_{r_h}) \diff r + B_t ,
\end{align}
where $r_h = h \lfloor \frac{r}{h} \rfloor$. We prove the following inequality: for $h\in (0,1)$ and $k \in \N$,
\begin{align}\label{eq:firstbound-0}
\sup _{t \in [0,1] }\big\|X_{t}-X_{t}^{h,k}\big\|_{L^{m}(\Omega)}
&  \leq C \, \Big( \|b-b^k\|_{\mathcal{B}_\infty^{\gamma-1}} + \|b^k\|_\infty h^{\frac{1}{2}-\varepsilon} +  \|b^k\|_\infty \|b^k\|_{\mathcal{C}^1} h^{1-\varepsilon} \Big).
\end{align}
This error bound is the sum of a stability bound between processes with respective drifts $b$ and $b^k$, and a numerical error for an SDE with drift $b^k$. Then, letting $b^k$ converge to $b$ and choosing carefully $k$ as a function of $h$, we get the following rate of convergence
\begin{align}\label{eq:rate-0}
\forall h\in (0,1), \quad \sup _{t \in [0,1] }\big\|X_{t}-X_{t}^{h,k}\big\|_{L^{m}(\Omega)} & \leq C h^{\frac{1}{2(1-\gamma)}-\varepsilon}.
\end{align}
Observe that the rate $\frac{1}{2(1-\gamma)}$ is always greater than $H$ and increases with the regularity of the drift.
The general version of this result is presented in Theorem \ref{thmmain-SDE} and discussed thereafter. In particular, it can be shown that given the bound \eqref{eq:firstbound-0}, one cannot obtain a better rate than the one in \eqref{eq:rate-0}.
This extends the result of \citet{butkovsky2021approximation} to negative values of $\alpha \equiv \gamma $, and matches the $1/2-\varepsilon$ rate of convergence obtained in the  case $\gamma=0$, when $b$ is a bounded measurable function. As a consequence of this strong-error approximation, we recover previous results from \cite{CatellierGubinelli,GHM} on the strong existence and pathwise uniqueness of solutions of the SDE \eqref{eqSDE}.

Finally, beyond the Catellier-Gubinelli regime, we investigate the limit case where $b \in \mathcal{B}_p^\gamma$ for some $p \in [1,+\infty)$ with $\gamma-d/p = 1-1/(2H)$, to obtain a non-explicit rate of convergence for the strong error of the numerical approximation. As a byproduct, we again deduce that $X$ is a strong solution and that it is pathwise unique in a class of H\"older continuous processes. 

~

Our method relies on a sensitivity analysis with respect to the drift for stochastic evolutions with fractional Brownian motion. A similar idea appears in \cite{GHM} for controlling the difference between two solutions of SDEs with respective drifts $b$ and $b^k$ by some Besov norm of $b-b^k$. However here, we compare the solution of the SDE with drift $b$ and the tamed-Euler approximation with drift $b^k$. We control the difference with respect to the time-step $h$ and Besov norms of $b-b^k$ and $b^k$. In particular, our method does not rely on the Girsanov transform and allows us to treat the limit case. The proofs exploit several new regularisation properties of the $d$-dimensional fBm and of the discrete-time fBm using the stochastic sewing Lemma developed by \citet{le2020stochastic} and somehow extend Davie's lemma~\cite[Prop. 2.1]{Davie} to this framework. Namely, for functions $f$ in Besov spaces of negative regularity (resp. bounded $f$ for the discrete-time fBm), we obtain upper bounds on the moments of quantities such as $\int_{s}^t f(x+B_{r})\, dr$ (resp. $\int_{s}^t f(x+B_{r_h})\, dr$) in terms of $x$, $(t-s)$ and $h$, see Propositions~\ref{propregfBm},~\ref{propbound-E1-SDE} and \ref{propnewbound-E2}. 

The limit case $\gamma-d/p = 1-1/(2H)$, $p<+\infty$, requires a version of the stochastic sewing lemma with critical exponents that induces a logarithmic factor in the result (see \cite[Theorem 4.5]{athreya2020well} and \cite[Lemma  4.10]{FHL}). We use this lemma in Proposition~\ref{propbound-E1-SDE-critic} to prove an upper bound on the moments of $\int_s^t b^k(X_r)-b^k(X^{h,k}_r) \, dr$. This leads to the following bound for $\mathcal{E}^{h,k} = X - X^{h,k}$,
\begin{align*}
\| \mathcal{E}^{h,k}_{t} - \mathcal{E}^{h,k}_{s} \|_{L^m} &\leq C \, \left(\|\mathcal{E}^{h,k}\|_{L^\infty_{[s,t]}L^m} + \epsilon(h,k) \right) \, (t-s)^{\frac{1}{2}}  \\
&\quad + C  \, \Big( \|\mathcal{E}^{h,k}\|_{L^\infty_{[s,t]}L^m}+\epsilon(h,k) \Big) \, \left| \log \big( \|\mathcal{E}^{h,k}\|_{L^\infty_{[s,t]}L^m} + \epsilon(h,k) \big) \right| \, (t-s),
\end{align*}
for some $\epsilon(h,k) = o(1)$.
We then introduce a critical Gr\"onwall-type lemma with logarithmic factors (Lemma \ref{lemrate-critical}) which yields a control of $\|\mathcal{E}^{h,k}\|_{L^\infty_{[0,1]}L^m}$ by a power of $\epsilon(h,k)$.

\paragraph{Organisation of the paper.}
We start with definitions, notations and preliminary results in Section \ref{secnumerical-analysis-SDE}, then state our main results in Section \ref{secmain}. 
The strong convergence of the numerical scheme \eqref{defEulerSDE} to the solution of \eqref{eqSDE} is established in Section \ref{secoverview-SDE}. This proof relies strongly on estimates and regularisation lemmas which are stated and proven in Section \ref{secstochastic-sewing}.
In Section \ref{secsimulations}, we provide examples of SDEs that can be approximated with our result. 
We also run simulations of the scheme \eqref{defEulerSDE} and observe empirical rates of convergence which are consistent with the theoretical results. In Appendix \ref{app:sec:regfBm}, we gather some proofs based of regularisation properties of the fBm and in Appendix \ref{app:sec:gronwall1} we prove the critical Gr\"onwall-type lemma that is useful for the limit case.

\section{Framework and results}\label{secnumerical-analysis-SDE}

\subsection{Notations and definitions}\label{secnotations-SDE}

In this section, we define notations that are used throughout the paper.

\begin{itemize}
\item On a probability space $(\Omega,\mathcal{F},\mathbb{P})$, we denote by $\mathbb{F} = (\mathcal{F}_{t})_{t\in [0,1]}$ a filtration that satisfies the usual conditions. 
\item The conditional expectation given $\mathcal{F}_{t}$ is denoted by $\mathbb{E}^{t}$ when there is no risk of confusion on the underlying filtration.

\item The $L^m(\Omega)$ norm, $m \in [1,\infty]$, of a random variable $X$ is denoted by $\|X\|_{L^m}$ and the space $L^m(\Omega)$ is simply denoted by $L^m$.

\item  We write $L^\infty_{I}$ for the space of bounded measurable functions on a subset $I$ of $[0,1]$ and $L^\infty_I L^m \colon= L^\infty(I, L^m(\Omega))$. For a Borel-measurable function $f\colon\R^d\to \R^d$, denote the classical $L^\infty$ and $\mathcal{C}^1$ norms of $f$ by 
$\|f \|_\infty = \sup_{x \in \R^d} |f(x)|$ and $\|f\|_{\mathcal{C}^1} = \| f \|_\infty + \sup_{x \neq y} \frac{|f(x)-f(y)|}{|x-y|}$. 

\item We denote by $\mathcal{C}^\alpha_{I}$ the space of H\"older continuous functions and when $\alpha=0$, we use the notation $\mathcal{C}_{I}$.

\item For all $S,T \in [0,1]$, define the simplex $\Delta_{S,T} = \{ (s,t) \in [S,T], s < t \}$.

\item For a process $Z\colon \Delta_{0,1} \times \Omega \rightarrow \mathbb{R}^d$, we write
\begin{equation*}
[Z]_{\mathcal{C}^\alpha_{I} L^{m}} := \sup_{\substack{s,t \in I \\ s < t }}  \frac{\| Z_{s,t} \|_{L^{m}}}{|t-s|^{\alpha}} 
~~\mbox{and}~~
\|Z\|_{L^\infty_I L^{m}} := \sup_{\substack{s \in I  }} \| Z_{0,s}\|_{L^{m}}.
\end{equation*}

\item In applications of the stochastic sewing lemma, we will need to consider increments of $Z$, which are given for any triplet of times $(s, u, t)$ such that $ s \leq u \leq t $ by $\delta Z_{s, u, t}\colon=Z_{s, t}-Z_{s, u}-Z_{u, t}$.

\item Finally, given a process $Z\colon [0,1] \times \Omega \rightarrow \mathbb{R}^d$, $\alpha \in (0,1]$, $m \in [1,\infty)$ and $q \in [1,\infty]$, we consider the following seminorm: for any $0 \leq s \leq t \leq 1$,
\begin{align}\label{eqdefbracket}
   [Z]_{\mathcal{C}_{[s,t]}^{\alpha}L^{m,q}}\colon= \sup_{(u,v) \in \Delta_{s,t}}\frac{\|\EE^u[|Z_v-Z_u|^m]^{\frac{1}{m}}\|_{L^q}}{(v-u)^\alpha},
\end{align}
where the conditional expectation is taken with respect to the filtration the space is equipped with. 
By the tower property and Jensen's inequality for conditional expectation, we know that
\begin{align} \label{eqboundSeminorms}
   [Z]_{\mathcal{C}_{[s,t]}^{\alpha} L^m}= [Z]_{\mathcal{C}^{\alpha}_{[s,t]} L^{m,m}} \leq  [Z]_{\mathcal{C}_{[s,t]}^{\alpha} L^{m,\infty}}.
\end{align}

\item For $f \colon [0,1] \times \Omega \rightarrow \R$, the process $\delta f: \Delta_{0,1} \times \Omega \rightarrow \R^d$ is defined as $\delta f_{s,t} = f_t - f_s$. 
Moreover, for $\alpha \in[0,1]$, $m \in [2,\infty)$ and $I$ a subset of $[0,1]$, we denote by $\mathcal{C}^\alpha_{I} L^m$ the space of $L^m$-valued mappings that are $\alpha$-H\"older continuous on $I$, that is $f \in \mathcal{C}^\alpha_{I} L^m$ if
\begin{align*}
 [\delta f ]_{\mathcal{C}^\alpha_{I} L^m} =  \sup_{\substack{s,t \in I \\ t \neq s }} \frac{\| f_t-f_s \|_{L^m}}{|t-s|^{\alpha}} < +\infty.
\end{align*}
For simplicity, we write $[f]_{\mathcal{C}^\alpha_{I} L^m} \equiv  [\delta f ]_{\mathcal{C}^\alpha_{I} L^m} $ and for $q \in [1,\infty]$, we write 
\begin{align*}
 [ f ]_{\mathcal{C}^\alpha_{I} L^{m,q}} :=  \sup_{\substack{s,t \in I \\ t \neq s }} \frac{\| \EE^s [|f_t-f_s|^m ]^{\frac{1}{m}} \|_{L^q}}{|t-s|^{\alpha}} < +\infty.
\end{align*}

\end{itemize}

\paragraph{Heat kernel.} For any $t>0$, denote by $g_{t}$ the Gaussian kernel on $\mathbb{R}^d$ with variance $t$, $g_{t}(x)=\frac{1}{(2 \pi \, t)^{d/2}} \exp \left(-\frac{|x|^{2}}{2 t}\right)$, 
and by $G_{t}$ the associated Gaussian semigroup on $\mathbb{R}^d$: for $f\colon \R^d\to \R^d$,
\begin{align}\label{eqsemi-group-gaussian}
G_t f(x) = \int_{\mathbb{R}^d} g_t(x-y) \, f(y) \diff y .
\end{align}

\paragraph{Besov spaces.}
We use the same definition of nonhomogeneous Besov spaces as in \cite{bahouri2011fourier}, which we write here for any dimension $d$. 
Let $\chi,\varphi\colon \R^d\to \R$ be the smooth radial functions which are given by \cite[Proposition 2.10]{bahouri2011fourier}, with $\chi$ supported on a ball while $\varphi$ is supported on an annulus. Let $v_{-1}$ and $v$ respectively be the inverse Fourier transform of $\chi$ and $\varphi$. Denote by $\mathcal{F}$ the Fourier transform and $\mathcal{F}^{-1}$ its inverse.

The nonhomogeneous dyadic blocks $\Delta_j, j\in \N\cup\{-1\}$ are defined for any $\R^d$-valued tempered distribution $u$ by
\begin{align*}
\Delta_{-1} u  = \mathcal{F}^{-1} \left(\chi \mathcal{F}u \right)  ~~\text{ and }~~ \Delta_{j}u  = \mathcal{F}^{-1} \left(\varphi(2^{-j}\cdot) \mathcal{F}u \right)  ~\text{ for } j \ge 0.
\end{align*}
Let $\gamma \in \R$ and $p \in [1, \infty]$. We denote by $\mathcal{B}_p^\gamma$ the nonhomogeneous Besov space $\mathcal{B}_{p,\infty}^\gamma(\mathbb{R}^d, \R^d)$  of $\R^d$-valued tempered distributions $f$ such that 
\begin{align*}
\| f \|_{\mathcal{B}_p^\gamma} = \sup_{j \ge -1} 2^{j \gamma} \| \Delta_j f \|_{L^p(\R^d)} < \infty .
\end{align*}

Let $1\leq p_1 \leq p_2 \leq \infty$. The space $\mathcal{B}_{p_1}^\gamma$ continuously embeds into $\mathcal{B}^{\gamma-d(1/p_1-1/p_2)}_{p_2}$, which we write as ${\mathcal{B}_{p_1}^\gamma \hookrightarrow \mathcal{B}^{\gamma-d(1/p_1-1/p_2)}_{p_2}}$, see e.g.  
\cite[Prop.~2.71]{bahouri2011fourier}.

~

Finally, we denote by $C$ a constant that can change from line to line and that does not depend on any parameter other than those specified in the associated lemma, proposition or theorem. When we want to make the dependence of $C$ on some parameter $a$ explicit, we will write $C(a)$.

~

To give a meaning to Equation \eqref{eqSDE} with distributional drift, we first need to precise in which sense those drifts are approximated. 

\begin{definition}\label{defconv-gamma-}
Let $\gamma \in \mathbb{R}$ and $p \in [1,\infty]$. We say that a sequence of smooth bounded functions $(b^k)_{k \in \mathbb{N}}$ converges to $b$ in $\B_p^{\gamma-}$ as $n$ goes to infinity if
\begin{equation}\label{eqconv-in-gamma-}
\begin{cases}
\displaystyle \sup_{k \in \mathbb{N}} \|b^k\|_{\B_p^\gamma} \leq \|b\|_{\B_p^\gamma} < \infty, \\
\displaystyle \lim_{n \rightarrow \infty} \|b-b^k\|_{\B_p^{\gamma'}} = 0, \quad \forall \gamma' < \gamma.
\end{cases}
\end{equation}
\end{definition}
Lemma \ref{lemreg-S} gives an explicit example of sequence $b^k$ that converge to $b$ in $\B_p^{\gamma-}$, namely $b^k=G_{1/k}b$.

Following \cite{NualartOuknine}, in dimension $d=1$, we recall the notion of $\mathbb{F}$-fBm which extends the classical definition of $\mathbb{F}$-Brownian motion. There exists a one-to-one operator $\mathcal{A}_{H}$ (which can be written explicitly in terms of fractional derivatives and integrals, see \cite[Definition 2.3]{anzeletti2021regularisation}) such that for $B$ an fBm, the process $W\colon=\mathcal{A}_{H}B$ is a Brownian motion. Then we say that $B$ is an $\mathbb{F}$-fBm if $W$ is an $\mathbb{F}$-Brownian motion. In any dimension $d \ge 1$, we say that $B$ is an $\R^d$-valued $\mathbb{F}$-fBm, if the components of $B$ are independent and each of them is an $\mathbb{F}$-fBm.

\begin{definition}\label{defsol-SDE}
Let $\gamma \in \mathbb{R}$, $b \in \mathcal{B}_\infty^\gamma$, and $X_0 \in \mathbb{R}^d$. As in \cite{anzeletti2021regularisation}, we define the following notions.
\begin{itemize}
\item \emph{Weak solution:} 
a couple $((X_t)_{t \in [0,1]},(B_t)_{t \in 
[0,1]})$ defined on some filtered probability space  
$(\Omega,\mathcal{F},\mathbb{F},\mathbb{P})$ is a weak solution to \eqref{eqSDE} on 
$[0,1]$, with 
initial condition $X_0$, if 
\begin{itemize}[nosep,leftmargin=1em,labelwidth=*,align=left]
\item $B$ is an $\R^d$-valued $\mathbb{F}$-fBm\text{;}

\item $X$ is adapted to $\mathbb{F}$\text{;}

\item there exists an $\R^d$-valued process $(K_t)_{t \in [0,1]}$ such that, a.s.,
\begin{equation}\label{solution1}
X_t=X_0+K_t+B_t  ~\text{ for all } t \in [0,1] \text{;}
\end{equation}

\item for every sequence $(b^k)_{k\in \mathbb{N}}$ of smooth bounded functions converging to $b$ in $\mathcal{B}^{\gamma-}_\infty$, we have
\begin{equation}\label{approximation2}
       \sup_{t\in [0,1]}\left|\int_0^t b^k(X_r) dr 
       -K_t\right|   \underset{n\rightarrow \infty}{\longrightarrow} 0 ~\text{ in probability}.
\end{equation}
\end{itemize}
If the couple is clear from the context, we simply say that  $(X_t)_{t \in [0,1]}$ is a weak 
solution.

\item \emph{Pathwise uniqueness:} As in the classical literature on SDEs, we say that pathwise uniqueness holds if for any two solutions $(X,B)$ and $(Y,B)$ defined on the same filtered probability space with the same
fBm $B$ and same initial condition $X_0 \in \mathbb{R}^d$, $X$ and $Y$
 are indistinguishable.

\item \emph{Strong solution:} A weak solution $(X,B)$ such that $X$ is 
$\mathbb{F}^B$-adapted is called a strong solution, where $\mathbb{F}^B$ denotes the filtration generated by $B$.

\end{itemize}
\end{definition}

\subsection{Weak existence}\label{secprem}
In the regime on $\gamma \in \mathbb{R}$, $H < 1/2$ such that
\begin{align} \label{eqassumptionweak}
0 > \gamma> \frac{1}{2} -\frac{1}{2H}, \tag{H1}
\end{align}
there is existence of a weak solution to \eqref{eqSDE}. This was proven in dimension $1$ in \cite[Theorem 2.8]{anzeletti2021regularisation}, then extended to the multidimensional, time-dependent setting in \cite[Theorem 8.2]{GaleatiGerencser}. In the following theorem, we add a statement on the regularity of the solution. We omit the proof, which follows the same lines as the proof of \cite[Theorem 2.8]{anzeletti2021regularisation}.

\begin{theorem}[\cite{anzeletti2021regularisation,GaleatiGerencser}]  \label{thWP}
Let $\gamma \in \mathbb{R}$ and $b \in \mathcal{B}_\infty^\gamma$. Assume that \eqref{eqassumptionweak} holds. Then there exists a weak 
solution $X$ to \eqref{eqSDE} such that for any $m\geq 2$,
\begin{align}\label{eqregweak}
[X-B]_{\mathcal{C}^{1+H \gamma}_{[0,1]}L^{m,\infty}}<\infty.
\end{align}
\end{theorem}

The condition \eqref{eqassumptionweak} covers both the regime $\gamma > 1-1/(2H)$, referred to sub-critical case hereafter, for which there is strong existence \cite{CatellierGubinelli,GHM}, and the limit case $b \in \mathcal{B}_p^\gamma$ for some $p <+\infty$ and $\gamma-d/p=1-1/(2H)$, for which strong existence is only known in dimension one \cite{anzeletti2021regularisation}.

\begin{remark}\label{rk:regweaksol}
In the following, we will assume a condition more restrictive than \eqref{eqassumptionweak}, 
namely that $\gamma\geq 1-1/(2H)$. Under this condition and in view of the inequality~\eqref{eqboundSeminorms}, the previous theorem gives the existence of a weak solution with the following lower regularity:
\begin{align}\label{eq:weakerreg}
[X-B]_{\mathcal{C}^{\frac{1}{2}+H}_{[0,1]}L^{m}}<\infty.
\end{align}
In Theorem~\ref{thmmain-SDE} and Corollary~\ref{corbn-choice}, starting from a weak solution $X$ with regularity \eqref{eq:weakerreg}, we will show the convergence of the Euler scheme towards $X$, thus establishing uniqueness (see Corollary~\ref{corSP} for the precise statement) in the class of weak solutions with regularity~\eqref{eq:weakerreg}. Since the latter class is larger than the class of weak solutions which satisfy~\eqref{eqregweak}, this will provide a stronger uniqueness statement.
\end{remark}

\subsection{Main results}\label{secmain}
Let $(b^k)_{k \in \mathbb{N}}$ be a sequence of smooth functions that approximates $b$. Consider the tamed Euler scheme \eqref{defEulerSDE} associated to \eqref{eqSDE} with a time-step $h \in (0,1)$. The main result of this paper is the following theorem. It describes the regularity and convergence of the tamed Euler scheme.

\begin{theorem}\label{thmmain-SDE} 
Let $H < 1/2$, $\gamma \in \mathbb{R}$ satisfying
\begin{align}\label{eqcond-gamma-p-H}
0 > \gamma \geq 1-\frac{1}{2H}. \tag{H2}
\end{align}
Let $m \in [2, \infty)$. Let $b \in \mathcal{B}_\infty^\gamma$ and $(b^k)_{k \in \mathbb{N}}$ be a sequence of smooth functions such that $\sup_{k \in \mathbb{N}} \|b^k\|_{\B_\infty^\gamma} \leq \|b\|_{\B_\infty^\gamma}$.
Let $X_0$ be an $\mathcal{F}_0$-measurable random variable, $(X,B)$ be a weak solution to \eqref{eqSDE} which satisfies \eqref{eq:weakerreg}. Let $(X^{h,k})_{h \in (0,1), k \in \mathbb{N}}$ be the tamed Euler scheme defined in \eqref{defEulerSDE}, on the same probability space and with the same fBm $B$ as $X$. Let $\eta \in (0,H)$, $\mathcal{D}$ a sub-domain of $(0,1) \times \mathbb{N}$ and assume that
\begin{align}\label{eqassump-bn-bounded}
\sup_{(h,k)\in \mathcal{D}} \| b^k \|_{\infty} h^{\frac{1}{2}-H} < \infty \ \ \text{  and }~  \sup_{(h,k)\in \mathcal{D}  } \| b^k \|_{\mathcal{C}^1}  h^{\frac{1}{2}+H-\eta} < \infty. \tag{H3}
\end{align}

\begin{enumerate}[label=(\alph*)]
\item \underline{Regularity of the tamed Euler scheme}: It holds $\displaystyle \sup_{(h,k)\in \mathcal{D}} [X^{h,k}-B]_{\mathcal{C}^{\frac{1}{2}+H }_{[0,1]} L^{m, \infty}}  < \infty$. 
\end{enumerate}
Let $\varepsilon \in (0,1/2)$.
\begin{enumerate}
\item[(b)] \underline{The sub-critical case}: If $\gamma >1-1/(2H)$, there exists $C>0$ that depends on $m, \gamma, d, \varepsilon, \|b\|_{\mathcal{B}^\gamma_{\infty}}$ such that for all $(h,k)\in \mathcal{D}$, the following bound holds:
\begin{align}\label{eqmain-result-SDE}
\begin{split}
[X - X^{h,k}]_{\mathcal{C}^{\frac{1}{2}}_{[0,1]} L^{m}} 
&  \leq C \,  \Big( \|b-b^k\|_{\mathcal{B}_\infty^{\gamma-1}} + \|b^k\|_\infty h^{\frac{1}{2}-\varepsilon} +  \|b^k\|_\infty \|b^k\|_{\mathcal{C}^1} h^{1-\varepsilon} \Big).
\end{split}
\end{align}

\item[(c)] \underline{The limit case}: If $\gamma=1-1/(2H)$, assume further that there exists $p \in [1,+\infty)$ such that $b \in \mathcal{B}_p^{\gamma+d/p} \hookrightarrow \mathcal{B}^\gamma_{\infty}$. 
Denote $\tilde{\gamma} = \gamma+d/p$ and assume that $\sup_{k \in \mathbb{N}} \|b^k\|_{\B_p^{\tilde\gamma}} \leq \|b\|_{\B_p^{\tilde\gamma}}$. Let $\zeta \in (0,1/2)$, $\M$ be the constant given by Proposition \ref{propbound-E1-SDE-critic}, and $\delta \in (0, e^{-\M})$.
Then there exists $C>0$ that depends on $m, p, \gamma,d, \varepsilon, \zeta, \delta, \|b\|_{\B_p^{\tilde\gamma}}$ such that for all $(h,k)\in \mathcal{D}$, the following bound holds:
\begin{equation}\label{eqmain-result-SDE-critic}
\begin{split}
[X - X^{h,k}]_{\mathcal{C}^{\frac{1}{2}-\zeta}_{[0,1]} L^{m}} &  \leq  C \Big( \|b-b^k\|_{\mathcal{B}_p^{\tilde\gamma-1}} (1 + |\log(\|b-b^k\|_{\mathcal{B}_p^{\tilde\gamma-1}})|)  + \|b^k\|_\infty h^{\frac{1}{2}-\varepsilon} \\ 
& \quad +  \|b^k\|_{\mathcal{C}^1} \|b^k\|_\infty  h^{1-\varepsilon} \Big) ^{(e^{-\M}-\delta)} .
\end{split}
\end{equation}

\end{enumerate}

\end{theorem} 

\begin{remark}
\begin{itemize}[noitemsep]
\item The weak solution $X$ of Theorem \ref{thmmain-SDE} and the tamed Euler scheme $X^{h,k}$ start at the same point $X_0$, so the previous error bounds also hold for the strong error in uniform norm, since we have 
\begin{equation}\label{eqboundsup}
\sup _{t \in [0,1] }\big\|X_{t}-X_{t}^{h,k}\big\|_{L^{m}} \leq [X - X^{h,k}]_{\mathcal{C}^{\frac{1}{2}-\zeta}_{[0,1]} L^{m}} .
\end{equation}
\item Note here that we make sense of the strong error while only manipulating a weak solution since we assumed the tamed Euler scheme to be defined in the same probability space with the same driving fBm.
\item The nature of the terms that appear in the upper bounds \eqref{eqmain-result-SDE}-\eqref{eqmain-result-SDE-critic} is discussed in Section~\ref{sec:rate-discussion}.

\end{itemize}
\end{remark}

In the upper bounds \eqref{eqmain-result-SDE}-\eqref{eqmain-result-SDE-critic}, it is important to choose carefully the sequence $(b^k)_{k \in \mathbb{N}}$ to obtain a good rate of convergence of the numerical scheme. The following corollary states assumptions on $(b^k)_{k \in \mathbb{N}}$ under which we obtain an explicit rate of convergence. These assumptions \eqref{eqbn-inf}-\eqref{eqbn-inf}-\eqref{eqbn-b}-\eqref{eqbn-b-critic} hold for example if $b^k= G_{1/k} b$ for $k \in \mathbb{N}^*$ (see Lemma \ref{lemreg-S}).

\begin{corollary}\label{corbn-choice}
Let the assumptions of Theorem \ref{thmmain-SDE} hold, and recall that $(X,B)$ is a weak solution to \eqref{eqSDE} which satisfies \eqref{eq:weakerreg} and $(X^{h,k})_{h \in (0,1), k \in \mathbb{N}}$ is the tamed Euler scheme defined on the same probability space and with the same fBm $B$.
Assume that $(b^k)_{k \in \mathbb{N}}$ satisfies for any $k \in \mathbb{N}^*$
\begin{align}
&\|b^k\|_\infty  \leq C\, \|b\|_{\mathcal{B}_\infty^\gamma}  \ k^{-\frac{1}{2}\gamma }, \label{eqbn-inf} \\
&\| b^k \|_{\mathcal{C}^1}  \leq C\, \|b\|_{\mathcal{B}_\infty^\gamma}  \, k^{\frac{1}{2}-\frac{1}{2}\gamma} \label{eqbn-C1} .
\end{align}
For $h \in (0, 1/2)$, define
$k_h = \left\lfloor h^{-\frac{1}{1-\gamma}}\right\rfloor$.
 Then we have
 \begin{align}
 &\sup_{\substack{h \in (0,\frac{1}{2})}} [X^{h,k_h}-B]_{\mathcal{C}^{\frac{1}{2}+H}_{[0,1]} L^{m, \infty}} < \infty \label{equnifscheme} .
 \end{align}

\begin{enumerate}[label=(\alph*)]
\item \underline{The sub-critical case}: Let $\varepsilon \in (0,1/2)$. If $\gamma \in (1-1/(2H) ,0)$ and
\begin{align}
&\|b-b^k\|_{\mathcal{B}_\infty^{\gamma-1}}  \leq C\, \|b\|_{\mathcal{B}_\infty^\gamma} \  k^{-\frac{1}{2}},  \label{eqbn-b} 
\end{align} 
then there exists $C>0$ that depends on $ m, \gamma, \varepsilon,  \|b\|_{\mathcal{B}_\infty^{\gamma}}$ such that the following bound holds:
\begin{align}
\forall h\in ( 0,\tfrac{1}{2} ),\quad   [X - X^{h,k_h}]_{\mathcal{C}^{\frac{1}{2}}_{[0,1]} L^{m}}  \leq C \, h^{\frac{1}{2(1-\gamma)}-\varepsilon}. \label{eqrate1}
\end{align}

\item \underline{The limit case}: If $\gamma=1-1/(2H)$, assume further that there exists $p \in [1,+\infty)$ such that $b \in \mathcal{B}_p^{\gamma+d/p} \hookrightarrow \mathcal{B}^\gamma_{\infty}$. 
Denote $\tilde{\gamma} = \gamma+d/p$. Assume that $\sup_{k \in \mathbb{N}} \|b^k\|_{\B_p^{\tilde\gamma}} \leq \|b\|_{\B_p^{\tilde\gamma}}$ and that
\begin{align}
&\|b-b^k\|_{\mathcal{B}_p^{\tilde\gamma-1}}  \leq C\, \|b\|_{\mathcal{B}_p^{\tilde\gamma}} \  k^{-\frac{1}{2}} \label{eqbn-b-critic} .
\end{align}
Let $\zeta \in (0,1/2)$, $\M$ be the constant given by Proposition \ref{propbound-E1-SDE-critic}, and $\delta \in (0, e^{-\M})$. Then there exists $C>0$ that depends on $m, p, \gamma, \zeta, \delta, \|b\|_{\mathcal{B}_p^{\tilde\gamma}}$ such that the following bound holds:
\begin{align}
\forall h\in ( 0,\frac{1}{2} ),\quad   [X - X^{h,k_{h}}]_{\mathcal{C}^{\frac{1}{2}-\zeta}_{[0,1]} L^{m}}  \leq C \, h^{H (e^{-\M}-\delta)}. \label{eqrate1-critic}
\end{align}

\end{enumerate}
\end{corollary}

\begin{remark}
 For instance, if each component of $b$ is a signed measure, then $b \in \mathcal{B}_1^{0} \hookrightarrow \mathcal{B}_p^{-d+d/p}$ (see \cite[Proposition 2.39]{bahouri2011fourier}). Hence the previous result (Corollary~\ref{corbn-choice}$(a)$) yields a rate $\frac{1}{2(1+d)}-\varepsilon$, which holds for $H < \frac{1}{2(1+d)}$. In the limit case, when $H=\frac{1}{2(1+d)}$, the rate becomes $H e^{-\M}-\varepsilon$.
\end{remark}

The next corollary follows from the convergence of the tamed Euler scheme stated in Corollary~\ref{corbn-choice}: since the scheme is adapted to $\mathbb{F}^B$ and converges to any weak solution $X$ which satisfies \eqref{eq:weakerreg},
we deduce both uniqueness and that the weak solution $X$ is adapted to $\mathbb{F}^B$ (it is therefore a strong solution). 

\begin{corollary}\label{corSP}
Let $\gamma \in \mathbb{R}$ and $b \in \mathcal{B}_\infty^\gamma$. Assume that \eqref{eqcond-gamma-p-H} holds.
\begin{enumerate}[label=(\alph*)]
\item \underline{The sub-critical case}: If $\gamma >1-1/(2H)$, there exists a strong solution $X$ to \eqref{eqSDE} such that $[X-B]_{\mathcal{C}^{1/2+H}_{[0,1]}L^{m,\infty}}<\infty$ for any $m\geq 2$ and pathwise uniqueness holds in the class of strong solutions $Y$ that satisfy ${[Y-B]_{\mathcal{C}^{1/2}_{[0,1]} L^{2}}<\infty}$.
\item \underline{The limit case}: If $\gamma=1-1/(2H)$, assume further that there exists $p \in [1,+\infty)$ such that $b \in \mathcal{B}_p^{\gamma+d/p} \hookrightarrow \mathcal{B}^\gamma_{\infty}$. Then there exists a strong solution $X$ to \eqref{eqSDE} such that $[X-B]_{\mathcal{C}^{1/2+H}_{[0,1]}L^{m,\infty}}<\infty$ for any $m\geq 2$ and pathwise uniqueness holds in the class of strong solutions $Y$ that satisfy ${[Y-B]_{\mathcal{C}^{1/2+H}_{[0,1]} L^{2}}<\infty}$.
\end{enumerate}
\end{corollary}

The proof of Corollary \ref{corSP} is given in Section~\ref{subsecStrongEx}.

\begin{remark}
~
\begin{itemize}
\item It is now possible to construct the tamed Euler scheme on any probability space (rich enough to contain an $\mathbb{F}$-fBm), which is of practical importance for  simulations.
\item The Euler scheme selects the unique solution in the class of solutions $Y$ such that ${[Y-B]_{\mathcal{C}^{1/2}_{[0,1]}L^{2}}<\infty}$ in the sub-critical case, and in the class of solutions $Y$ such that ${[Y-B]_{\mathcal{C}^{1/2+H}_{[0,1]}L^{2}}<\infty}$ in the limit case.
\end{itemize}
\end{remark}

For $\gamma > 0$, $\mathcal{B}_\infty^\gamma$ is continuously embedded in the H\"older space $\mathcal{C}^{\gamma}$. In \cite{butkovsky2021approximation}, it was proved that the Euler scheme achieves a rate $1/2+H\gamma-\varepsilon$. Moreover, if $b$ is a bounded measurable function, the rate is $1/2-\varepsilon$. To close the gap between the present results and \cite{butkovsky2021approximation}, we handle the case $\gamma=0$. Note that $\mathcal{B}_\infty^0$ contains strictly $L^\infty(\R^d)$ (see e.g. \cite[Section 2.2.2, eq (8) and Section 2.2.4, eq (4)]{runst2011sobolev}) which was the space considered in \cite{butkovsky2021approximation}. 
Let $b \in \mathcal{B}_\infty^0$. By the definition of Besov spaces, we know that $b \in \mathcal{B}_\infty^{-\eta}$ for all $\eta>0$. Choosing $\eta$ small enough so that $- \eta > 1-1/(2H)$, we can apply Theorem \ref{thmmain-SDE} and obtain a rate of convergence as in Corollary \ref{corbn-choice}. This is summarized in the following corollary.

\begin{corollary}\label{corgama=d/p}
Let the assumptions of Theorem \ref{thmmain-SDE} hold. Let $B$ be an $\mathbb{F}$-fBm with $H < 1/2$, $b \in \mathcal{B}_\infty^0$ and $m \ge 2$. There exists a strong solution $X$ to \eqref{eqSDE} such that $[X-B]_{\mathcal{C}^{1/2+H}_{[0,1]}L^{m,\infty}}<\infty$. Besides, pathwise uniqueness holds in the class of solutions $X$ such that $[X-B]_{\mathcal{C}^{1/2}_{[0,1]} L^{2}}<\infty$. 

Let $\varepsilon \in (0,1/2)$.
Then there exists a constant $C$ that depends only on $m, \varepsilon,  \|b\|_{\mathcal{B}_\infty^{0}}$ such that for any $h \in (0,1/2)$ and $k\in \N$, the following bound holds:
\begin{align*}%
[X - X^{h,k}]_{\mathcal{C}^{\frac{1}{2}}_{[0,1]} L^{m}} &  \leq C \left( \|b-b^k\|_{\mathcal{B}_{\infty}^{-1}} + \|b^k\|_\infty h^{\frac{1}{2}-\varepsilon} +  \|b^k\|_\infty \|b^k\|_{\mathcal{C}^1} h^{1-\varepsilon} \right).
\end{align*}
Moreover, for $k_h=\lfloor h^{-1} \rfloor$ and $(b^{n})_{k \in \mathbb{N}}$ satisfying \eqref{eqbn-b}, \eqref{eqbn-inf} and \eqref{eqbn-C1}, we have
\begin{align*}
 \forall h\in ( 0,\tfrac{1}{2}), \quad  [X - X^{h,k_h}]_{\mathcal{C}^{\frac{1}{2}}_{[0,1]} L^{m}} & \leq C h^{\frac{1}{2}-\varepsilon},  \\
\sup_{\substack{h \in (0,\frac{1}{2})}} [X^{h,k_h}-B]_{\mathcal{C}^{\frac{1}{2}+H}_{[0,1]} L^{m, \infty}} & < \infty  .
\end{align*}
Furthermore, under \eqref{eqassump-bn-bounded}, we have $\displaystyle \sup_{(h,k)\in \mathcal{D}} [X^{h,k}-B]_{\mathcal{C}^{\frac{1}{2}+H }_{[0,1]} L^{m, \infty}}  < \infty$.
\end{corollary}

Theorem \ref{thmmain-SDE}, Corollary \ref{corbn-choice} and Corollary \ref{corgama=d/p} are proven in Section \ref{secoverview-SDE}. 

\subsection{Discussion on the results}\label{sec:rate-discussion}
First we discuss the optimality of the rate of convergence obtained from the bound \eqref{eqmain-result-SDE}, and then we comment on this rate.

\paragraph{Discussion on the optimality of the rate of convergence.}

In the sub-critical regime, the bound \eqref{eqmain-result-SDE} consists of two terms: the first term $\|b-b^k\|_{\mathcal{B}_\infty^{\gamma-1}}$ represents a stability error which is due to the taming part of the Euler scheme and the $\mathcal{B}_\infty^{\gamma-1}$ norm is the expected scale to control such error (see also \cite[Theorem 3.2]{GaleatiGerencser}).
As for the second term $\|b^k \|_{\infty} h^{1/2-\varepsilon} + \| b^k\|_{\infty} \| b^k \|_{\mathcal{C}^1} h^{1-\varepsilon}$, it is of the same order as $\|b^k \|_{\infty} h^{1/2-\varepsilon}$ for some simple examples (see Section \ref{subsecCor2.5}). The latter expression represents the optimal error of the Euler scheme for SDEs with bounded drifts so we also expect such a term to arise in our context. Hence we believe that the bound \eqref{eqmain-result-SDE} cannot be improved for generic drifts in $\mathcal{B}_{\infty}^\gamma,\ \gamma \leq 0$.
Given the bound \eqref{eqmain-result-SDE}, given any sequence $(b^{k})_{k\in \N}$ that approximates $b$ and any sequence $(k_{h})_{h\in \N}$ such that $k_{h} \to +\infty$ as $h \to 0$, one cannot improve the rate from \eqref{eqrate1}. This is proven in \cite[Section 2.4]{GHR24} by exhibiting an example of singular function $b$ for which the bound \eqref{eqmain-result-SDE} is always larger or equal to the rate \eqref{eqrate1}. This also means that the choice of $b^k=G_{\frac{1}{k}}b$ and $k_h = h^{\frac{\gamma}{2(1-\gamma)}}$ yields the optimal rate.

\paragraph{Comments on the rate of convergence.}
The main novelty of this paper is that we approximate numerically fractional SDEs with distributional drifts, including the limit case where $\gamma=\tilde\gamma-d/p=1-1/(2H)$, $p<+\infty$ and $b \in \B_p^{\tilde{\gamma}}$. 
The orders obtained here and in \cite{butkovsky2021approximation} can be read in Figure \ref{figorders} and are summarized in Table \ref{tabtrue-summarySDE}. Let us make a few comments.
\begin{itemize}
\item For $\gamma <0$, the largest value of $H$ one can chose is $\frac{1}{2(1-\gamma)}-\varepsilon$. This yields an order of convergence $\frac{1}{2(1-\gamma)}-\varepsilon \approx H-\varepsilon$.
\item For a fixed $H<1/2$, the lowest regularity one can take for $b$ is $b \in \mathcal{B}_\infty^{1-\frac{1}{2H}+\varepsilon}$ for any $\varepsilon>0$, and get an order of convergence that is approximately $H-\varepsilon$.
\item The order of convergence is $1/2-\varepsilon$ when $\gamma = 0$, for any $H < 1/2$.
\item If $H<1/2$ and  $b \in \mathcal{B}_p^{\tilde\gamma}$ for some $p<+\infty$ with $\tilde\gamma-d/p=1-1/(2H)$, the order of convergence is $H e^{-\M}-\varepsilon$ for some constant $\M$ (not displayed on Figure \ref{figorders}). 
\end{itemize}

When $\gamma \ge 0$, the rate of convergence that was obtained in \cite{butkovsky2021approximation} is $(1/2+ H\gamma )\wedge 1 -\varepsilon$. On the boundary $\gamma=0$, this rate coincides with the one we obtained, that is $1/2-\varepsilon$. 
Observe that \eqref{eqcond-gamma-p-H} implies the following inequality
\begin{align}\label{eqtworates}
\frac{1}{2}+ H \gamma  \ge \frac{1}{2(1-\gamma)} .
\end{align}
Therefore, the rate $(1/2+ H\gamma )\wedge 1 -\varepsilon$ from \cite{butkovsky2021approximation} does not directly extend to negative values of $\gamma <0$, although when $\gamma $ is $\varepsilon$-close to $1-1/(2H)$, there is almost equality in \eqref{eqtworates}.

All these observations are confirmed by our numerical simulations (see Section~\ref{subsecsim}). In particular, we observe that for an irregular drift smoothed by a Gaussian kernel, the order of convergence seems indeed close to $1/(2(1-\gamma))$ and that it does not change for all admissible values of $H$.

\definecolor{mydarkgreen}{RGB}{255,0,0}
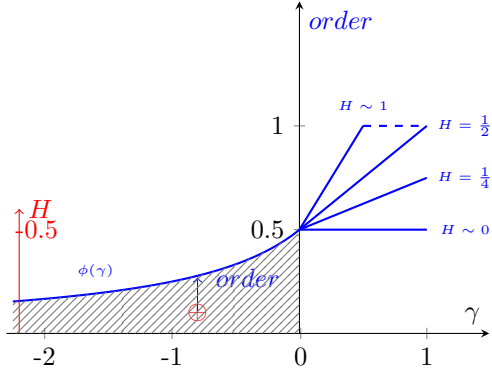
\begin{figure}[h!]
\begin{center}
\begin{tikzpicture}
\begin{axis}[
    axis lines=middle,
    xmin=-2.3,ymin=0,
    xmax=1.49,ymax=1.6,
    xlabel=$\gamma$,
    ylabel=$\color{blue}order$,
    samples=100,      
    width=8cm,        
    height=6cm,        
    xlabel shift=10pt,
    xtick={-2,-1,0.01,1},  
    xticklabels={-2, -1, 0, 1},  
        ytick={0.5,1},  
    yticklabels={0.5,1},  
]
\addplot[domain=-2.25:0, smooth, blue, thick] {1/(2*(1-x))};
\addplot[domain=0:1, smooth, blue, thick] {1/2};
\addplot[domain=0:1, smooth, blue, thick] {1/2+0.25*x};
\addplot[domain=0:1, smooth, blue, thick] {1/2+0.5*x};
\addplot[domain=0:0.5, smooth, blue, thick] {1/2+1*x};
\addplot[domain=0.5:1.0, smooth, blue, thick, dashed] {1};
\node (n9) at (-1.6,0.3) {\color{blue}\tiny$\phi(\gamma)$};
\node (n0) at (1.3,0.50) {\color{blue}\tiny$H\sim 0$};
\node (n1) at (1.3,0.75) {\color{blue}\tiny$H=\frac{1}{4}$};
\node (n2) at (1.3,1.0) {\color{blue}\tiny$H=\frac{1}{2}$};
\node (n4) at (0.5,1.1) {\color{blue}\tiny$H\sim 1$}; 
\draw[mydarkgreen] (-0.8,0.1) node {$\oplus$};
\draw[->, blue] (-0.8,0.1) -- (-0.8,0.27) node[right] {$\,\, order$};
\draw[->, mydarkgreen] (-2.2,0) -- (-2.2,0.6) node[right] {$H$};
\draw[mydarkgreen] (-2.2,0.5) node {\,\,\,\,\,\,\,\,-0.5};
\addplot[draw=none, postaction={pattern=north east lines, pattern color=gray}, domain=-2.25:0] {1/(2*(1-x))} \closedcycle; 
\end{axis}

\end{tikzpicture}
\end{center}
\caption[Rate of convergence summary for SDEs]{The hashed region represents admissible values of $\gamma$ and $H$ under \eqref{eqcond-gamma-p-H}, when $\gamma <0$. On the blue lines, we read the order of convergence: for a fixed point $(\gamma, H)$ with $\gamma<0$ (see the red cross), one can read the order of convergence given $\phi(\gamma) \colon =\frac{1}{2(1-\gamma)}$ by projecting the point on $\phi$ from below (see the blue arrow). For $\gamma>0$, the order varies with $H$, and therefore several curves are used to represent it.}
\label{figorders}
\end{figure}

\begin{table}[htp!]
\centering
\begin{tabular}{|c|c|c|c|c|}
\hline
\textit{Drift}
& $\begin{array}{ll}
\gamma=\tilde\gamma - \frac{d}{p} = 1-\frac{1}{2H}, \\
p <+\infty~\text{and}~b\in \B_p^{\tilde\gamma}
\end{array}$
& $\gamma \in (1-\frac{1}{2H},0)$
& $\gamma = 0$ 
& $\gamma>0$\\ \hline
\textit{Rate}
& $H e^{-\M}-\varepsilon$
& $\frac{1}{2(1-\gamma)}- \varepsilon $
& $\frac{1}{2}- \varepsilon $ 
& $\Big( \frac{1}{2}+H\gamma \Big)  \wedge 1 -\varepsilon$\\ \hline
\end{tabular}
\caption{Rate of convergence of the tamed Euler scheme depending on the Besov regularity of the drift.}
\label{tabtrue-summarySDE}
\end{table}

In the proof, the difference between our rate of convergence and the rate obtained in \cite{butkovsky2021approximation} comes from the term $E^{2,h,k}$ defined in \eqref{eqdefE}. This term characterises the convergence of the Euler scheme, since it can be written as
\begin{align*}
\int_s^t \left( b^k(X^{h,k}_r) - b^k(X^{h,k}_{r_h}) \right) dr .
\end{align*}
In \cite[Lemma 4.2]{butkovsky2021approximation}, the authors use the $\| b^k \|_{\mathcal{C}^\alpha}$ norm of $b^k$ and Girsanov's theorem to obtain a convergence that depends on $H$. Here, we do not use Girsanov's theorem to avoid an exponential dependence on $\| b^k \|_{\mathcal{C}^\alpha}$, as this term could go to $+\infty$ as $k\to +\infty$ in the case of a genuine distribution. Instead, using sewing techniques we prove the bound stated in Corollary \ref{cornewbound-E2}, which does not depend directly on $H$.

\section{Convergence of the tamed Euler scheme}\label{secoverview-SDE}

In this section, we prove the main results stated in Section~\ref{secnumerical-analysis-SDE}. They rely on several general regularisation properties of the fBm and the discrete-time fBm, which are stated and proven in Section~\ref{secstochastic-sewing}.

\subsection{Proof of Theorem \ref{thmmain-SDE}}\label{secproof-mainth}

Let $\gamma$ satisfying \eqref{eqcond-gamma-p-H} and $b \in \mathcal{B}_\infty^\gamma$. In the limit case, recall that by a Besov embedding, $\mathcal{B}_p^\gamma \hookrightarrow \mathcal{B}_{\bar{p}}^{\gamma-d/p+d/\bar{p}}$ for any $\bar{p} \ge p$. Setting $\bar{\gamma}=\gamma-d/p+d/\bar{p}$, we have $b \in \mathcal{B}_{\bar{p}}^{\bar{\gamma}}$ and $\gamma-d/p = \bar{\gamma} - d/\bar{p}$. So, without any loss of generality, we can always assume that $p$ is as large as we want and in particular we can assume $p \ge m$ to apply Proposition \ref{propregfBm} and Proposition \ref{propbound-E1-SDE-critic}.

\smallskip

Let $\mathcal{D}$ be a sub-domain of $(0,1) \times \mathbb{N}$ satisfying \eqref{eqassump-bn-bounded}. 
Let $(X,B)$ be a weak solution to \eqref{eqSDE} defined on a filtered probability space $(\Omega,\mathcal{F},\mathbb{F},\PP)$ satisfying $[X-B]_{\mathcal{C}^{1/2+H}_{[0,1]} L^{m}}< +\infty$ (which we recall is satisfied by the weak solution given by Theorem \ref{thWP}). On this probability space and with the same fBm $B$, we define the tamed Euler scheme $(X^{h,k})_{h>0,k\in \N}$. 
For all $t>0$, recall from \eqref{solution1} that $K_t \colon= X_t - B_t - X_0$ and define
\begin{align}\label{eqdef-Khn}
K_{t}^k \colon= \int_0^{t} b^k(X_r) \diff r \ \text{ and } \ K^{h,k}_t \colon=  \int_0^t b^k(X^{h,k}_{r_h}) \diff r   .
\end{align}
Now set the notation for the error as
\begin{align*}
\mathcal{E}_t^{h,k} & \colon= K_{t} - K^{h,k}_{t}, \quad  t \ge 0.
\end{align*}
Let $0 \leq S \leq T \leq 1$. The error is decomposed as
\begin{align}\label{eqerror-first-bound-SDE}
[\mathcal{E}^{h,k}]_{\mathcal{C}^{\frac{1}{2}-\zeta}_{[S,T]} L^{m}} & \leq [ K - K^k ]_{\mathcal{C}^{\frac{1}{2}-\zeta}_{[S,T]} L^{m}}  + [ K^{k}-K^{h,k}]_{\mathcal{C}^{\frac{1}{2}-\zeta}_{[S,T]} L^{m}} \nonumber \\
& \leq [ K - K^k ]_{\mathcal{C}^{\frac{1}{2}-\zeta}_{[S,T]} L^{m}}  + [E^{1,h,k}]_{\mathcal{C}^{\frac{1}{2}-\zeta}_{[S,T]} L^{m}}  + [E^{2,h,k}]_{\mathcal{C}^{\frac{1}{2}-\zeta}_{[S,T]} L^{m}}   ,
\end{align}
where $\zeta=0$ in the sub-critical case, and for all $s < t$ we denote
\begin{equation}\label{eqdefE}
\begin{split}
E_{s,t}^{1,h,k} &  \colon= \int_s^t  b^k(X_0 + K_r + B_r) - b^k(X_0 + K^{h,k}_r +  B_r)  \, dr, \\
E_{s,t}^{2,h,k} &  \colon= \int_s^t b^k(X_0 + K^{h,k}_r + B_r) - b^k(X_0 + K^{h,k}_{r_h} + B_{r_h})  \, dr .
\end{split}
\end{equation}
We also denote 
\begin{align}\label{eqdefepsilonhn}
\epsilon(h,k) \colon= [K-K^k]_{\mathcal{C}^{\frac{1}{2}}_{[0,1]}L^m} + [E^{2,h,k}]_{\mathcal{C}^{\frac{1}{2}}_{[0,1]} L^{m}}.
\end{align}

\paragraph{Organisation of the proof.}

To perform the proof, we will need an \emph{a priori} estimate on the norm of the tamed Euler scheme $K^{h,k}$ (for Theorem~\ref{thmmain-SDE}$(a)$), and bounds on the norms of $K-K^k$, $E^{1,h,k}$ and $E^{2,h,k}$. These bounds rely heavily on regularising properties of the fBm and discrete-time fBm which are postponed to Section~\ref{secstochastic-sewing}. For clarity, we precise here the articulation between the various subsections of Section~\ref{secstochastic-sewing} and the current proof:
\begin{itemize}[noitemsep]
\item In Section~\ref{secproofs-SDE}, general regularising properties of the fBm are established. These are used in this proof (see next paragraph) to obtain estimates on $K-K^k$.
\item In Section~\ref{subsecbound-E2},  general regularising properties of the discrete-time fBm are established.
\item In Section~\ref{subsecreg-schema}, the results of Section~\ref{subsecbound-E2} are used to prove an \emph{a priori} estimate on  the norm of the tamed Euler scheme $K^{h,k}$.
\item In Section~\ref{subsecE1hn}, we exploit the results of Section~\ref{secproofs-SDE} and Section~\ref{subsecreg-schema} to obtain estimates on $E^{1,h,k}$, see Corollary~\ref{corbound-E1-SDE} for the sub-critical case, and Proposition~\ref{propbound-E1-SDE-critic} for the limit case.
\item In Section~\ref{subsecfurtherdtimefBm}, the aim is to show Corollary~\ref{cornewbound-E2-general}, which is a quadrature bound between a fractional process and its discrete-time approximation based on regularisation properties of the discrete-time fBm.
\item In Section~\ref{subsecreg-E2hn}, Corollary~\ref{cornewbound-E2-general} is used to obtain the H\"older regularity of $E^{2,h,k}$ in both sub-critical and limit cases, see Corollary~\ref{cornewbound-E2}.
\end{itemize}

\smallskip

The proof of Theorem \ref{thmmain-SDE}$(a)$ follows from a careful application of the general regularisation property of the discrete-time fBm stated in Proposition~\ref{propbound-Khn}. The precise result is stated and proven in Section~\ref{subsecreg-schema}, see Corollary~\ref{corbound-Khn}.

\smallskip

In order to prove Theorem \ref{thmmain-SDE}$(b)$ and $(c)$, we will provide bounds on the quantities that appear in the right-hand side of \eqref{eqerror-first-bound-SDE}. 
Let us start with the bounds on $K-K^k$ in both sub-critical and limit cases.

\paragraph{Bound on $K-K^k$.}

Let $j,k\in \N$. 

First, in the case $\gamma \in( 1-\frac{1}{2H},0)$, we apply Proposition \ref{propregfBm}$(a)$ with $f=b^{j}-b^k$, $\tau=1/2+H$, $\beta=\gamma-1$ and $\psi=X-B$. Using $[X-B]_{\mathcal{C}^{1/2+H}_{[0,1]} L^m} < \infty$, it comes that for any $(s,t) \in \Delta_{S,T}$,
\begin{align*}
\|K_t^{j}-K_s^{j} - K_{t}^k + K_s^{k} \|_{L^m} & \leq  C \| b^{j} - b^k \|_{\mathcal{B}_\infty^{\gamma-1}}   (t-s)^{1+H(\gamma-1)} .
\end{align*}
Hence $(K^{j}_t-K_s^{j})_{j \in \mathbb{N}}$ is a Cauchy sequence in $L^m(\Omega)$ and therefore it converges. We also know by definition of $X$ that $K^{j}_t-K_s^{j}$ converges in probability to $K_t-K_s$. Thus $K^{j}_t-K_s^{j}$ converges in $L^m$ to $K_t-K_s$. Now by the convergence of $b^{j}$ to $b$ in $\mathcal{B}_\infty^{\gamma-1}$, we get
\begin{align*}
\|K_t-K_s - K_{t}^k + K_s^{k} \|_{L^m} & \leq  C \| b- b^k \|_{\mathcal{B}_\infty^{\gamma-1}} \, (t-s)^{1+H(\gamma-1)} .
\end{align*}
Dividing by $(t-s)^{\frac{1}{2}}$ and taking the supremum over $(s,t)$ in $\Delta_{S,T}$ (recall that $\frac{1}{2} + H(\gamma-1) \ge 0$), we get that
\begin{align}\label{eqproba-conv-SDE}
[K-K^{k}]_{\mathcal{C}^{\frac{1}{2}}_{[S,T]} L^{m}} & \leq  C \|b-b^k\|_{\mathcal{B}_\infty^{\gamma-1}}.
\end{align}

\smallskip

Now in the limit case, \emph{i.e.} with $b \in \mathcal{B}_p^{\tilde\gamma}$, $\tilde\gamma-d/p=1-1/(2H)$ and $p \in [m,+\infty)$, we will apply Proposition \ref{propregfBm}$(b)$ with $f=b^{j}-b^k$, $\beta=\tilde\gamma-1$ and $\psi=X-B$. Since $(b^j)$ converges to $b$ in $\mathcal{B}^{\tilde\gamma-}_{p}$, there is $ \|b^j-b^k\|_{\mathcal{B}^{\tilde\gamma}_{p} }\leq 2 \|b\|_{\mathcal{B}^{\tilde\gamma}_{p}} \vee 1$, and therefore
\begin{align*}
\left| \log \frac{\| b^j - b^k \|_{\mathcal{B}_p^{\tilde\gamma-1}}}{\| b^j - b^k \|_{\mathcal{B}_p^{\tilde\gamma-}}} \right| & \leq  \log( 2 \| b \|_{\mathcal{B}_p^{\tilde\gamma}} \vee 1)  + |  \log(\| b^j - b^k \|_{\mathcal{B}_p^{\tilde\gamma-1}}) | \\
& \leq C (1 +  |  \log(\| b^j - b^k \|_{\mathcal{B}_p^{\tilde\gamma-1}}) |  ) .
\end{align*} 
Besides, $[X-B]_{\mathcal{C}^{1/2+H}_{[0,1]} L^m} <\infty$, hence for any $(s,t) \in \Delta_{S,T}$, Proposition \ref{propregfBm}$(b)$ reads
\begin{align*}
\|K_t^{j}-K_s^{j} - K_{t}^k + K_s^{k} \|_{L^m} & \leq  C \| b^{j} - b^k \|_{\mathcal{B}_p^{\tilde\gamma-1}}  (1+|\log \| b^j - b^k \|_{\mathcal{B}_p^{\tilde\gamma-1}} | )  (t-s)^{\frac{1}{2}} .
\end{align*}
As in the sub-critical case, we deduce that 
\begin{align}\label{eqproba-conv-SDE-critic}
[K-K^{k}]_{\mathcal{C}^{\frac{1}{2}}_{[S,T]} L^{m}} & \leq  C \|b-b^k\|_{\mathcal{B}_p^{\tilde\gamma-1}} (1+ |\log \|b-b^k\|_{\mathcal{B}_p^{\tilde\gamma-1}} |) .
\end{align}

\paragraph{Bound on $E^{1,h,k}$.} 

Recall that $\mathcal{E}^{h,k}= K-K^{h,k}$ and that $X^{h,k}$ satisfies $\sup_{(h,k) \in \mathcal{D}} [X^{h,k}-B]_{\mathcal{C}^{1/2+H}_{[0,1]} L^{m, \infty}}<+\infty$ by Theorem \ref{thmmain-SDE}$(a)$ since \eqref{eqcond-gamma-p-H} and \eqref{eqassump-bn-bounded} hold.

In the sub-critical case, we have from Corollary \ref{corbound-E1-SDE} that there exists $C>0$ such that for any $(s,t)\in \Delta_{S,T}$, any $k \in \mathbb{N}$ and $h \in (0,1)$,
\begin{align*}
\|E^{1,h,k}_{s,t} \|_{L^{m}} &  \leq C \Big( [\mathcal{E}^{h,k}]_{\mathcal{C}^{\frac{1}{2}}_{[S,T]} L^{m}} + \|\mathcal{E}^{h,k}_{S} \|_{L^m} \Big) (t-s)^{1+ H (\gamma-1)}.
\end{align*}
Divide by $(t-s)^{1/2}$ and take the supremum over $(s,t)\in \Delta_{S,T}$ to get
\begin{align}\label{eqbound-E1-SDE}
[ E^{1,h,k} ]_{\mathcal{C}^{\frac{1}{2}}_{[S,T]} L^m} \leq C \Big( [\mathcal{E}^{h,k}]_{\mathcal{C}^{\frac{1}{2}}_{[S,T]} L^{m}} + \|\mathcal{E}^{h,k}_{S} \|_{L^m} \Big) (T-S)^{\frac{1}{2}+ H (\gamma-1)} .
\end{align}

In the limit case, since $\sup_{(h,k) \in \mathcal{D}} [X^{h,k}-B]_{\mathcal{C}^{1/2+H}_{[0,1]} L^{m, \infty}}<\infty$ and $[X-B]_{\mathcal{C}^{1/2+H}_{[0,1]} L^{m}}<\infty$, we can apply Proposition~\ref{propbound-E1-SDE-critic} to obtain the existence of $\ell_{0}>0$ and $\mathbf{M}>0$ such that if $T-S\leq \ell_{0}$, then for any $(s,t)\in \Delta_{S,T}$,
\begin{align*}
 \| E^{1,h,k}_{s,t}\|_{L^{m}} &\leq \M \, \bigg(1+ \Big|\log \frac{T^H}{ \| \mathcal{E}^{h,k} \|_{L^\infty_{[S,T]} L^{m}} +\epsilon(h,k)}\Big| \bigg) \, \Big( \|\mathcal{E}^{h,k} \|_{L^\infty_{[S,T]} L^{m}} + \epsilon(h,k)\Big) \, (t-s) \nonumber\\
 &\quad+ \M \, \Big(\|\mathcal{E}^{h,k} \|_{L^\infty_{[S,T]} L^{m}} + [\mathcal{E}^{h,k} ]_{\mathcal{C}^{\frac{1}{2}-\zeta}_{[S,T]} L^{m}} \Big)\,   (t-s)^{\frac{1}{2}} .
\end{align*}
It follows that
\begin{align*}
 \| E^{1,h,k}_{s,t}\|_{L^{m}} &\leq  \M \, \Big(1+ \big|\log \big(\|\mathcal{E}^{h,k} \|_{L^\infty_{[S,T]} L^{m}} +\epsilon(h,k)\big)\big| \Big) \, \Big( \|\mathcal{E}^{h,k} \|_{L^\infty_{[S,T]} L^{m}} + \epsilon(h,k)\Big) \, (t-s) \\
 &\quad+ C \, \Big( (1+|\log T| (t-s)^{\frac{1}{2}}) \big( \|\mathcal{E}^{h,k} \|_{L^\infty_{[S,T]} L^{m}}  + \epsilon(h,k)\big)+ [\mathcal{E}^{h,k} ]_{\mathcal{C}^{\frac{1}{2}-\zeta}_{[S,T]} L^{m}} \Big)\, (t-s)^{\frac{1}{2}}.
\end{align*}
Now use that $1\geq T\geq t-s$ to deduce that $|\log T| (t-s)^{\frac{1}{2}}$ is bounded on the set $\{(s,t,T)\colon T\in (0,1] \text{ and } s<t\leq T \}$.  
Since $\|\mathcal{E}^{h,k} \|_{L^\infty_{[S,T]} L^{m}} \leq \| \mathcal{E}_S^{h,k} \|_{L^m} + [\mathcal{E}^{h,k}]_{\mathcal{C}^{\frac{1}{2}-\zeta}_{[S,T]} L^{m}}$, we get using $\M (t-s) \leq C (t-s)^{\frac{1}{2}}$,
\begin{align*}
 \| E^{1,h,k}_{s,t}\|_{L^{m}} &\leq  \M \, \big|\log \big(\|\mathcal{E}^{h,k} \|_{L^\infty_{[S,T]} L^{m}} +\epsilon(h,k)\big) \big| \, \Big( \|\mathcal{E}^{h,k} \|_{L^\infty_{[S,T]} L^{m}} + \epsilon(h,k)\Big) \, (t-s) \\
 &\quad+ C \, \Big(\|\mathcal{E}^{h,k}_{S} \|_{L^{m}} + [\mathcal{E}^{h,k} ]_{\mathcal{C}^{\frac{1}{2}-\zeta}_{[S,T]} L^{m}} + \epsilon(h,k) \Big)\, (t-s)^{\frac{1}{2}}.
\end{align*}
Divide by $(t-s)^{1/2-\zeta}$ and take the supremum over $(s,t)\in \Delta_{S,T}$ to get
\begin{equation}\label{eqbound-E1-SDE-critic}
\begin{split}
 [ E^{1,h,k} ]_{\mathcal{C}^{\frac{1}{2}-\zeta}_{[S,T]} L^{m}} &\leq  \M \, \Big( \big|\log \big(\|\mathcal{E}^{h,k} \|_{L^\infty_{[S,T]} L^{m}} +\epsilon(h,k)\big)\big| \Big) \, \Big( \|\mathcal{E}^{h,k} \|_{L^\infty_{[S,T]} L^{m}} + \epsilon(h,k)\Big) \, (T-S)^{\frac{1}{2}+\zeta} \\
 &\quad+ C \, \Big(\|\mathcal{E}^{h,k}_{S} \|_{L^{m}} + [\mathcal{E}^{h,k} ]_{\mathcal{C}^{\frac{1}{2}-\zeta}_{[S,T]} L^{m}}  + \epsilon(h,k)\Big)\, (T-S)^{\zeta} .
\end{split}
\end{equation}

\paragraph{Bound on $E^{2,h,k}$.} By Corollary \ref{cornewbound-E2}, we have the following bound for $\varepsilon \in (0,\frac{1}{2})$, and $(s,t) \in \Delta_{S,T}$
\begin{align*}
\| E^{2,h,k}_{s,t} \|_{L^m} \leq C \left( \|b^k\|_\infty h^{\frac{1}{2}-\varepsilon} +  \|b^k\|_{\mathcal{C}^1} \|b^k\|_\infty  h^{1-\varepsilon} \right)   (t-s)^{\frac{1}{2}} . 
\end{align*}
Dividing by $(t-s)^{\frac{1}{2}}$ and taking the supremum over $(s,t)$ in $\Delta_{S,T}$, we get
\begin{align}\label{eqbound-E2-SDE}
[ E^{2,h,k}  ]_{\mathcal{C}^{\frac{1}{2}}_{[S,T]} L^{m}} & \leq  C \left( \|b^k\|_\infty h^{\frac{1}{2}-\varepsilon} +  \|b^k\|_{\mathcal{C}^1} \|b^k\|_\infty  h^{1-\varepsilon} \right)     .
\end{align}
This is where we avoid using Girsanov's theorem and rely instead on a bound that involves the $\mathcal{C}^1$ norm of $b^k$. This simply comes from estimates when $t-s\leq h$ of the form $|\int_{s}^t f(\psi_{r}+B_{r}) - f(\psi_{r}+B_{r_{h}}) \, dr| \lesssim \|f\|_{\mathcal{C}^1}\, (t-s)\, h^{H-}$, at a scale where the discretised noise cannot regularise anymore.
More rigorously, the previous bound is again obtained by a stochastic sewing argument.

\paragraph{Conclusion in the sub-critical case.} 
Using \eqref{eqbound-E1-SDE} in \eqref{eqerror-first-bound-SDE}, and recalling the definition of $\epsilon(h,k)$ in \eqref{eqdefepsilonhn}, we get 
\begin{align*}
[\mathcal{E}^{h,k}]_{\mathcal{C}^{\frac{1}{2}}_{[S,T]} L^{m}} &\leq \epsilon(h,k) + C  \Big( [\mathcal{E}^{h,k}]_{\mathcal{C}^{\frac{1}{2}}_{[S,T]} L^{m}} + \| \mathcal{E}^{h,k}_{S} \|_{L^m} \Big) (T-S)^{\frac{1}{2}+ H (\gamma-1)}.
\end{align*}
Hence for $T-S\leq (2C)^{-1/(1/2+H(\gamma-1))} =\colon \ell_{0}$, we get
\begin{align}\label{eqboundseminorm}
[\mathcal{E}^{h,k}]_{\mathcal{C}^{\frac{1}{2}}_{[S,T]} L^{m}} \leq 2\epsilon(h,k) + \| \mathcal{E}^{h,k}_{S} \|_{L^m} .
\end{align}
Then the inequality 
\begin{align*}
\|\mathcal{E}_{S}^{h,k}\|_{L^m} & \leq \|\mathcal{E}_{S-\ell_0}^{h,k}\|_{L^m}  + \|\mathcal{E}_{S}^{h,k}-\mathcal{E}_{S-\ell_0}^{h,k} \|_{L^m} \\
 & \leq \|\mathcal{E}_{S-\ell_{0}}^{h,k}\|_{L^m} + \ell_0^{\frac{1}{2}} [\mathcal{E}^{h,k}]_{\mathcal{C}^{\frac{1}{2}}_{[S-\ell_{0},S]} L^{m}}
\end{align*}
can be plugged in \eqref{eqboundseminorm} and iterated until $S-n \ell_{0}$ is smaller than $0$ for $n \in \N$ large enough. It follows that 
\begin{align*}
[\mathcal{E}^{h,k}]_{\mathcal{C}^{\frac{1}{2}}_{[0,1]} L^{m}} \leq C\epsilon(h,k) ,
\end{align*}
and in view of \eqref{eqdefepsilonhn}, \eqref{eqproba-conv-SDE} and \eqref{eqbound-E2-SDE}, we obtain the result \eqref{eqmain-result-SDE} of Theorem~\ref{thmmain-SDE}$(b)$.

\paragraph{Conclusion in the limit case.}
Using \eqref{eqbound-E1-SDE-critic} in \eqref{eqerror-first-bound-SDE}, we get that if $T-S\leq \ell_{0}$, 
\begin{align*}
[\mathcal{E}^{h,k}]_{\mathcal{C}^{\frac{1}{2}-\zeta}_{[S,T]} L^{m}} 
&\leq [K-K^{k}]_{\mathcal{C}^{\frac{1}{2}-\zeta}_{[S,T]} L^{m}} + [E^{2,h,k}]_{\mathcal{C}^{\frac{1}{2}-\zeta}_{[S,T]} L^{m}} \\
&\quad + \M \, \Big( \big|\log \big(\|\mathcal{E}^{h,k} \|_{L^\infty_{[S,T]} L^{m}} +\epsilon(h,k)\big)\big| \Big) \, \Big( \|\mathcal{E}^{h,k} \|_{L^\infty_{[S,T]} L^{m}} + \epsilon(h,k)\Big) \, (T-S)^{\frac{1}{2}+\zeta} \\
 &\quad+ C \, \Big(\|\mathcal{E}^{h,k}_{S} \|_{L^{m}} + [\mathcal{E}^{h,k} ]_{\mathcal{C}^{\frac{1}{2}-\zeta}_{[S,T]} L^{m}} + \epsilon(h,k) \Big) \, (T-S)^\zeta .
\end{align*}
We observe that $[K-K^{k}]_{\mathcal{C}^{\frac{1}{2}-\zeta}_{[S,T]} L^{m}} + [E^{2,h,k}]_{\mathcal{C}^{\frac{1}{2}-\zeta}_{[S,T]} L^{m}} \leq (T-S)^{\zeta} \, \epsilon(h,k)$. Let $\ell>0$ satisfying
\begin{align}\label{eqboundEll}
\ell < ( C )^{-\frac{1}{\zeta}} \wedge 1 \wedge \ell_{0} .
\end{align}
Passing the term $[\mathcal{E}^{h,k}]_{\mathcal{C}^{\frac{1}{2}-\zeta}_{[S,T]} L^{m}}$ from the r.h.s to the l.h.s, we get for any $S<T$ such that $T-S\leq \ell$
\begin{align}\label{eq1/2-zeta-bound}
[\mathcal{E}^{h,k}]_{\mathcal{C}^{\frac{1}{2}-\zeta}_{[S,T]} L^{m}} 
&\leq \frac{1+C}{1-C \ell^\zeta} (\epsilon(h,k) +  \, \|\mathcal{E}^{h,k}_{S} \|_{L^{m}})\,  (T-S)^{\zeta} \nonumber \\
&\quad +  \frac{\M}{1-C \ell^\zeta}  \,  \big|\log \big(\|\mathcal{E}^{h,k} \|_{L^\infty_{[S,T]} L^{m}} +\epsilon(h,k)\big)\big|  \, \Big( \|\mathcal{E}^{h,k} \|_{L^\infty_{[S,T]} L^{m}} + \epsilon(h,k)\Big) \, (T-S)^{\frac{1}{2}+\zeta} .
\end{align}
Hence denoting $C_1 = \frac{1+C}{1-C \ell^\zeta}$ and $\C_2 =\frac{\M}{1-C \ell^\zeta} $, we have for $T-S \leq \ell$,
 \begin{equation}\label{eqboundIncE}
\begin{split}
\| \mathcal{E}^{h,k}_{T} - \mathcal{E}^{h,k}_{S} \|_{L^m} &\leq C_1 \, \left(\epsilon(h,k) +\|\mathcal{E}^{h,k}\|_{L^\infty_{[S,T]}L^m} \right) \, (T-S)^{\frac{1}{2}}  \\
&\quad + \C_2  \, \Big( \|\mathcal{E}^{h,k}\|_{L^\infty_{[S,T]}L^m}+\epsilon(h,k) \Big) \, \big| \log \big( \|\mathcal{E}^{h,k}\|_{L^\infty_{[S,T]}L^m} + \epsilon(h,k) \big)\big| \, (T-S). 
\end{split}
\end{equation}
At this point, it seems that we could conclude with a Gr\"onwall-type argument. However, consider first the simpler Gr\"onwall-type inequality $\| f_t - f_s \| \leq C \| f \|_{L^\infty
_{[s,t]}}^{\alpha} (t-s)$,
where $f_0=0$. One can check that $f=0$ when $\alpha \ge 1$, but that there exists non-zero functions that satisfy the inequality when $\alpha <1$. The following lemma describes the critical case $\alpha=1$ with a logarithmic factor and a perturbation $\eta$ of the function $f$, which corresponds to the situation in \eqref{eqboundIncE}. It quantifies, in terms of the perturbation $\eta$, an argument that appears in the proofs of \cite[Prop. 3.6]{athreya2020well} and \cite[Prop. 6.1]{anzeletti2021regularisation}. 
\begin{lemma}\label{lemrate-critical}
Let $(E, \|\cdot\|)$ be a normed vector space. For $\ell >0, \ C_{1}, C_{2} \ge 0$ and $\eta\in(0,1)$ we consider the set $\mathcal{R}(\eta,\ell, C_{1}, C_{2})$ of functions  defined from $[0,1]$ to $E$ characterised as follows:  $f\in \mathcal{R}(\eta,\ell, C_{1}, C_{2})$  if $f$ is bounded, $f_{0}=0$ and for any $s\leq t\in [0,1]$ such that $t-s \leq \ell$,
\begin{equation}\label{eqboundIncf}
\begin{split}
\|f_{t} - f_{s}\| &\leq C_1 \, (\|f\|_{L^\infty_{[s,t]}E} + \eta) \, (t-s)^{\frac{1}{2}}  \\
&\quad + C_2 \, ( \|f\|_{L^\infty_{[s,t]}E}+\eta ) \, \big| \log \big( \|f\|_{L^\infty_{[s,t]}E} + \eta \big)\big| \, (t-s). 
\end{split}
\end{equation}
Then for any $\delta\in(0, e^{-C_{2}})$, there exists $\bar{\eta} \equiv \bar{\eta}(C_{1},C_{2},\ell,\delta)$ such that for any $\eta<\bar{\eta}$ and any $f\in \mathcal{R}(\eta,\ell, C_{1}, C_{2})$,
\begin{equation*}
\| f \|_{L^\infty_{[0,1]} E} \leq \eta^{e^{-C_{2}}-\delta} .
\end{equation*}
\end{lemma}
The proof is postponed to the Appendix~\ref{app:sec:gronwall1}.
\begin{remark}
The real-valued function $g_{t}= \eta^{e^{-C t}} - \eta$ solves the differential equation $g'_{t} = -C\, (g_{t} + \eta)\, \log(g_{t}+\eta), \ \forall t \in [0,1]$.
Thus the result of Lemma~\ref{lemrate-critical} seems to give an optimal bound up to the correction $\delta$.
\end{remark}

Let $\delta \in (0, \frac{e^{-\M}}{4})$. Recalling that $\C_2 = \frac{\M}{1-C \ell^\zeta}$, let us also choose $\ell$ which still satisfies \eqref{eqboundEll} and which is small enough in order to have $e^{-\C_2} \geq e^{-\M}-\delta \ge \delta$. We now apply Lemma~\ref{lemrate-critical} with $E= L^m$, $\ell=\ell$, $C_1=C_1$, $C_2 = \C_2$. The term with the factor $C_1$ is only a small perturbation, that is why we do not keep track of the second constant in the bound \eqref{eqbound-E1-SDE-critic} of $E^{1,h,k}$. By \eqref{eqboundIncE}, we have that $\mathcal{E}^{h,k}$ belongs to $\mathcal{R}(\epsilon(h,k), \ell, C_1, \C_2)$ for $(h,k)\in \mathcal{D}$. Therefore, there exists $\bar{\epsilon} \equiv \bar{\epsilon} (C_1, \C_2, \ell, \delta)$ such that if $\epsilon(h,k) < \bar{\epsilon}$, we have

\begin{align*}
\| \mathcal{E}^{h,k} \|_{L^\infty_{[0,1]} L^m} \leq  \epsilon(h,k)^{e^{-\C_2}-\delta} \leq \epsilon(h,k)^{e^{-\M}-2\delta}  .
\end{align*}
Let $\epsilon \equiv \epsilon(C_1,\M,\ell, \delta, \zeta) < \bar{\epsilon} \wedge 1 $ such that $\epsilon^{e^{-\M}-2\delta} < \frac{e^{-1}}{2}$. Then, if $\epsilon(h,k) < \epsilon$, we have 
$$\| \mathcal{E}^{h,k} \|_{L^\infty_{[0,1]}L^m}  + \epsilon(h,k) \leq \epsilon(h,k)^{e^{-\M}-2\delta} + \epsilon(h,k)  < 2 \epsilon(h,k)^{e^{-\M}-2 \delta} < e^{-1}.$$
Since $x \mapsto x | \log(x) |$ is increasing over $(0, e^{-1})$, in view of \eqref{eq1/2-zeta-bound}, we have that over any interval $I$ of size $\ell$,
\begin{align*}
[ \mathcal{E}^{h,k} ]_{ \mathcal{C}^{\frac{1}{2}-\zeta}_{I} L^m} \leq C\, \epsilon(h,k)^{(e^{-\M}-2\delta)}\, (1+|\log(\epsilon(h,k))| ) .
\end{align*} 
Since $\ell$ is fixed independently of $\epsilon(h,k)$, summing at most $\frac{1}{\ell}$ of these bounds, we get that if $\epsilon(h,k) < \epsilon$
\begin{align*}
[ \mathcal{E}^{h,k} ]_{ \mathcal{C}^{\frac{1}{2}-\zeta}_{[0,1]} L^m} \leq C\, \epsilon(h,k)^{(e^{-\M}-2\delta)}\,(1+|\log(\epsilon(h,k))| ) \leq C\, \epsilon(h,k)^{(e^{-\M}-4\delta)} .
\end{align*}
From Corollary \ref{corbound-Khn} and the property $[K]_{\mathcal{C}^{1/2+H}_{[0,1]} L^{m,\infty}} < \infty$, we have that under \eqref{eqassump-bn-bounded}, $$\sup_{h,k\in \mathcal{D}} [ \mathcal{E}^{h,k} ]_{ \mathcal{C}^{\frac{1}{2}-\zeta}_{[0,1]} L^m} \leq  [K]_{\mathcal{C}^{1/2+H}_{[0,1]} L^{m,\infty}} + \sup_{h,k\in \mathcal{D}}  [K^{h,k}]_{\mathcal{C}^{1/2+H}_{[0,1]} L^{m,\infty}} <\infty.$$ It follows that there exists a constant $C$ such that for all $(h,k)\in \mathcal{D}$,
\begin{align*}
[ \mathcal{E}^{h,k} ]_{ \mathcal{C}^{\frac{1}{2}-\zeta}_{[0,1]} L^m} \leq C\, \epsilon(h,k)^{(e^{-\M}-4\delta)} .
\end{align*}

In view of \eqref{eqdefepsilonhn}, we obtain \eqref{eqmain-result-SDE-critic} from~Theorem \ref{thmmain-SDE}$(c)$.

\subsection{Proof of Corollary \ref{corbn-choice}}\label{subsecCor2.5}

We will do the computations with $k_h = \lfloor h^{-\alpha} \rfloor $ for some $\alpha>0$ and prove that the upper bound on $ [\mathcal{E}^{h,k_h}]_{\mathcal{C}^{1/2}_{[0,1]} L^m}$ given by Theorem \ref{thmmain-SDE} is minimized for $\alpha = 1/(1-\gamma)$ in the sub-critical regime and $\alpha=1/(1-\gamma)$.

First, the inequalities \eqref{eqbn-inf} and \eqref{eqbn-C1} imply that $b^{k_h}$ satisfy \eqref{eqassump-bn-bounded} with $\mathcal{D}=\{ (h,k_h) , h \in (0,1/2) \}$, any $\eta \in (0,H)$ and $\alpha \leq (2(H-\eta)+1)/(1-\gamma)$. Therefore, we deduce \eqref{equnifscheme} from Theorem \ref{thmmain-SDE}$(a)$. 

\paragraph{The sub-critical case.}
Since $b^{k_h}$ satisfies \eqref{eqbn-b}, \eqref{eqbn-inf}, \eqref{eqbn-C1}, the result of Theorem \ref{thmmain-SDE}$(b)$ reads
\begin{align*}
 [\mathcal{E}^{h,k_h}]_{\mathcal{C}^{\frac{1}{2}}_{[0,1]} L^m} & \leq C \Big( \lfloor h^{-\alpha} \rfloor ^{-\frac{1}{2}}  +  \lfloor h^{-\alpha} \rfloor ^{-\frac{1}{2} \gamma}  h^{\frac{1}{2}-\varepsilon} + \lfloor h^{-\alpha} \rfloor ^{\frac{1}{2}- \gamma} h^{1-\varepsilon}  \Big) .
\end{align*}
Since $-\frac{1}{2}\gamma>0$ and $\frac{1}{2}-\gamma>0$, we have
\begin{align*}
 \lfloor h^{-\alpha} \rfloor ^{\frac{1}{2}-\gamma} \leq h^{-\frac{\alpha}{2}} h^{\alpha\gamma} \text{ and } \lfloor h^{-\alpha} \rfloor ^{-\frac{1}{2}\gamma}  \leq  h^{\frac{\alpha}{2}\gamma} .
\end{align*}
Moreover, since $ h \in (0, \frac{1}{2})$ and $ \lfloor h^{-\alpha} \rfloor  > h^{-\alpha} -1$, we have
\begin{align*}
\lfloor h^{-\alpha} \rfloor^{-\frac{1}{2}} & \leq (1-h^{\alpha})^{-\frac{1}{2}} h^{\frac{\alpha}{2}} 
 \leq \left( 1-\frac{1}{2^\alpha} \right)^{-\frac{1}{2}} h^{\frac{\alpha}{2}} \leq C h^{\frac{\alpha}{2}} .
\end{align*}
It follows that 
\begin{align*}
[\mathcal{E}^{h,k}]_{\mathcal{C}^{\frac{1}{2}}_{[0,1]} L^m} & \leq C \Big(h^{\frac{\alpha}{2}} +  h^{\frac{\alpha}{2}\gamma} h^{\frac{1}{2}-\varepsilon} +  h^{-\frac{\alpha}{2}} h^{\alpha\gamma} h^{1-\varepsilon}\Big).
\end{align*}
Now we optimize over $\alpha$. Introduce the following functions:
\begin{align*}
f_1(\alpha) = \frac{\alpha}{2} ,\quad f_2(\alpha) = \frac{\alpha}{2}  \gamma+\frac{1}{2} ~~\mbox{and}~~ f_3(\alpha) & = \left( \gamma-\frac{1}{2} \right) \, \alpha+1 ,\quad \alpha>0.
\end{align*}
Observe that $f_1$ is increasing and $f_2,f_3$ are decreasing. Moreover, we have
\begin{align}\label{eqscaling}
f_1(\alpha)=f_2(\alpha)=f_3(\alpha) \Leftrightarrow \alpha^\star = \frac{1}{1-\gamma} . 
\end{align}
It follows that the error is minimized at $\alpha=\alpha^\star$. Let $k_h =\lfloor h^{-\alpha^\star} \rfloor$. This yields rate of convergence of order $ 1/(2(1-\gamma))-\varepsilon$, which proves \eqref{eqrate1}.

\paragraph{The limit case.} 
Since $\mathcal{B}_p^{\tilde\gamma} \hookrightarrow \mathcal{B}_\infty^{\gamma}$ (recall that $\tilde{\gamma}=\gamma+d/p$), we see that $b^k$ satisfies \eqref{eqbn-inf} and \eqref{eqbn-C1}.
Using \eqref{eqbn-b-critic}, the result of Theorem \ref{thmmain-SDE}$(c)$ reads
\begin{align*}
[\mathcal{E}^{h,k_h}]_{\mathcal{C}^{\frac{1}{2}-\zeta}_{[0,1]} L^m} \leq C \Big( \lfloor h^{-\alpha} \rfloor ^{-\frac{1}{2}} (1+|\log( \lfloor h^{-\alpha} \rfloor ^{-\frac{1}{2}})|)  +  \lfloor h^{-\alpha} \rfloor ^{-\frac{1}{2}\gamma}  h^{\frac{1}{2}-\varepsilon} + \lfloor h^{-\alpha} \rfloor ^{\frac{1}{2}-\gamma)} h^{1-\varepsilon}  \Big)^{e^{-M}-\delta} .
\end{align*}
Optimising over $\alpha$ again, we find $\alpha^\star = 2H$. This yields $[\mathcal{E}^{h,k_h}]_{\mathcal{C}^{1/2-\zeta}_{[0,1]} L^m}  \leq C \Big( h^{H-\varepsilon} |\log(h)| \Big)^{e^{-\M}-\delta}$, which proves \eqref{eqrate1-critic}

\subsection{Proof of Corollary \ref{corSP}}\label{subsecStrongEx}

Assuming \eqref{eqcond-gamma-p-H},
 we let $(X,B)$ and $(\Omega,\mathcal{F},\mathbb{F},\PP)$ be a weak solution to \eqref{eqSDE} given by Theorem~\ref{thWP}, which satisfies 
$[X-B]_{\mathcal{C}^{1/2+H}_{[0,1]} L^m}<+\infty$ (see Remark~\ref{rk:regweaksol}).
 On this probability space and with the same fBm $B$, we define the tamed Euler scheme $(X^{h,k})_{h>0,k\in \N}$. As in Corollary \ref{corbn-choice}, we let $(b^k)_{k\in \N}$ satisfy \eqref{eqbn-inf} and \eqref{eqbn-C1}.
 
\paragraph{The sub-critical case.} Let $(b^k)_{k\in \N}$ satisfy also~\eqref{eqbn-b}. Set $k_h = \lfloor h^{-\frac{1}{1-\gamma}} \rfloor$ and consider the scheme $(X^{h,k_h})_{h \in (0,1)}$.

First, observe that $X^{h,k_h}$ is $\mathbb{F}^B$-adapted. In view of \eqref{eqboundsup} and Corollary \ref{corbn-choice}$(a)$, $X^{h,k_h}_{t}$ converges to $X_{t}$ in $L^m$, for each $t\in [0,1]$. Hence $X_{t}$ is $\mathcal{F}_{t}^B$-measurable and $X$ is therefore a strong solution.

As for the uniqueness, let $Y$ be a strong solution to \eqref{eqSDE} with the same fBm $B$, such that $[Y-B]_{\mathcal{C}^{1/2+H}_{[0,1]} L^{2}}<\infty$. Then by Corollary~\ref{corbn-choice}$(a)$, the Euler scheme $X^{h,k_h}$ approximates both $X$ and $Y$. So $X$ and $Y$ are modifications of one another. Since they are continuous processes, they are indistinguishable. This proves uniqueness in the class of solutions $Y$ such that $[Y-B]_{\mathcal{C}^{1/2+H}_{[0,1]} L^{2}}<\infty$. Observe that Theorem~\ref{thWP} gives a solution in this class such that $[X-B]_{\mathcal{C}^{1/2+H}_{[0,1]}L^{m,\infty}}< +\infty$, for any $m\geq2$.

We now extend the previous pathwise uniqueness to strong solutions that satisfy a weaker regularity. Denote by $X^1$ the solution given by the first part of the proof, which satisfies $[X^1-B]_{\mathcal{C}^{1/2+H}_{[0,1]}L^{2,\infty}}< +\infty$. Assume that there also exists a strong $X^2$ with less regularity, namely such that $[X^2-B]_{\mathcal{C}^{1/2}_{[0,1]}L^{2}}<+\infty$. Let $0 \leq S < T \leq 1$. We will show that $X^1=X^2$. Apply Proposition~\ref{propbound-E1-SDE} with $\psi=X^1-B$, $\phi=X^2-B$, $m=2$, $\tau=1/2$ and $f=b^k$ to get that for $(s,t)\in \Delta_{S,T}$, there is
\begin{align*}
\Big\|\int_{s}^t b^k(X^1_{r})\, dr - \int_{s}^t b^k(X^2_{r})\, dr \Big\|_{L^2} \leq C \|b^k\|_{\mathcal{B}^\gamma_{\infty}} \left( [X^1-X^2]_{\mathcal{C}^{1/2}_{[S,T]}L^2} + \| X^1_{S}-X^2_{S}\|_{L^2} \right) (t-s)^{1+H(\gamma-1)} .
\end{align*}
Since $X^1$ and $X^2$ are solutions in the sense of Definition~\ref{defsol-SDE}, the integrals $\int_{s}^t b^k(X^1_{r})\, dr$ and $\int_{s}^t b^k(X^2_{r})\, dr$ converge in probability respectively to $X^1_{t}-X^1_{s} - (B_t-B_s)$ and $X^2_{t}-X^2_{s} - (B_t-B_s)$. Up to taking a subsequence, the convergence holds in $L^2$, and we thus get
\begin{align*}
\| X^1_t-X^2_t - (X^1_s-X^2_s) \|_{L^2} \leq C \left( [X^1-X^2]_{\mathcal{C}^{1/2}_{[S,T]}L^2} + \| X^1_S-X^2_S \|_{L^2} \right) (t-s)^{1+H(\gamma-1)} .
\end{align*}
Divide now by $(t-s)^{1/2}$ and take the supremum over $(s,t)\in \Delta_{S,T}$ to get
\begin{align*}
[X^1-X^2]_{\mathcal{C}^{1/2}_{[S,T]}L^2}\leq C \left( [X^1-X^2]_{\mathcal{C}^{1/2}_{[S,T]}L^2} + \| X^1_{S}-X^2_{S} \|_{L^2} \right) (T-S)^{1/2+H(\gamma-1)} .
\end{align*}
Taking $S=0$ and $T \leq \ell$ with $\ell := (2C)^{-\frac{1}{1/2+H(\gamma-1)}}$, we deduce that $[X^1-X^2]_{\mathcal{C}^{1/2}_{[0,\ell]}L^2}=0$. Iterating the argument over intervals of size $\ell$, it comes that $[X^1-X^2]_{\mathcal{C}^{1/2}_{[0,1]}L^2}=0$, which concludes the proof.

\paragraph{The limit case.} Assume now that $(b^k)_{k\in \N}$ satisfies~\eqref{eqbn-b-critic}, let $k_{h} = \lfloor h^{2H} \rfloor$ and consider the scheme $(X^{h,k_{h}})_{h \in (0,1)}$. Using Corollary~\ref{corbn-choice}$(b)$ and following the same arguments as in the sub-critical case, we show that $X$ is a strong solution. As for the uniqueness, let $Y$ be a strong solution to \eqref{eqSDE} with the same fBm $B$, such that $[Y-B]_{\mathcal{C}^{1/2+H}_{[0,1]} L^{2}}<\infty$. Then by Corollary~\ref{corbn-choice}$(b)$, the Euler scheme $X^{h,k_h}$ approximates both $X$ and $Y$. 
As in the sub-critical case, this proves uniqueness in the class of solutions $Y$ such that $[Y-B]_{\mathcal{C}^{1/2+H}_{[0,1]} L^{2}}<\infty$.

\subsection{Proof of Corollary \ref{corgama=d/p}}

Let $\varepsilon \in (0, 1/4)$. Then $b$ also belongs to $\mathcal{B}_\infty^{-\varepsilon}$ and \eqref{eqcond-gamma-p-H} is satisfied with $\gamma \equiv -\varepsilon$. Therefore, Corollary~\ref{corSP} states that there exists a strong solution $X$ to \eqref{eqSDE} which satisfies $X-B \in \mathcal{C}_{[0,T]}^{1/2+H} L^{m, \infty}$, which is pathwise unique in the class of strong solutions $Y$ that satisfy $[Y-B]_ {\mathcal{C}_{[0,T]}^{1/2} L^{2}} < +\infty$.

 To prove the second part of the corollary, apply Theorem~\ref{thmmain-SDE}$(b)$ with $\gamma= -\varepsilon$ and $p=\infty$, to get that for $\mathcal{D}$ satisfying \eqref{eqassump-bn-bounded}, we have $\sup_{(h,k)\in \mathcal{D} } [X^{h,k}-B]_{\mathcal{C}^{1/2+H}_{[0,1]} L^m} < \infty$. Moreover, noting that $\|b-b^k\|_{\mathcal{B}_\infty^{-\varepsilon-1}}\leq C\|b-b^k\|_{\mathcal{B}_\infty^{-1}}$, it comes that
\begin{align*}
 [X - X^{h,k}]_{\mathcal{C}^{\frac{1}{2}}_{[0,1]} L^{m}} &  \leq C \left( \|b-b^k\|_{\mathcal{B}_\infty^{-1}} + \|b^k\|_\infty h^{\frac{1}{2}-\varepsilon} + \|b^k\|_{\mathcal{C}^1}  \|b^k\|_\infty  h^{1-\varepsilon} \right) .
\end{align*}
Now take $k_h=\lfloor h^{-\alpha}\rfloor$. As in Section~\ref{subsecCor2.5}, using that $(b^k)_{k \in \mathbb{N}}$ satisfies \eqref{eqbn-b}, \eqref{eqbn-inf} and \eqref{eqbn-C1}, we have $\sup_{h \in (0,1/2)} [X^{h,k_h}-B]_{\mathcal{C}^{1/2+H}_{[0,1]} L^m} < \infty$ and
\begin{align*}
 [X - X^{h,k_h}]_{\mathcal{C}^{\frac{1}{2}}_{[0,1]} L^{m}} &  \leq C \left( h^{\frac{\alpha}{2}} + h^{\frac{1}{2}-\varepsilon-\frac{\alpha\varepsilon}{2}} +  h^{1-\varepsilon-\frac{\alpha}{2}-\alpha\varepsilon} \right) .
\end{align*}
Optimising over $\alpha$ as before, we find $\alpha^\star = 1/(1+\varepsilon)$, which yields a rate of convergence of order $\frac{1}{2(1+\varepsilon)}-\varepsilon$. Since $\frac{1}{2(1+\varepsilon)} \geq 1/2 - \varepsilon$, it finally comes that $[X - X^{h,k_h}]_{\mathcal{C}^{\frac{1}{2}}_{[0,1]} L^{m}} \leq C\, h^{\frac{1}{2}-\eta-\delta} = C\, h^{\frac{1}{2}-2\varepsilon}$.

\section{Regularisation effect of the fBm and the discrete-time fBm}\label{secstochastic-sewing}

In Sections \ref{secbesov} and \ref{secproofs-SDE} we extend technical lemmas on the regularisation effects of the fBm from \cite{anzeletti2021regularisation}, which are used in the proof of Theorem~\ref{thmmain-SDE} and for the regularity of $E^{1,h,k}$ in Section~\ref{subsecE1hn}. Then we prove new regularisation results of the discrete-time fBm in Sections \ref{subsecbound-E2} and \ref{subsecfurtherdtimefBm}, which are used to obtain the regularity of the tamed Euler scheme (Section~\ref{subsecreg-schema}) and of $E^{2,h,k}$ (Section~\ref{subsecreg-E2hn}).

\subsection{Besov estimates}\label{secbesov}

The first lemma below is a generalisation to $\R^d$ of Lemma A.2 in \cite{athreya2020well} and follows its proof exactly, so we omit it.

\begin{lemma}\label{lembesov-spaces}
Let $f$ be a tempered distribution on $\R^d$ and let $\beta \in \R$, $p\in [1,\infty]$. Then for any $a_1,a_2,a_3 \in \mathbb{R}^d$ and $\alpha, \alpha_1, \alpha_2 \in [0,1]$, one has
\begin{itemize}
\item[(i)] $\| f(a + \cdot ) \|_{\mathcal{B}_p^\beta} \leq \| f \|_{\mathcal{B}_p^\beta}$ .
\item[(ii)] $\| f(a_1 + \cdot) - f(a_2 + \cdot) \|_{\mathcal{B}_p^\beta} \leq C |a_1 - a_2 |^{\alpha} \| f \|_{\mathcal{B}_p^{\beta+\alpha}}$ .
\item[(iii)] $\| f(a_1 + \cdot) - f(a_2 + \cdot) -f(a_3 + \cdot) + f(a_3+a_2-a_1+\cdot) \|_{\mathcal{B}_p^\beta} \leq C |a_1 - a_2 |^{\alpha_1} |a_1 - a_3 |^{\alpha_2} \| f \|_{\mathcal{B}_p^{\beta+\alpha_1 + \alpha_2}} .$
\end{itemize} 
\end{lemma}

\smallskip

The following Besov estimates for the Gaussian semigroup are borrowed or adapted from \cite{bahouri2011fourier,athreya2020well}. 
\begin{lemma}\label{lemreg-S}
Let $\beta\in \R$, $p\in [1,\infty]$ and $f \in \mathcal{B}_p^\beta$. Then 
\begin{enumerate}[label=(\roman*)]

\item If $\beta<0$, $\| G_t f \|_{L^p(\R^d)} \leq C\, \|f \|_{\mathcal{B}_p^\beta}\, t^{\frac{\beta}{2}}$, for all $t > 0$.

\item If $\beta-\frac{d}{p}<0$, $\| G_t f \|_{\infty} \leq C\, \|f \|_{\mathcal{B}_p^\beta}\, t^{\frac{1}{2}(\beta - \frac{d}{p})}$, for all $t > 0$.
 
 \item $\|G_t f - f\|_{\mathcal{B}_p^{\beta-\varepsilon}} \leq C\, t^{\frac{\varepsilon}{2}}\, \| f \|_{\mathcal{B}_p^\beta}$ for all $\varepsilon\in (0,1]$ and $t>0$. 
 In particular, it follows that $\lim_{t \rightarrow
 0} \|G_t f -f\|_{\mathcal{B}_p^{\tilde{\beta}}}=0$ for every $\tilde{\beta}< \beta$.
 
 \item $\sup_{t>0} \|G_t f \|_{\mathcal{B}_p^\beta} \leq \| f \|_{\mathcal{B}_p^\beta}$.
 
 \item If $\beta-\frac{d}{p}<0$, $\|G_t f \|_{\mathcal{C}^1} \leq C\, \| f \|_{\mathcal{B}_p^\beta}  \, t^{\frac{1}{2}(\beta- \frac{d}{p}-1)}$ for all $t>0$.
\end{enumerate}
\end{lemma}

\begin{proof}
\begin{enumerate}[label={\it(\roman*)}]
\item  The proof of Lemma A.3$(i)$ in \cite{athreya2020well} extends right away to dimension $d\geq1$.

\item Using $(i)$ for $\beta-\frac{d}{p}$ instead of $\beta$ and the embedding $\mathcal{B}^\beta_{p} \hookrightarrow \mathcal{B}^{\beta-\frac{d}{p}}_{\infty}$, there is
\begin{equation*}
\| G_{t} f\|_{L^\infty(\R^d)} \leq C \, \|f\|_{\mathcal{B}^{\beta-d/p}_{\infty}} \, t^{\frac{1}{2}(\beta-\frac{d}{p})} \leq C \, \|f\|_{\mathcal{B}^{\beta}_{p}} \, t^{\frac{1}{2}(\beta-\frac{d}{p})} .
\end{equation*}

\item This is an adaptation of Lemma A.3$(ii)$ in \cite{athreya2020well} to dimension $d\geq 1$ that we detail briefly. From \cite[Lemma 4]{MourratWeber}, we have that for $g$ such that the support of $\mathcal{F}g$ is in a ball of radius $\lambda\geq 1$, there is for all $t\geq 0$,
\begin{equation*}
\|G_{t}g - g\|_{L^p(\R^d)} \leq C\, (t\lambda^2\wedge 1) \|g\|_{L^p(\R^d)}.
\end{equation*}
For any $j\geq -1$, the support of $\mathcal{F}(\Delta_{j}f)$ is included in a ball of radius $2^j$. Hence, 
\begin{align*}
2^{j(\beta-\varepsilon)}\|G_{t}(\Delta_{j}f) - \Delta_{j}f\|_{L^p(\R^d)} &\leq C\, 2^{j(\beta-\varepsilon)} (t 2^{2j}\wedge 1) \|\Delta_{j}f\|_{L^p(\R^d)}\\
&\leq C\, 2^{-j\varepsilon} (t 2^{2j}\wedge 1)^{\frac{\varepsilon}{2}}\, 2^{j\beta} \|\Delta_{j}f\|_{L^p(\R^d)}\\
&\leq C\,  t^{\frac{\varepsilon}{2}}\, 2^{j\beta} \|\Delta_{j}f\|_{L^p(\R^d)}.
\end{align*}
The result follows. 

\item--~{\it(v)} The proof is the same as in the $d=1$ case, see \cite[Lemma A.3$(iii)$ and Lemma A.3$(iv)$]{athreya2020well}.
\end{enumerate}
\end{proof}

The next lemma describes some time regularity estimates of random functions of the fractional Brownian motion in Besov norms.

\begin{lemma}\label{lemreg-B}
Let $(\Omega,\mathcal{F},\mathbb{F},\mathbb{P})$ be a filtered probability space and $B$ be an $\mathbb{F}$-fBm. 
Let $\beta<0$, $p \in [1,\infty]$ and $e\in \N^*$. Then there exists a constant $C>0$ such that for any $(s,t)\in \Delta_{0,1}$, any bounded measurable function  $f\colon \R^d\times \R^e\to \R^d$ and any $\mathcal{F}_{s}$-measurable $\R^e$-valued random variable $\Xi$ satisfying $\|f(\cdot,\Xi)\|_{\mathcal{C}^1}<\infty$ almost surely, it holds
\begin{enumerate}[label=(\roman*)]

\item $\EE^{s}[f(B_{t},\Xi)]=G_{\sigma_{{s},{t}}^2}f(\EE^{s}[B_{t}],\Xi)$, where $G$ is the Gaussian semigroup introduced in \eqref{eqsemi-group-gaussian} and $$\sigma_{{s},{t}}^2\colon=\var{(B^{(i)}_{t}-\EE^{s}[B^{(i)}_{t}])},$$ for any component $B^{(i)}$ of the fBm\text{;} 

\item $| \mathbb{E}^s f( B_t,\Xi) | \leq C \|f(\cdot,\Xi)\|_{\mathcal{B}_p^\beta} (t-s)^{H(\beta-\frac{d}{p})}$\text{;}

\item $\| f(B_t,\Xi) - \mathbb{E}^s f(B_t,\Xi) \|_{L^1} \leq C \big\| \| f(\cdot,\Xi) \|_{\mathcal{C}^1}\big\|_{L^2} (t-s)^H$ .
\item Furthermore, for any $u$ in the interval $(s,t)$ and $m \in[1, p]$ there exists a constant $C>0$ such that 
$\left\|\mathbb{E}^{u}\left[f\left(B_{t}, \Xi\right)\right]\right\|_{L^{m}} \leq C \left\| \| f(\cdot, \Xi)\|_{\mathcal{B}_{p}^{\beta}}\right\|_{L^{m}} \left(t-u\right)^{H \beta}\left(u-s\right)^{-\frac{d}{2 p}}\left(t-s\right)^{d\frac{1-2 H}{2 p}}$.
\end{enumerate}
\end{lemma}

\begin{proof}
The proofs of $(i)$, $(ii)$, $(iii)$ are similar to $(a)$, $(b)$, $(c)$ in \cite[Lemma 5.1]{anzeletti2021regularisation} but they rely now on Lemma~\ref{lembesov-spaces} and Lemma~\ref{lemreg-S}. We only reproduce the proof of $(iv)$ which is similar to $(d)$ in \cite[Lemma 5.1]{anzeletti2021regularisation},  to emphasize where the dimension $d$ appears.

For $u \in (s,t)$, we have from $(i)$ that
\begin{align*}
\mathbb{E}^{s} | \EE^{u} [ f(B_{t}, \Xi ) ] |^m 
& = \EE^s | G_{\sigma^2_{u,t}} f(\EE^u B_t , \Xi) |^m \\
& = \EE^{s} | G_{\sigma^2_{u,t}} f(\EE^u B_t - \EE^{s} B_t + \EE^s B_t , \Xi) |^m .
\end{align*}
Notice that $\EE^s B_t$ is independent of $\EE^u B_t - \EE^s B_t$, which is a Gaussian variable with mean zero and covariance $\sigma^2_{s,u,t} I_{d}$ where $\sigma^2_{s,u,t} \colon = \textrm{Var}(\EE^u B_t^{(i)} - \EE^s B_t^{(i)})$,
for any component $B^{(i)}$ of the fBm. It follows that
\begin{align*}
\mathbb{E}^{s} | \EE^{u} [ f(B_{t}, \Xi ) ] |^m &  =  \int_{\R^d} g_{ \sigma^{2}_{s,u,t} } (y) | G_{\sigma^2_{u,t}} f( \EE^s B_t + y, \Xi) |^m \diff y .
\end{align*}
Let $q=\frac{p}{m}$ and $q'=\frac{q}{q-1}$. Using H\"older's inequality, we get
\begin{align*}
\mathbb{E}^{s} | \EE^{u} [ f(B_{t}, \Xi ) ] |^m &\leq   \| g_{ \sigma^{2}_{s,u,t} } \|_{L^{q'}(\R^d)} \|G_{\sigma^2_{u,t}} f(\cdot, \Xi) \|_{L^p(\R^d)}^m  \\
&= \| G_{ \sigma^{2}_{s,u,t} } \delta_0 \|_{L^{q'}(\R^d)} \|G_{\sigma^2_{u,t}} f(\cdot, \Xi) \|_{L^p(\R^d)}^m .
\end{align*}
By Besov embedding, $\delta_0 \in \mathcal{B}_1^0 \hookrightarrow \mathcal{B}_{q'}^{-d+d/q'} = \mathcal{B}_{q'}^{-dm/p}$. Hence by Lemma \ref{lemreg-S}$(i)$,
\begin{align*}
\mathbb{E}^{s} | \EE^{u} [ f(B_{t}, \Xi ) ] |^m &  \leq   C \| \delta_0 \|_{\mathcal{B}_{q'}^{-dm/p}} \ \sigma_{s,u,t}^{-dm/p} \ \| f(\cdot, \Xi) \|^m_{\mathcal{B}_p^\beta} \ \sigma_{u,t}^{\beta m} .
\end{align*}
The fBm has the following local nondeterminism properties (see e.g. (C.3) and (C.5) in \cite{anzeletti2021regularisation}): there exists $C_{1}, C_{2}>0$ such that
\begin{align}\label{eqLND}
\sigma^2_{u,t}  = C_{1} (t-u)^{2H} ~~ \mbox{and}~~ \sigma^2_{s,u,t}\ge C_{2} (u-s) (t-s)^{-1+2H} . 
\end{align}
It follows that
\begin{align*}
\mathbb{E}^{s} | \EE^{u} [ f(B_{t}, \Xi ) ] |^m &  \leq  C \| \delta_0 \|_{\mathcal{B}_{q'}^{-dm/p}} \ \| f(\cdot, \Xi) \|^m_{\mathcal{B}_p^\beta} \  (u-s)^{-\frac{dm}{2p}} (t-s)^{(1-2H)\frac{dm}{2p}}  \ (t-u)^{ H \beta m} .
\end{align*}
We conclude by taking the expectation in the above inequality and raising both sides to the power $1/m$.
\end{proof}

\subsection{Regularisation properties of the $d$-dimensional fBm}\label{secproofs-SDE}

From now on, $X$ will always denote a weak solution to \eqref{eqSDE} with drift $b\in \mathcal{B}^\gamma_{\infty}$, with $\gamma \in \mathbb{R}$ satisfying \eqref{eqcond-gamma-p-H}. For such $X$, recall that the process $K$ is defined by \eqref{solution1}. Let $(b^k)_{k\in \mathbb{N}}$ be a sequence of smooth functions that converges to $b$ in $\mathcal{B}_\infty^{\gamma-}$. For $k \in \mathbb{N}$ and $h \in (0,1)$, recall that $X^{h,k}$ denotes the tamed Euler scheme \eqref{defEulerSDE} and that the process $K^{h,k}$ is defined by \eqref{eqdef-Khn}. 

\smallskip

First, Lemma \ref{lem1streg} extends \cite[Lemma C.3]{anzeletti2021regularisation} to dimension $d\geq 1$. Its proof relies on the stochastic sewing Lemma of \cite{le2020stochastic} (recalled in Lemma~\ref{lemSSL}) and  is close to the proof of \cite[Lemma D.2]{anzeletti2021regularisation}. It is postponed to the Appendix~\ref{app1streg}.

\begin{lemma} \label{lem1streg}
Let $\beta \in (-1/(2H),0)$ such that $\beta-d/p \in (-1/H,0)$. Let $m \in [2, \infty]$, $q \in [m, \infty]$ and assume that $p\in [q,+\infty]$. Then there exists a constant $C>0$ such that for any $0\leq S\leq T$, 
any $\mathcal{F}_{S}$-measurable random variable $\Xi$ in $\R^e$ and any bounded 
measurable function $f\colon \mathbb{R}^d\times\R^e \rightarrow \mathbb{R}^d$ fulfilling
\begin{enumerate}[label=(\roman*)]
    \item $\EE\left[ \|f(\cdot,\Xi)\|_{\mathcal{C}^1}^2\right]<\infty$\text{;}
    \item $\EE\left[ \|f(\cdot,\Xi)\|_{\mathcal{B}_p^{\beta}}^q\right]<\infty$,
\end{enumerate}
we have  for any $(s,t) \in \Delta_{S,T}$ that
\begin{equation}\label{eqregulINT}
    \Big\| \Big( \EE^S \Big| \int_s^t f(B_r,\Xi) \, dr \Big|^m \Big)^{\frac{1}{m}}  \Big\|_{L^q} \leq C \, \| \|f(\cdot,\Xi)\|_{\mathcal{B}_p^{\beta}}\|_{L^q}\, (t-s)^{1+H(\beta-\frac{d}{p})} .
\end{equation}
\end{lemma}

As a consequence of Lemma \ref{lemSSL} and Lemma \ref{lem1streg}, we get the following property of regularisation of the $d$-dimensional fBm, which will be used in the proof of Theorem \ref{thmmain-SDE}. It can be compared to \cite[Lemma 7.1]{anzeletti2021regularisation}, which is stated for one-dimensional processes in the sub-critical case only. The proof is postponed to Appendix \ref{appregfBm}.

\begin{proposition}\label{propregfBm}
Let $m\in[2,\infty)$, $q \in [m, +\infty]$.
 
 \begin{enumerate}[label=(\alph*)]
 \item\label{item3.5(a)} \underline{The sub-critical case}: let $\beta\in (-1/(2H),0)$ and let $\tau \in (0,1)$ such that $H(\beta-1)+\tau>0$. 
There exists a constant $C>0$
 such that for any $f\in \mathcal{C}^\infty_b(\mathbb{R}^d, \R^d)\cap \mathcal{B}_\infty^\beta$, any $\R^d$-valued stochastic process $(\psi_t)_{t\in[0,1]}$ adapted to $\mathbb{F}$, any $(S,T) \in \Delta_{0,1}$ and $(s,t) \in \Delta_{S,T}$ we have
\begin{equation} \label{eq3.5a}
	\begin{split}
		\Big\| \Big( \EE^S \Big| \int_s^t f(B_r+\psi_r) \, dr \Big|^m \Big)^{\frac{1}{m}}  \Big\|_{L^q}  \leq  &\, C\,  \|f\|_{\mathcal{B}_\infty^\beta}(t-s)^{1+H\beta} \\
		&+C \|f\|_{\mathcal{B}_\infty^\beta} [\psi]_{\mathcal{C}^\tau_{[S,T]}L^{m,q}} \, (t-s)^{1+H(\beta -1)+\tau}.	
	\end{split}
\end{equation}

\item\label{item3.5(b)} \underline{The limit case}: let $p \in [m,+\infty)$, $\beta-d/p=-1/(2H)$. There exists a constant $C>0$
 such that for any $f\in \mathcal{C}^\infty_b(\mathbb{R}^d, \R^d) \cap \mathcal{B}_p^{\beta+1}$, any $\R^d$-valued stochastic process $(\psi_t)_{t\in[0,1]}$ adapted to $\mathbb{F}$, any $(S,T) \in \Delta_{0,1}$ and any $(s,t) \in \Delta_{S,T}$, we have
\begin{equation} \label{eq3.5b}
	\begin{split}
		\Big\| \int_s^t f(B_r+\psi_r) \, dr \Big\|_{L^m} \leq  &\, C\,  \|f\|_{\mathcal{B}_p^{\beta}} \, \left(1+ \left| \log\frac{\|f\|_{\mathcal{B}_p^{\beta}}}{\|f\|_{\mathcal{B}_p^{\beta+1}}} \right| \right) \, \left( 1+[\psi]_{\mathcal{C}^{\frac{1}{2}+H}_{[S,T]}L^m} \right)  (t-s)^{\frac{1}{2}}.	
	\end{split}
\end{equation}
\end{enumerate}

\end{proposition}

Before giving the next key estimates of this section, we need a corollary of Lemma~\ref{lem1streg}.
\begin{corollary} \label{cor4inc}
Let $\beta \in (-1/(2H),0)$. Let $m \in [2, \infty]$ and $p \in [m,+\infty)$ such that $\beta-d/p \in (-1/H,0)$.
 Let $\lambda, \lambda_1, \lambda_2 \in (0,1]$ and assume that $\beta>-1/(2H)+\lambda$ and $\beta>-1/(2H)+\lambda_1+\lambda_2$. There exists a constant $C>0$
 such that for any $f \in \mathcal{C}_b^\infty(\R^d,\R^d) \cap \mathcal{B}_p^\beta$, any $0\leq s \leq u \leq t \leq 1$, any $\mathcal{F}_s$-measurable random variables $\kappa_1,\kappa_2 \in L^m$ and any $\mathcal{F}_u$-measurable random variables $\kappa_3, \kappa_4 \in L^m$, there is
\begin{align*}
    \Big\|\int_u^t &\left(f(B_r+\kappa_1)-f(B_r+\kappa_2)-f(B_r+\kappa_3)+f(B_r+\kappa_4) \right)dr \Big\|_{L^m} \nonumber\\
    &\leq C \|f\|_{\mathcal{B}_p^\beta}\, \|\EE^s[|\kappa_1-\kappa_3|^m]^{1/m}\|_{L^\infty}^{\lambda_2}\, \|\kappa_1-\kappa_2\|_{L^m}^{\lambda_1}\, (t-u)^{1+H(\beta-\lambda_1-\lambda_2-\frac{d}{p})} \\
    &\quad + C\|f\|_{\mathcal{B}_p^\beta}\, \|\kappa_1-\kappa_2 - \kappa_3 +\kappa_4\|_{L^m}^\lambda\,  (t-u)^{1+H(\beta-\lambda-\frac{d}{p})}.
\end{align*}
\end{corollary}

\begin{proof}
The proof is identical to the one-dimensional version of this result, see Corollary C.5 of \cite{anzeletti2021regularisation}, so we do not repeat it. It relies on Lemma~\ref{lem1streg} and Lemma~\ref{lembesov-spaces}$(iii)$.
\end{proof}

The following proposition shows a regularisation property of the $d$-dimensional fBm on increments of the form $\int_0^t f(\psi_r+B_r)-f(\phi_r+B_r) \, dr$.
\begin{proposition}\label{propbound-E1-SDE}
Let $(\psi_t)_{t\in[0,1]}, \, (\phi_t)_{t\in[0,1]}$ be two $\mathbb{R}^d$-valued stochastic processes adapted to $\mathbb{F}$. Let $f \in \mathcal{C}^\infty_{b}(\mathbb{R}^d,\mathbb{R}^d) \cap \mathcal{B}_\infty^\gamma$ and $m \in [2,\infty)$.
Assume that $\gamma \in( 1-\frac{1}{2H},0)$ and let $\tau \in (0,1)$ such that
\begin{align}\label{eqtau-cond}
\left( \tau\wedge \frac{1}{2}  \right)+ H \left( \gamma-1 \right) > 0 .
\end{align}
There exists a constant $C := C(m,\gamma,d) >0$ such that for any $0 \leq S < T \leq 1$ and $(s, t)\in \Delta_{S,T}$, %
\begin{equation}\label{eqssl-o-on-2-SDE}
\begin{split}
& \Big\| \int_s^t  f(\psi_r+ B_r) - f(\phi_r+  B_r)  \diff r \Big\|_{L^{m}}   \\ 
& \quad \leq C\,  \| f \|_{\mathcal{B}^{\gamma}_{\infty}} \Big( 1 + [\psi]_{\mathcal{C}^{\frac{1}{2}+H}_{[S,T]} L^{m,\infty}} \Big)  \left( [\psi-\phi]_{\mathcal{C}^{\tau}_{[S,T]} L^{m}} + \|\psi_S-\phi_S\|_{L^m} \right)(t-s)^{1+H(\gamma-1)} .
\end{split}
\end{equation}
\end{proposition}

\begin{remark}
Observe that the roles of $\psi$ and $\phi$ are not symmetric, since one needs ${[\psi]_{\mathcal{C}^{1/2+H}_{[S,T]} L^{m,\infty}}<+\infty}$ for the bound~\eqref{eqssl-o-on-2-SDE} to be useful, while one only needs $\phi$ to be in $\mathcal{C}^{\tau}_{[0,1]} L^{m}$. In Section~\ref{secoverview-SDE}, $\psi$ is replaced by the Euler scheme, which is shown to have finite $\mathcal{C}^{1/2+H}_{[0,1]} L^{m,\infty}$-seminorm uniformly in $h$ and $k$, and $\phi$ is replaced by any weak solution $X$ which satisfies~\eqref{eq:weakerreg}.
\end{remark}

\begin{proof}
Let $0\leq S<T\leq1$.
For $(s,t)\in \Delta_{S,T}$, let
\begin{align} \label{eqProp51A}
A_{s,t} = \int_{s}^{t}  f(\psi_{s} + B_r) - f(\phi_{s}+ B_r) \, dr  ~~\mbox{and}~~
\mathcal{A}_{t} = \int_S^{t}  f(\psi_r + B_r) - f(\phi_r+ B_r)  \, dr  .
\end{align}
Assume without any loss of generality that $[\psi]_{\mathcal{C}^{1/2+H}_{[S,T]} L^{m,\infty}}$, $[\psi-\phi]_{\mathcal{C}^{\tau}_{[S,T]} L^{m}}$ and $\|\psi_S-\phi_S\|_{L^m}$ are finite, otherwise the result is trivial. 

Let $\varepsilon \in (0,\gamma-(1-1/(2H)))$. In the following, we check the conditions in order to apply Lemma~\ref{lemSSL} (with $q=m$). To show that \eqref{sts1} and \eqref{sts2} hold true with $\varepsilon_1=  \tau \wedge \frac{1}{2} + H(\gamma - 1) {>0}$, $\alpha_1=0$ and $\varepsilon_2=1/2+H(\gamma-1)+\varepsilon/2>0$, $\alpha_2=0$, we prove that there exists a constant $C>0$ independent of $s,t,S$ and $T$ such that for $u = (s+t)/2$,
\begin{enumerate}[label=(\roman*)]
\item \label{item51(1)} 
$\|\EE^{s} [\delta A_{{s},u,{t}}]\|_{L^m}\leq C\, \|f \|_{\mathcal{B}_\infty^\gamma} ( [\psi]_{\mathcal{C}^{1/2+H}_{[S,T]} L^{m,\infty}}  + 1 ) ( [ \psi-\phi]_{\mathcal{C}^{\tau}_{[S,T]} L^{m}} +\| \psi_{S}-\phi_{S} \|_{L^{m}}) (t-s)^{1+\varepsilon_1} $\text{;}

\item \label{item51(2)} 
$\| \delta A_{{s},u,{t}}\|_{L^m} \leq C\, \|f\|_{\mathcal{B}_\infty^\gamma}\,\Big(  [\psi]_{\mathcal{C}^{1/2+H}_{[S,T]}L^{m,\infty}} +1 \Big) ( [ \psi-\phi ]_{\mathcal{C}^{\tau}_{[S,T]} L^{m}} + \|\psi_S-\phi_S\|_{L^m} ) \,   (t-u)^{\frac{1}{2}+\varepsilon_2} $\text{;}

\item \label{item51(3)} If \ref{item51(1)} and \ref{item51(2)} are satisfied, \eqref{sts3} gives the 
convergence in probability of $\sum_{i=1}^{N_n-1} A_{t^n_i,t^n_{i+1}}$ along any sequence of 
partitions $\Pi_n=\{t_i^n\}_{i=1}^{N_n}$ of $[S,t]$ with mesh converging to $0$. We will prove 
that the limit is the process $\mathcal{A}$ given in \eqref{eqProp51A}.
\end{enumerate}

Assume for now that \ref{item51(1)}, \ref{item51(2)} and \ref{item51(3)} hold. Applying Lemma~\ref{lemSSL} and recalling \eqref{eqbounds-nu}, we obtain that
\begin{equation}\label{eqssl-goal}
\begin{split}
    \Big\| \int_{s}^{t} & f(B_r+\psi_r) - f(B_r+\phi_r) \, dr \Big\|_{L^m} \\
   &\leq \| A_{{s},{t}}\|_{L^m}+ C\, \|f\|_{\mathcal{B}_\infty^\gamma} ([\psi]_{\mathcal{C}^{1/2+H}_{[S,T]}L^{m,\infty}} +1) ([\psi-\phi ]_{\mathcal{C}^{\tau}_{[S,T]} L^{m}} + \|\psi_S-\phi_S\|_{L^m} ) (t-s)^{1+H(\gamma-1)+ \tau \wedge \frac{1}{2}}  \\
    &\quad +C\, \|f \|_{\mathcal{B}_\infty^\gamma} ( [\psi]_{\mathcal{C}^{1/2+H}_{[S,T]}L^{m,\infty}}  + 1 )  ([ \psi-\phi]_{\mathcal{C}^{\tau}_{[S,T]} L^{m}} +\| \psi_{S}-\phi_{S} \|_{L^{m}}) (t-s)^{1 + H(\gamma-1)+\frac{\varepsilon}{2}} .
\end{split}
\end{equation}
To bound $\| A_{{s},{t}}\|_{L^m}$, we apply Lemma~\ref{lem1streg} with $q=m$ and $\beta=\gamma-1$, and for $\Xi = (\psi_{s},\phi_{s})$.
As $f$ is smooth and bounded, the first assumption of Lemma~\ref{lem1streg} is verified. By Lemma~\ref{lembesov-spaces}$(i)$, $ \|f(\cdot + \psi_{s})-f(\cdot + \phi_{s})\|_{\mathcal{B}^{\gamma-1}_{\infty}} \leq 2\| f\|_{\mathcal{B}^{\gamma-1}_{\infty}}$, hence the second assumption of Lemma~\ref{lem1streg} is verified. It follows by Lemma~\ref{lem1streg} that 
\begin{align}\label{eqAst-}
\| A_{s,t}\|_{L^m} &\leq C\, \| \|f(\psi_s+\cdot)-f(\phi_s+\cdot)  \|_{\mathcal{B}_\infty^{\gamma-1}} \|_{L^m} \, (t-s)^{1+H(\gamma-1)}\ \nonumber\\
&\leq C\, \|f\|_{\mathcal{B}^{\gamma}_\infty} \, \| \psi_{s}-\phi_{s} \|_{L^{m}} \, (t-s)^{1+H(\gamma-1)}\nonumber\\
&\leq  C\, \|f\|_{\mathcal{B}_\infty^\gamma} ( [ \psi-\phi ]_{\mathcal{C}^{\tau}_{[S,T]} L^{m}} + \|\psi_S-\phi_S\|_{L^m} ) \, (t-s)^{1+H(\gamma-1)}.
\end{align}
Injecting the previous bound in \eqref{eqssl-goal}, we get \eqref{eqssl-o-on-2-SDE}.

We now check that the conditions \ref{item51(1)}, \ref{item51(2)} and \ref{item51(3)} actually hold.

\smallskip

Proof of \ref{item51(1)}:  For $u \in [s,t]$, we have $\delta A_{s,u,t} = \int_u^{t}   f(\psi_{s} +B_r) - f(\phi_{s}+ B_r) -  f(\psi_u+ B_r) +f(\phi_u + B_r) \, d r$.
By the tower property of conditional expectation and Fubini's theorem, we have
\begin{align}\label{eqdeltaAst-decomp-SDE}
\mathbb{E}^{s} \delta A_{s,u,t} & =  \int_u^{t} \mathbb{E}^{s}\, \mathbb{E}^u \Big[ f(\psi_{s} +B_r) - f(\phi_{s}+ B_r) - f(\psi_u + B_r) + f(\phi_u+ B_r) \Big] \, dr  \nonumber \\ 
 & =\colon\int_u^{t} \mathbb{E}^{s}\, \mathbb{E}^u [ F(B_r,s,u) + \tilde{F}(B_r,s,u) ] \diff r  ,
\end{align}
where 
\begin{align*}
F(\cdot,s,u) & = f(\psi_{s} + \cdot)-f(\phi_{s} +  \cdot) - f(\psi_u+\cdot)  + f(\psi_u + \phi_{s} - \psi_{s} + \cdot), \\
\tilde{F}(\cdot,s,u)& =  f(\phi_u+\cdot) - f(\psi_u+ \phi_{s} - \psi_{s} + \cdot).
\end{align*}
By Lemma \ref{lemreg-B}$(ii)$, we have that
\begin{align*}
| \mathbb{E}^u F(B_r,s,u) | \leq \|  F(\cdot,s,u) \|_{\mathcal{B}_\infty^{\gamma-2}} \,  (r-u)^{H(\gamma-2)} ,\\ 
| \mathbb{E}^u \tilde{F}(B_r,s,u) | \leq  \|\tilde{F}(\cdot,s,u) \|_{\mathcal{B}_\infty^{\gamma-1}} \, (r-u)^{H(\gamma-1)}  .
\end{align*}
Moreover, by Lemma \ref{lembesov-spaces}$(iii)$, it comes that 
\begin{align*}
\EE^s \| F(\cdot,s,u) \|_{\mathcal{B}_\infty^{\gamma-2}} &\leq  \|f \|_{\mathcal{B}_\infty^\gamma}\, | \psi_{s}-\phi_{s} |\, \EE^s| \psi_{s}-\psi_u |\\
&\leq \|f \|_{\mathcal{B}_\infty^\gamma}\, | \psi_{s}-\phi_{s} |\, \left(\EE^s| \psi_{s}-\psi_u |^m\right)^{\frac{1}{m}}\\
&\leq \|f \|_{\mathcal{B}_\infty^\gamma}\, | \psi_{s}-\phi_{s} |\, [\psi]_{\mathcal{C}^{\frac{1}{2}+H}_{[S,T]} L^{m,\infty}} (u-s)^{\frac{1}{2}+H}  .
\end{align*}
Besides,
\begin{align*}
\| \| \tilde{F}(\cdot,s,u) \|_{\mathcal{B}_\infty^{\gamma-1}} \|_{L^{m}} & \leq \|f \|_{\mathcal{B}_\infty^\gamma} \| \psi_{s}-\psi_u -\phi_{s} + \phi_u \|_{L^{m}} \\
& \leq  \|f \|_{\mathcal{B}_\infty^\gamma} \, [ \psi-\phi ]_{\mathcal{C}^{\tau}_{[S,T]} L^{m}}\, (u-s)^{\tau}  .
\end{align*}
Plugging the previous bounds in \eqref{eqdeltaAst-decomp-SDE} and using $\| \psi_{s}-\phi_{s} \|_{L^m} \leq \| \psi_{S}-\phi_{S} \|_{L^m} + (T-S)^\tau [ \psi-\phi]_{\mathcal{C}^{\tau}_{[S,T]} L^{m}}$, we obtain
\begin{align*}
\| \mathbb{E}^{s} \delta A_{s,u,t} \|_{L^{m}} & \leq C\, \|f \|_{\mathcal{B}_\infty^\gamma}  \Big(  [\psi]_{\mathcal{C}^{\frac{1}{2}+H}_{[S,T]} L^{m,\infty}}  + 1 \Big) ( [ \psi-\phi]_{\mathcal{C}^{\tau}_{[S,T]} L^{m}} +\| \psi_{S}-\phi_{S} \|_{L^{m}} )\\
&\quad \times \left( (t-s)^{1 + \tau + H(\gamma- 1)}  + (t-s)^{1 + \frac{1}{2} + H + H(\gamma - 2)}  \right)  \\
& \leq  C\, \|f \|_{\mathcal{B}_\infty^\gamma} \, \Big(  [\psi]_{\mathcal{C}^{\frac{1}{2}+H}_{[S,T]} L^{m,\infty}}  + 1 \Big)  ([ \psi-\phi]_{\mathcal{C}^{\tau}_{[S,T]} L^{m}} +\| \psi_{S}-\phi_{S} \|_{L^{m}} )(t-s)^{1 + \tau \wedge \frac{1}{2} + H(\gamma - 1)} .
\end{align*}

\smallskip

Proof of \ref{item51(2)}: Apply Corollary~\ref{cor4inc} with $p=+\infty$, $\beta=\gamma$, $\lambda=1$, $\lambda_{1}=1$, $\lambda_{2}=\varepsilon$, $\kappa_{1}=\psi_{s}$, $\kappa_{2}=\phi_{s}$, $\kappa_{3}=\psi_{u}$ and $\kappa_{4}=\phi_{u}$. This yields
\begin{align*}
\| \delta A_{s,u,t} \|_{L^m} 
&\leq C \|f\|_{\mathcal{B}_\infty^\gamma}\, \|\EE^s[|\psi_{s}-\psi_{u}|^m]^{1/m}\|_{L^\infty}^{\varepsilon}\, \|\psi_{s}-\phi_{s}\|_{L^m}\, (t-u)^{1+H(\gamma-1-\varepsilon)} \\
    &\quad + C\|f\|_{\mathcal{B}_\infty^\gamma}\, \|\psi_{s}-\phi_{s} - \psi_{u} +\phi_{u}\|_{L^m}\,  (t-u)^{1+H(\gamma-1)}.
\end{align*}
Hence we get from \eqref{eqdefbracket} that 
\begin{align*}
\| \delta A_{s,u,t} \|_{L^m} 
&\leq C \|f\|_{\mathcal{B}_\infty^\gamma}\, [\psi]_{\mathcal{C}^{\frac{1}{2}+H}_{[S,T]}L^{m,\infty}}^{\varepsilon}\, \|\psi_{s}-\phi_{s}\|_{L^m}\, (t-u)^{1+H(\gamma-1)+\frac{\varepsilon}{2}} \\
    &\quad + C\|f\|_{\mathcal{B}_\infty^\gamma}\, [\psi-\phi]_{\mathcal{C}^\tau_{[S,T]}L^m}\,  (t-u)^{1+H(\gamma-1)+\tau},
\end{align*} 
and use $[\psi]_{\mathcal{C}^{1/2+H}_{[S,T]} L^{m,\infty}}^\varepsilon \leq [\psi]_{\mathcal{C}^{1/2+H}_{[S,T]} L^{m,\infty}}+1$ and $\|\psi_{s}-\phi_{s}\|_{L^m} \leq \|\psi_{S}-\phi_{S}\|_{L^m} + [\psi-\phi]_{\mathcal{C}^\tau_{[S,T]}L^m}$ to prove \ref{item51(2)}.

\smallskip

Proof of \ref{item51(3)}: Finally, for a sequence $(\Pi_n)_{n \in 
\mathbb{N}}$ of partitions of $[S,t]$ with $\Pi_n=\{t_i^n\}_{i=1}^{N_n}$ and mesh size $|\Pi_n|$
converging to zero, we have
\begin{align*}
\left\| \mathcal{A}_{t} - \sum_{i=1}^{N_{n}-1} A_{t_i^n,t_{i+1}^n} \right\|_{L^m}
& \leq \sum_{i=1}^{N_{n}-1} \int_{t_i^n}^{t_{i+1}^n} \left\| f(\psi_r +B_r) - f(\phi_r + B_r) - f(\psi_{t_i^n}+B_r) + f(\phi_{t_i^n}+B_r) \right\|_{L^m}  \, dr\\
& \leq \sum_{i=1}^{N_{n}-1} \int_{t_i^n}^{t_{i+1}^n} \left\| f(\psi_r +B_r)- f(\psi_{t_i^n}+B_r) \right\|_{L^m} + \left\| f(\phi_r + B_r) - f(\phi_{t_i^n}+B_r) \right\|_{L^m} dr\\
&\leq \|f\|_{\mathcal{C}^1} \sum_{i=1}^{N_{n}-1} \int_{t_i^n}^{t_{i+1}^n} \left\| \psi_r - \psi_{t_i^n}\right\|_{L^m} + \left\|\phi_r - \phi_{t_i^n}\right\|_{L^m} dr.
\end{align*}
Now use that $\| \psi_r - \psi_{t_i^n}\|_{L^m} \leq [\psi]_{\mathcal{C}^{1/2+H}_{[S,T]}L^m} |\Pi_n|^{1/2+H} \leq [\psi]_{\mathcal{C}^{1/2+H}_{[S,T]}L^{m,\infty}} |\Pi_n|^{1/2+H}$ and for $\sigma={\tau\wedge (1/2+H)}$, $\| \phi_r - \phi_{t_i^n}\|_{L^m} \leq C ([\phi-\psi]_{\mathcal{C}^{\tau}_{[S,T]}L^m} + [\psi]_{\mathcal{C}^{1/2+H}_{[S,T]}L^{m,\infty}} ) |\Pi_n|^{\sigma}$ to get
\begin{align*}
\left\| \mathcal{A}_{t} - \sum_{i=1}^{N_{n}-1} A_{t_i^n,t_{i+1}^n} \right\|_{L^m} \leq C\|f\|_{\mathcal{C}^1} (T-S) ([\phi-\psi]_{\mathcal{C}^{\tau}_{[S,T]}L^m} + [\psi]_{\mathcal{C}^{1/2+H}_{[S,T]}L^{m,\infty}} ) |\Pi_n|^{\sigma} \underset{n \to \infty}{\longrightarrow} 0.
\end{align*}
\end{proof}


\subsection{Regularisation properties of the $d$-dimensional discrete-time fBm}\label{subsecbound-E2}

In this section, the same quantities as in Section~\ref{secproofs-SDE} are considered, but with a discrete-time noise.

\begin{lemma}\label{lembound-Khn}

Recall that $\gamma$ satisfies \eqref{eqcond-gamma-p-H}. Let $m \in [2, \infty)$, $q \in [m,\infty]$.  There exists a constant $C>0$ such that for any $0 \leq S < T \leq 1$, any $\mathbb{R}^d$-valued $\mathcal{F}_S$-measurable random variable $\psi$, any $f \in \mathcal{C}^\infty_{b}(\mathbb{R}^d, \mathbb{R}^d) \cap \mathcal{B}_\infty^\gamma$, any $h>0$ and any
 $(s, t)\in \Delta_{S,T}$, we have
\begin{align*}%
\Big\| \Big( \EE^S \Big| \int_s^t  f(\psi + B_{r_h}) \diff r  \Big|^m \Big)^{\frac{1}{m}} \Big\|_{L^q} \nonumber \leq C\, \Big( \|f\|_\infty\,  h^{\frac{1}{2}-H} + \| f \|_{\mathcal{B}_\infty^\gamma} \Big) (t-s)^{\frac{1}{2}+H}  .
\end{align*}
\end{lemma}

\begin{proof}
We will check the conditions in order to apply Lemma~\ref{lemSSL}. For $(s,t) \in \Delta_{S,T}$, let 
\begin{align*}
A_{s,t} = \mathbb{E}^s \int_s^t  f(\psi + B_{r_h}) \diff r  ~~\mbox{and}~~
\mathcal{A}_t = \int_S^t f(\psi+ B_{r_h}) \diff r .
\end{align*}
Let $u\in [s,t]$ and notice that $\mathbb{E}^s \delta A_{s,u,t}=0$, so \eqref{sts1} holds with $\Gamma_1=0$. We will prove that \eqref{sts2} holds with $\alpha_2=0$ and
$$ \Gamma_2 =  C \|f\|_\infty   h^{\frac{1}{2}-H} + C \| f \|_{\mathcal{B}_p^\gamma} .$$
\paragraph{The case $t-s\leq 2h$.} In this case we have
\begin{align}\label{eqAst<h-Khn}
|A_{s,t}| \leq \|f\|_\infty (t-s) \leq C  \|f\|_\infty h^{\frac{1}{2}-H} (t-s)^{\frac{1}{2}+H} .
\end{align}

\paragraph{The case $t-s>2h$.} Here we split $A_{s,t}$ in two
\begin{equation*}
A_{s,t} = \mathbb{E}^s \int_s^{s+2h}  f(\psi + B_{r_h}) \diff r + \mathbb{E}^s \int_{s+2h}^t  f(\psi + B_{r_h}) \diff r .
\end{equation*}
For the first part, we obtain 
\begin{equation*}
\Big|\mathbb{E}^s \int_s^{s+2h} f(\psi + B_{r_h}) \diff r \Big| \leq 2h\, \|f\|_\infty \leq C\, \|f\|_\infty h^{\frac{1}{2}-H} (t-s)^{\frac{1}{2}+H}.
\end{equation*}
Denote the second part by
\begin{align*}
J \colon= \int_{s+2h}^t \mathbb{E}^s f(\psi+ B_{r_h})  \diff r .
\end{align*}
Using Lemma \ref{lemreg-B}$(ii)$ and Lemma~\ref{lembesov-spaces}$(i)$, we have
\begin{align*}
\|J \|_{L^q} & \leq C \int_{s+2h}^t  \| f \|_{\mathcal{B}_\infty^\gamma} (r_{h}-s)^{H\gamma}  \diff r .
\end{align*}
Since $2(r_{h}-s) \geq r-s$, we obtain
\begin{align}\label{eqAst>h-Khn}
\|J \|_{L^q} & \leq C \int_{s+2h}^t  \| f \|_{\mathcal{B}_\infty^\gamma} (r-s)^{H\gamma}  \diff r \nonumber\\
& \leq C \| f \|_{\mathcal{B}_\infty^\gamma} (t-s)^{1+H\gamma} \nonumber\\
& \leq C \| f \|_{\mathcal{B}_\infty^\gamma} (t-s)^{\frac{1}{2}+H} .
\end{align}
Overall, combining \eqref{eqAst<h-Khn} and \eqref{eqAst>h-Khn}, we obtain that for all $s \leq t$,
\begin{align*}
\| A_{s,t}\|_{L^{q}} & \leq C \Big(   \|f\|_{\infty} h^{\frac{1}{2}-H} (t-s)^{\frac{1}{2}+H} + \| f \|_{\mathcal{B}_\infty^\gamma} (t-s)^{\frac{1}{2}+H} \Big) .
\end{align*}
Thus for any $u\in [s,t]$,
\begin{align*}
\| \delta A_{s,u,t}\|_{L^{q}} & \leq \| A_{s,t}\|_{L^{m}}+\| A_{s,u}\|_{L^{m}}+\| A_{u,t}\|_{L^{m}}\\
&\leq  C \Big(   \|f\|_{\infty} h^{\frac{1}{2}-H} + \| f \|_{\mathcal{B}_\infty^\gamma} \Big)  (t-s)^{\frac{1}{2}+H} .
\end{align*}
The power in $(t-s)$ is strictly larger than $1/2$, so \eqref{sts2} holds.

\paragraph{Convergence in probability.}
Finally, for a sequence $(\Pi_n)_{n \in 
\mathbb{N}}$ of partitions of $[S,t]$ with $\Pi_n=\{t_i^n\}_{i=1}^{N_n}$ and mesh size 
converging to zero, we have
\begin{align*}
\Big\| \mathcal{A}_t -  \sum_{i=1}^{N_{n}-1} A_{t_i^n,t_{i+1}^n} \Big\|_{L^1} & = \EE \left| \sum_{i=1}^{N_{n}-1} \int_{t_i^n}^{t_{i+1}^n}  f(\psi+ B_{r_h})  -\mathbb{E}^{t_i^n}[  f(\psi+ B_{r_h})]  \diff r \right| \\
& \leq  \sum_{i=1}^{N_{n}-1} \int_{t_i^n}^{t_{i+1}^n} \mathbb{E}\Big| f(\psi +B_{r_h}) -\mathbb{E}^{t_i^n} f(\psi +B_{r_h}) \Big| \diff r.
\end{align*}
Note that if $r_h \leq t_i^n$, then $\mathbb{E}| f(\psi +B_{r_h}) -\mathbb{E}^{t_i^n} f(\psi +B_{r_h})| = 0$. On the other hand, when $r_h \in (t_i^n,t_{i+1}^n]$ then in view of Lemma~\ref{lemreg-B}$(iii)$, we have
\begin{align*}
\mathbb{E}| f(\psi +B_{r_h}) -\mathbb{E}^{t_i^n} f(\psi +B_{r_h}) | \leq C \| f \|_{\mathcal{C}^1} |\Pi_n|^H .
\end{align*}
It follows that
\begin{align*}
\Big\| \mathcal{A}_t -  \sum_{i=1}^{N_{n}-1} A_{t_i^n,t_{i+1}^n} \Big\|_{L^1} & \leq \sum_{i=1}^{N_{n}-1} \int_{t_i^n}^{t_{i+1}^n} \| f \|_{\mathcal{C}^1} |\Pi_n|^H \diff r ,
\end{align*}
and therefore $\sum_{i=1}^{N_{n}-1} A_{t_{i}^n, t_{i+1}^n}$ converges in probability to $\mathcal{A}_{t}$ as $n \to +\infty$.
Hence we can apply Lemma~\ref{lemSSL} with $\varepsilon_1>0$ and $\varepsilon_2 = H$ to conclude that
\begin{align*}
\big\| \big( \EE^S | \mathcal{A}_t - \mathcal{A}_s |^m \big)^{\frac{1}{m}} \big\|_{L^q} & \leq \big\| \big(\EE^S | \mathcal{A}_t - \mathcal{A}_s- A_{s,t}  |^m \big)^{\frac{1}{m}} \big\|_{L^q} + \| A_{s,t} \|_{L^q}  \\
& \leq C\, \Big( \|f\|_\infty \, h^{\frac{1}{2}-\varepsilon} +\| f \|_{\mathcal{B}_\infty^\gamma}  \Big)  \, (t-s)^{\frac{1}{2}+H }  .
\end{align*}
\end{proof}

\begin{proposition}\label{propbound-Khn}
Recall that $\gamma$ satisfies \eqref{eqcond-gamma-p-H}. Let $\varepsilon \in (0,\frac{1}{2})$ and $m \in [2, \infty)$.
There exists a constant $C>0$ such that for any $\mathbb{R}^d$-valued $\mathbb{F}$-adapted process $(\psi_{t})_{t\in [0,1]}$, any  $f \in \mathcal{C}^\infty_{b}(\mathbb{R}^d, \mathbb{R}^d) \cap \mathcal{B}_\infty^\gamma$, any $0 \leq S < T \leq 1$ and any $(s, t)\in \Delta_{S,T}$, we have
\begin{align}\label{eqbound-Khn-general}
\Big\| \Big( \EE^S \Big| \int_s^t f(\psi_{r} + B_{r_h}) \diff r \Big|^m \Big)^{\frac{1}{m}}  \Big\|_{L^{\infty}} & \leq  C\, \Big( \|f\|_\infty\,  h^{\frac{1}{2}-H} + \| f \|_{\mathcal{B}_\infty^\gamma} \Big) (t-s)^{\frac{1}{2}+H}  \nonumber \\ 
&~ + C\,  [ \psi ]_{\mathcal{C}^{\frac{1}{2}+H}_{[S,T]} L^{m,\infty}}  \, \Big( \| f \|_{\mathcal{B}_\infty^{\gamma}} + \| f \|_{\mathcal{C}^1} h^{\frac{1}{2}+H-\varepsilon} \Big)   \, (t-s)^{1+\varepsilon} .
\end{align}
\end{proposition}

\begin{remark}\label{rmk:4.13}
A direct consequence of this proposition is that for any $(s, t)\in \Delta_{0,1}$, we have
\begin{align}\label{eqbound-general-discrete}
\Big\| \Big( \EE^s \Big| \int_s^t f(\psi_{r} + B_{r_h}) \diff r \Big|^m \Big)^{\frac{1}{m}}  \Big\|_{L^{\infty}} & \leq  C\, \Big( \|f\|_\infty\,  h^{\frac{1}{2}-H} + \| f \|_{\mathcal{B}_\infty^\gamma} \Big) (t-s)^{\frac{1}{2}+H}  \nonumber \\ 
&~ + C\,  [ \psi ]_{\mathcal{C}^{\frac{1}{2}+H}_{[s,t]} L^{m,\infty}}  \, \Big( \| f \|_{\mathcal{B}_\infty^{\gamma}} + \| f \|_{\mathcal{C}^1} h^{\frac{1}{2}+H-\varepsilon} \Big)   \, (t-s)^{1+\varepsilon} .
\end{align}
\end{remark}

\begin{proof}
We will check the conditions in order to apply Lemma~\ref{lemSSL} (with $q=\infty$). Let ${0 \leq S < T \leq 1}$. Assume that $[\psi]_{\mathcal{C}^{1/2+H}_{[S,T]}L^{m,\infty}}<\infty$, otherwise 
\eqref{eqbound-Khn-general} trivially holds. 
For any $(s,t) \in \Delta_{S,T}$, define 
\begin{align*}
A_{s,t} = \int_s^t f(\psi_s+ B_{r_h}) \diff r   ~~\mbox{and}~~
\mathcal{A}_t = \int_S^t f(\psi_r+ B_{r_h}) \diff r  .
\end{align*}
To show that \eqref{sts1} and \eqref{sts2} hold true with $\varepsilon_1= \varepsilon$, $\varepsilon_2=H >0$ and $\alpha_1=\alpha_2=0$, we prove that there exists a constant $C>0$ independent of $s,t,S$ and $T$ such that for $u = (s+t)/2$,
\begin{enumerate}[label=(\roman*)]
\item \label{item56(1)sde} 
$\|\EE^{{s}} [\delta A_{s,u,t}]\|_{L^\infty} \leq  C\, [ \psi ]_{\mathcal{C}^{\frac{1}{2}+H}_{[S,T]} L^{m,\infty}} \, \Big(\| f \|_{\mathcal{B}_\infty^{\gamma}} + \| f \|_{\mathcal{C}^1} h^{\frac{1}{2}+H-\varepsilon} \Big) \, (t-s)^{1+\varepsilon} $\text{;}

\item \label{item56(2)sde} 
$\big\| \big( \EE^S | \delta A_{s,u,t} |^m \big)^{\frac{1}{m}} \big\|_{L^\infty} \leq C\, \Big( \|f\|_{\infty}\, h^{\frac{1}{2}-H} + \| f \|_{\mathcal{B}_\infty^\gamma} \Big) \, ({t}-{s})^{\frac{1}{2}+H}$\text{;}

\item \label{item56(3)sde} If \ref{item54(1)} and \ref{item54(2)} are satisfied, \eqref{sts3} gives the 
convergence in probability of $\sum_{i=1}^{N_n-1} A_{t^n_i,t^n_{i+1}}$ along any sequence of 
partitions $\Pi_n=\{t_i^n\}_{i=1}^{N_n}$ of $[S,t]$ with mesh converging to $0$. We will prove 
that the limit is the process $\mathcal{A}$ given above.
\end{enumerate}

Assume for now that \ref{item56(1)sde}, \ref{item56(2)sde} and \ref{item56(3)sde} hold. Applying Lemma~\ref{lemSSL}, we obtain that
\begin{align*}
 \Big\| \Big( \EE^S \Big| \int_s^t f(\psi_{r} + B_{r_h}) \diff r \Big|^m \Big)^{\frac{1}{m}}  \Big\|_{L^{\infty}} 
   & \leq C\, \Big( \|f\|_{\infty}\, h^{\frac{1}{2}-H} + \| f \|_{\mathcal{B}_\infty^\gamma} \Big) \, ({t}-{s})^{\frac{1}{2}+H}  \\ 
   & ~ + C\, [ \psi ]_{\mathcal{C}^{\frac{1}{2}+H}_{[S,T]} L^{m,\infty}} \, \Big(\| f \|_{\mathcal{B}_\infty^{\gamma}} + \| f \|_{\mathcal{C}^1} h^{\frac{1}{2}+H-\varepsilon} \Big) \, (t-s)^{1+\varepsilon} \\
   &~+\big\| \big(\EE^S |A_{s,t}|^m\big)^{\frac{1}{m}} \big\|_{L^\infty}.
\end{align*}
Applying Lemma~\ref{lembound-Khn} with $q=\infty$ and $\psi = \psi_{s}$ for the last  term of the previous equation, we get \eqref{eqbound-Khn-general}.%

We now check that the conditions \ref{item56(1)sde}, \ref{item56(2)sde} and \ref{item56(3)sde} actually hold.

\smallskip

Proof of \ref{item56(1)sde}:  We have

\begin{align*}
\mathbb{E}^s \delta A_{s,u,t} & = \int_u^t \mathbb{E}^s [ f(\psi_s+ B_{r_h})- f(\psi_u+ B_{r_h}) ] \diff r  .
\end{align*}

\paragraph{The case $t-u \leq 2h$.}
In this case, using the Lipschitz norm of $f$, we have
\begin{align*}
| \mathbb{E}^s \delta A_{s,u,t} | & \leq  \|f\|_{\mathcal{C}^1} \int_u^t  \EE^s|\psi_s-\psi_u|  \diff r \leq \|f\|_{\mathcal{C}^1} \, (t-u)  \left(\EE^s|\psi_s-\psi_u|^m\right)^{\frac{1}{m}}. 
\end{align*}
Thus using the inequality $(t-u) (u-s)^{1/2+H} \leq C\, h^{1/2+H-\varepsilon} (t-u)^{1/2-H+\varepsilon} (u-s)^{1/2+H}  \leq C (t-s)^{1+\varepsilon} h^{\frac{1}{2}+H-\varepsilon}$, it comes
\begin{align*}
\| \mathbb{E}^s \delta A_{s,u,t} \|_{L^\infty} 
& \leq C\, \|f\|_{\mathcal{C}^1}\,  [ \psi ]_{\mathcal{C}^{\frac{1}{2}+H}_{[S,T]} L^{m,\infty}}  (t-s)^{1+\varepsilon} h^{\frac{1}{2}+H-\varepsilon} .
\end{align*}

\paragraph{The case $t-u > 2h$.} We split the integral between $u$ and $u+2h$ and then between $u+2h$ and $t$ as follows:
\begin{align*}
\EE^s \delta A_{s,u,t} &= \int_u^{u+2h} \mathbb{E}^s [ f(\psi_s+ B_{r_h})- f(\psi_u+ B_{r_h}) ] \diff r + \int_{u+2h}^t \mathbb{E}^s [ f(\psi_s+ B_{r_h})- f(\psi_u+ B_{r_h}) ] \diff r\\
&=\colon J_1 + J_2  .
\end{align*}
For $J_1$, we obtain as in the case $t-u\leq 2h$ that
\begin{align}\label{eqJ1Prop56}
\| J_{1}\|_{L^\infty} = \|\mathbb{E}^s \delta A_{s,u,u+2h}\|_{L^\infty}
& \leq  C  \|f\|_{\mathcal{C}^1} [ \psi ]_{\mathcal{C}^{\frac{1}{2}+H}_{[S,T]} L^{m,\infty}} (u+2h-s)^{1+\varepsilon}\, h^{\frac{1}{2}+H-\varepsilon} \nonumber\\
& \leq  C  \|f\|_{\mathcal{C}^1} [ \psi ]_{\mathcal{C}^{\frac{1}{2}+H}_{[S,T]} L^{m,\infty}} (t-s)^{1+\varepsilon}\, h^{\frac{1}{2}+H-\varepsilon} .
\end{align}
As for $J_2$, the tower property of the conditional expectation yields
\begin{align*}
J_{2} =\mathbb{E}^s \int_{u+2h}^t \mathbb{E}^u [ f(\psi_s+ B_{r_h})- f(\psi_u+ B_{r_h}) ] \diff r .
\end{align*}
Now use Lemma~\ref{lemreg-B}$(ii)$ and Lemma~\ref{lembesov-spaces}$(ii)$ to obtain
\begin{align*}
|J_{2}| &\leq C  \int_{u+2h}^t \mathbb{E}^s\| f(\psi_s+ \cdot) - f(\psi_u+ \cdot) \|_{\mathcal{B}_\infty^{\gamma-1}} (r_h-u)^{H(\gamma-1)}\, dr \\
&\quad \leq C \,  \|f \|_{\mathcal{B}_\infty^{\gamma}}\, \left(\mathbb{E}^s|\psi_{s}-\psi_{u}|^m\right)^{\frac{1}{m}} \int_{u+2h}^t (r_{h}-u)^{H(\gamma-1)} \, dr \, .
\end{align*}
Using the fact that $2(r_h-u) \ge (r-u)$, it comes
\begin{align}\label{eqJ2Prop56}
\| J_{2} \|_{L^\infty} & \leq C\,[ \psi ]_{\mathcal{C}^{\frac{1}{2}+H}_{[S,T]} L^{m,\infty}} \|f \|_{\mathcal{B}_\infty^{\gamma}} \, (u-s)^{\frac{1}{2}+H} \int_{u+2h}^t (r-u)^{H(\gamma-1)}  \, dr \nonumber\\
&\leq  C\, [ \psi ]_{\mathcal{C}^{\frac{1}{2}+H}_{[S,T]} L^{m,\infty}} \| f \|_{\mathcal{B}_\infty^{\gamma}} \, (u-s)^{\frac{1}{2}+H}\,  (t-u)^{1+H(\gamma-1)} \nonumber\\
& \leq  C\, [ \psi ]_{\mathcal{C}^{\frac{1}{2}+H}_{[S,T]} L^{m,\infty}} \| f \|_{\mathcal{B}_\infty^{\gamma}} \, (t-s)^{1+H} .
\end{align}
In view of the inequalities \eqref{eqJ1Prop56} and \eqref{eqJ2Prop56}, we have proven \ref{item56(1)sde}.

\medskip

Proof of \ref{item56(2)sde}: 
We write
\begin{equation*}
\big\| \big( \EE^S | \delta A_{{s},u,{t}} |^m \big)^{\frac{1}{m}} \big\|_{L^\infty} \leq \big\| \big( \EE^S | \delta A_{{s},{t}} |^m \big)^{\frac{1}{m}} \big\|_{L^\infty} + \big\| \big( \EE^S | \delta A_{{s},u} |^m \big)^{\frac{1}{m}} \big\|_{L^\infty} + \big\| \big( \EE^S | \delta A_{u,{t}} |^m \big)^{\frac{1}{m}} \big\|_{L^\infty} .
\end{equation*}
Applying Lemma~\ref{lembound-Khn} with $q=\infty$ for each term in the right-hand side of the previous inequality, respectively for $\psi = \psi_{s}$, $\psi_{s}$ again and $\psi_{u}$, we get \ref{item56(2)sde}. 

\medskip

Proof of \ref{item56(3)sde}: The proof is similar to point~\ref{item51(3)} of Proposition~\ref{propbound-E1-SDE}.
\end{proof}

\subsection{H\"older regularity of the tamed Euler scheme}\label{subsecreg-schema}

The H\"older regularity of the tamed Euler scheme, uniformly in $(h,k)$, can now be deduced from the properties of the previous section.

\begin{corollary}\label{corbound-Khn}
Assume \eqref{eqcond-gamma-p-H}, let $\mathcal{D}$ be a sub-domain of $(0,1) \times \mathbb{N}$ satisfying \eqref{eqassump-bn-bounded} and let $m \in [2, \infty)$. Recall also that $K^{h,k}$ was defined in \eqref{eqdef-Khn}. 
Then
\begin{align*}%
\sup_{(h,k)\in \mathcal{D}} [ K^{h,k} ]_{\mathcal{C}^{\frac{1}{2}+H}_{[0,1]} L^{m,\infty}}  < \infty .
\end{align*}
\end{corollary}
\begin{proof}
Let $(h,k)\in \mathcal{D}$ and $\varepsilon \leq H$.
In view of Equation \eqref{eqbound-general-discrete} from Remark \ref{rmk:4.13}, we have for $f=b^k$ and $\psi=X_0+K^{h,k}$ that there exists a constant $C$ such that for any $h \in (0,1)$, $k \in \mathbb{N}$ and $(s,t) \in \Delta_{0,1}$,
\begin{align*}
\Big\|  \Big( \EE^s \Big| \int_s^t b^k(X_0+K^{h,k}_{r} + B_{r_h}) \, dr \Big|^m  \Big)^{\frac{1}{m}} \Big\|_{L^{\infty}} & \leq  C\, \Big( \|b^k\|_\infty\,  h^{\frac{1}{2}-H} + \| b \|_{\mathcal{B}_\infty^\gamma} \Big) (t-s)^{\frac{1}{2}+H}  \nonumber \\ 
&~ + C\,  [ K^{h,k} ]_{\mathcal{C}^{\frac{1}{2}+H}_{[s,t]} L^{m,\infty}}  \, \Big( \| b \|_{\mathcal{B}_\infty^{\gamma}} + \| b^k \|_{\mathcal{C}^1} h^{\frac{1}{2}+H-\varepsilon} \Big)   \, (t-s)^{1+\varepsilon} ,
\end{align*}
where we used that $\|b^k\|_{\mathcal{B}_\infty^\gamma} \leq \| b \|_{\mathcal{B}_\infty^\gamma}$. In particular, for $0\leq S<T\leq 1$ and any $(s,t)\in \Delta_{S,T}$,
\begin{align*}
\Big\|  \Big( \EE^s \Big| \int_s^t b^k(X_0+K^{h,k}_{r} + B_{r_h}) \, dr \Big|^m  \Big)^{\frac{1}{m}} \Big\|_{L^{\infty}} & \leq  C\, \Big( \|b^k\|_\infty\,  h^{\frac{1}{2}-H} + \| b \|_{\mathcal{B}_\infty^\gamma} \Big) (t-s)^{\frac{1}{2}+H}  \nonumber \\ 
&~ + C\,  [ K^{h,k} ]_{\mathcal{C}^{\frac{1}{2}+H}_{[S,T]} L^{m,\infty}}  \, \Big( \| b \|_{\mathcal{B}_\infty^{\gamma}} + \| b^k \|_{\mathcal{C}^1} h^{\frac{1}{2}+H-\varepsilon} \Big)   \, (t-s)^{1+\varepsilon} .
\end{align*}
Moreover, using that $| K^{h,k}_r-K^{h,k}_{r_h} |  \leq \| b^k \|_\infty h$, we have
\begin{align*}
\Big\|  \Big( \EE^s \big|K^{h,k}_t-K^{h,k}_s \big|^m  \Big)^{\frac{1}{m}} \Big\|_{L^{\infty}} & \leq \Big\|  \Big( \EE^s \Big| \int_s^t b^k(X_0+K^{h,k}_{r_h} + B_{r_h}) -b^k(X_0+K^{h,k}_{r}  + B_{r_h})  \diff r \Big|^m  \Big)^{\frac{1}{m}} \Big\|_{L^{\infty}} \\
&\quad + \Big\|  \Big( \EE^s \Big| \int_s^t b^k(X_0+K^{h,k}_{r} + B_{r_h}) \diff r \Big|^m  \Big)^{\frac{1}{m}} \Big\|_{L^{\infty}} \\
& \leq C\, \| b^k\|_{\mathcal{C}^1} \| b^k \|_\infty h\, (t-s) + \Big\|  \Big( \EE^s \Big| \int_s^t b^k(X_0+K^{h,k}_{r} + B_{r_h}) \diff r \Big|^m  \Big)^{\frac{1}{m}} \Big\|_{L^{\infty}} \\
& \leq C\, \Big( \|b^k\|_\infty\,  h^{\frac{1}{2}-H} + \| b \|_{\mathcal{B}_\infty^\gamma} \Big) (t-s)^{\frac{1}{2}+H}  \nonumber \\ 
&~ + C\, \Big( [ K^{h,k} ]_{\mathcal{C}^{\frac{1}{2}+H}_{[S,T]} L^{m,\infty}}  ( \| b \|_{\mathcal{B}_\infty^{\gamma}} + \| b^k \|_{\mathcal{C}^1} h^{\frac{1}{2}+H-\varepsilon}) +\| b^k\|_{\mathcal{C}^1} \| b^k \|_\infty  h \Big)   \, (t-s)  .
\end{align*}
Now using \eqref{eqassump-bn-bounded} with $\eta=\varepsilon$ small enough, 
we get $\sup_{(h,k)\in \mathcal{D}} \| b^k \|_{\mathcal{C}^1} \|b^k \|_\infty h < \infty$ and
\begin{align*}
\Big\|  \Big( \EE^s \big|K^{h,k}_t-K^{h,k}_s \big|^m  \Big)^{\frac{1}{m}} \Big\|_{L^{\infty}}  \leq C\, (t-s)^{\frac{1}{2}+H} + C\, \Big( [ K^{h,k} ]_{\mathcal{C}^{\frac{1}{2}+H}_{[S,T]} L^{m,\infty}} + 1 \Big)   \, (t-s)  .
\end{align*}
Now divide by $(t-s)^{\frac{1}{2}+H}$ and take the supremum over $[S,T]$ to get that
\begin{align*}
 [ K^{h,k} ]_{\mathcal{C}^{\frac{1}{2}+H}_{[S,T]} L^{m,\infty}}  
\leq  C  + C\,  [ K^{h,k} ]_{\mathcal{C}^{\frac{1}{2}+H}_{[S,T]} L^{m,\infty}} \, (T-S)^{\frac{1}{2}-H} .
\end{align*}
Let $\ell = \big( \frac{1}{2C} \big)^{\frac{1}{1/2-H}} $. Then for $T-S \leq \ell$, we deduce
\begin{align*}
 [ K^{h,k} ]_{\mathcal{C}^{\frac{1}{2}+H}_{[S,T]} L^{m,\infty}} \leq 2C.
\end{align*}
Since $\ell$ does not depend on $h$ nor $k$, we get the H\"older regularity on the whole interval $[0,1]$.
\end{proof}


\subsection{H\"older regularity of $E^{1,h,k}$}\label{subsecE1hn}

We now apply Proposition~\ref{propbound-E1-SDE} in the sub-critical case to obtain the following bound on $E^{1,h,k}$, which is used in Section~\ref{secoverview-SDE}.

\begin{corollary}\label{corbound-E1-SDE}
Let $m \in [2,\infty)$ and assume that $\gamma \in( 1-\frac{1}{2H},0)$. Recall that the process $K^{h,k}$ was defined in \eqref{eqdef-Khn} and let $X_0$ be an $\mathcal{F}_0$-measurable random variable. There exists a constant $C>0$ such that for any $0 \leq S < T \leq 1$, any $(s, t)\in \Delta_{S,T}$, any $h\in (0,1)$ and any $k\in \N$,
\begin{align*}
& \Big\| \int_s^t b^k(X_0 + K_r + B_r) - b^k(X_0 + K^{h,k}_r +  B_r)  \, dr \Big\|_{L^{m}} \\ 
&\quad \leq C \| b \|_{\mathcal{B}_\infty^{\gamma}} \Big( 1 + [X^{h,k}-B]_{\mathcal{C}^{\frac{1}{2}+H}_{[S,T]} L^{m,\infty}}   \Big)  \Big( [K-K^{h,k}]_{\mathcal{C}^{\frac{1}{2}}_{[S,T]} L^{m}}+ \| K_S  - K^{h,k}_S \|_{L^m} \Big)(t-s)^{1+H(\gamma-1)}   .
\end{align*}
\end{corollary}

\begin{proof}
Notice that $\tau = 1/2$ satisfies \eqref{eqtau-cond}. Hence apply Proposition \ref{propbound-E1-SDE} with $\tau=1/2$, $f=b^k$, $\psi =X_0+ K^{h,k}$ and $\phi =   X_0+K$. Then, recall from \eqref{eqconv-in-gamma-} that $\| b^k \|_{\mathcal{B}_\infty^\gamma} \leq \| b \|_{\mathcal{B}_\infty^\gamma}$ to get the result. 
\end{proof}

The following proposition provides a result similar to Corollary~\ref{corbound-E1-SDE} but in the limit case.

\begin{proposition}\label{propbound-E1-SDE-critic}
Let $\gamma=1-1/(2H)$. Let $m \in [2, \infty)$, $p \in [m,+\infty)$ and $b \in \mathcal{B}_p^{\gamma+d/p}$. Denote $\tilde\gamma=\gamma+d/p$ and let $(b^k)_{k \in \mathbb{N}}$ be a sequence of smooth functions such that $\sup_{k \in \N}\| b^k\|_{\mathcal{B}_p^{\tilde\gamma}} \leq \| b\|_{\mathcal{B}_p^{\tilde\gamma}}$. Let $X_0$ be an $\mathcal{F}_0$-measurable random variable, $(X,B)$ be a weak solution to \eqref{eqSDE} satisfying $X-B \in \mathcal{C}^{1/2+H}_{[0,1]} L^m$ and let $(X^{h,k})_{h \in (0,1), k \in \mathbb{N}}$ be the tamed Euler scheme defined in \eqref{defEulerSDE}, on the same probability space and with the same fBm $B$ as $X$. Let $\mathcal{D}$ be a sub-domain of $(0,1) \times \N$ satisfying \eqref{eqassump-bn-bounded}. Let $\zeta\in (0,1/2)$. Recall also that $K^k$ and $K^{h,k}$ were defined in \eqref{eqdef-Khn}, and $\epsilon(h,k)$ was defined in \eqref{eqdefepsilonhn}.

There exist constants $ \mathbf{M} >0$ and $\ell_{0}>0$ such that for any $0 \leq S < T \leq 1$ which satisfy $T-S\leq \ell_{0}$, any $(s, t)\in \Delta_{S,T}$, any $(h,k) \in \mathcal{D}$,
\begin{align*}%
 \Big\| &\int_s^t  b^k( X_0+K_r+ B_r) - b^k(X_0+K^{h,k}_r+  B_r) \, dr \Big\|_{L^{m}}   \\ 
 &\leq  \M \, \bigg(1+ \Big|\log \frac{T^H}{ \|K- K^{h,k} \|_{L^\infty_{[S,T]} L^{m}} +\epsilon(h,k)}\Big| \bigg) \, \Big( \|K- K^{h,k} \|_{L^\infty_{[S,T]} L^{m}} + \epsilon(h,k)\Big) \, (t-s) \nonumber\\
 &\quad+ \M \, \Big(\|K- K^{h,k} \|_{L^\infty_{[S,T]} L^{m}} + [K- K^{h,k} ]_{\mathcal{C}^{\frac{1}{2}-\zeta}_{[S,T]} L^{m}} \Big)\,   (t-s)^{\frac{1}{2}} .
\end{align*}

\end{proposition}

\begin{remark}
The constant $\M$ is important in the proof of Theorem \ref{thmmain-SDE}$(c)$, as it appears in the order of convergence. 
\end{remark}

\begin{proof}
Let $0\leq S<T\leq1$.
For $(s,t)\in \Delta_{S,T}$, let $A_{s,t}$ and $\mathcal{A}_{t}$ be defined by
\begin{equation} \label{eqProp51A-critic}
\begin{split}
&A_{s,t} = \int_{s}^{t} b^k(X_0+K_s+ B_r) - b^k(X_0+K^{h,k}_s+  B_r) \, dr ,\\
&\mathcal{A}_{t} = \int_S^{t} b^k( X_0+K_r+ B_r) - b^k(X_0+K^{h,k}_r+  B_r) \, dr  ,
\end{split}
\end{equation}
 and let
\begin{align} \label{eqR}
R_{s,t} = \mathcal{A}_t-\mathcal{A}_s-A_{s,t}  .
\end{align}
In this proof, we write $\|K- K^{h,k} \|_{\mathcal{C}^{\frac{1}{2}-\zeta}_{[S,T]} L^{m}}$ for $\| K_{S}-K^{h,k}_{S} \|_{L^{m}} + [K- K^{h,k} ]_{\mathcal{C}^{\frac{1}{2}-\zeta}_{[S,T]} L^{m}}$.

Having assumed \eqref{eqassump-bn-bounded}, Corollary~\ref{corbound-Khn} gives that $\sup_{(h,k)\in \mathcal{D}} [ K^{h,k} ]_{\mathcal{C}^{1/2+H}_{[0,1]} L^{m,\infty}}$ is finite. In this proof, the constants sometimes depend on this supremum as well as on $[ K ]_{\mathcal{C}^{1/2+H}_{[0,1]} L^{m}}$.

Let $$ 0 < \varepsilon < \left( \gamma-1+\frac{1}{2H} \right) \wedge (1-2 \zeta) .$$
 Set $\tau=1/2+\varepsilon/2$. In the following, we will check the conditions \eqref{sts1}, \eqref{sts2} and \eqref{sts4} which permit to apply the stochastic sewing with critical exponent \cite[Theorem 4.5]{athreya2020well}. To show that \eqref{sts1}, \eqref{sts2} and \eqref{sts4} hold true with 
$\varepsilon_{1} =H$, $\alpha_1=0$, $\varepsilon_2=1/2+H(\gamma-1-d/p)+\varepsilon/2 = \varepsilon/2$, $\alpha_2=0$ and $\varepsilon_4 = \varepsilon/2$ we prove that there exists a constant $C>0$ independent of $s,t,S$ and $T$ such that for $u = (s+t)/2$,
\begin{enumerate}[label=(\roman*)]
\item \label{item53(1)} 
$\|\EE^{s} \delta A_{{s},u,{t}} \|_{L^m} \leq  C \, (t-s)^{1+H}$\text{;}

\myitem{(i')}\label{item53(1')}
$\|\EE^{s} \delta A_{{s},u,{t}}\|_{L^m}  \leq  C \, [R]_{\mathcal{C}^{\tau}_{[S,T]} L^{m}} \, (t-s)^{\frac{1}{2}+\tau} +C \Big( \|K- K^{h,k} \|_{L^\infty_{[S,T]} L^{m}} + \epsilon(h,k)\Big) {(t-s)}\text{;}$

\item \label{item53(2)} 
$\| \delta A_{{s},u,{t}}\|_{L^m} \leq C\, \|K- K^{h,k} \|_{\mathcal{C}^{\frac{1}{2}-\zeta}_{[S,T]} L^{m}} \,   (t-s)^{\frac{1}{2}+\varepsilon_2} $\text{;}

\item \label{item53(3)} If \ref{item53(1)} and \ref{item53(2)} are satisfied, \eqref{sts3} gives the 
convergence in probability of $\sum_{i=1}^{N_n-1} A_{t^n_i,t^n_{i+1}}$ along any sequence of 
partitions $\Pi_n=\{t_i^n\}_{i=1}^{N_n}$ of $[S,t]$ with mesh converging to $0$. We will prove 
that the limit is the process $\mathcal{A}$ given in \eqref{eqProp51A-critic}.
\end{enumerate}

Then from \cite[Theorem 4.5]{athreya2020well}, we get
\begin{align*}
 \| R_{s,t} \|_{L^m} 
 &\leq  C\, \bigg(1+ \Big|\log \frac{T^H}{ \|K- K^{h,k} \|_{L^\infty_{[S,T]} L^{m}} +\epsilon(h,k)}\Big| \bigg) \, \Big( \|K- K^{h,k} \|_{L^\infty_{[S,T]} L^{m}} + \epsilon(h,k)\Big) \, (t-s) \nonumber\\
 &\quad+ C\, \|K- K^{h,k} \|_{\mathcal{C}^{\frac{1}{2}-\zeta}_{[S,T]} L^{m}} \,   (t-s)^{\frac{1}{2}+\frac{\varepsilon}{2}} +C \, [R]_{\mathcal{C}^{\tau}_{[S,T]} L^{m}} \, (t-s)^{\frac{1}{2}+\tau} .
\end{align*}
Now recalling that $\tau=1/2+\varepsilon/2$, we can divide both sides by $(t-s)^\tau$ and take the supremum over $(s,t) \in \Delta_{S,T}$ to get
\begin{align*}%
 [R]_{\mathcal{C}^{\tau}_{[S,T]} L^{m}} 
 &\leq  C\, \bigg(1+ \Big|\log \frac{T^H}{ \|K- K^{h,k} \|_{L^\infty_{[S,T]} L^{m}} +\epsilon(h,k)}\Big| \bigg) \, \Big( \|K- K^{h,k} \|_{L^\infty_{[S,T]} L^{m}} + \epsilon(h,k)\Big) \, (T-S)^{1-\tau} \nonumber\\
 &\quad+ C\, \|K- K^{h,k} \|_{\mathcal{C}^{\frac{1}{2}-\zeta}_{[S,T]} L^{m}} \,   +C \, [R]_{\mathcal{C}^{\tau}_{[S,T]} L^{m}} \, (T-S)^{\frac{1}{2}} .
\end{align*}
For $S< T$ such that $T-S\leq (2C)^{-2}=\colon \ell_{0}$, we get that $C \, [R]_{\mathcal{C}^{\tau}_{[S,T]} L^{m}} \, (T-S)^{\frac{1}{2}} \leq (1/2) [R]_{\mathcal{C}^{\tau}_{[S,T]} L^{m}}$. We then subtract this quantity on both sides to get
\begin{align*}
 [R]_{\mathcal{C}^{\tau}_{[S,T]} L^{m}}
 &\leq  2 C\,  \bigg(1+ \Big|\log \frac{T^H}{ \|K- K^{h,k} \|_{L^\infty_{[S,T]} L^{m}} +\epsilon(h,k)}\Big| \bigg) \, \Big( \|K- K^{h,k} \|_{L^\infty_{[S,T]} L^{m}} + \epsilon(h,k)\Big) (T-S)^{1-\tau}\\
 &\quad + 2 C\, \|K- K^{h,k} \|_{\mathcal{C}^{\frac{1}{2}-\zeta}_{[S,T]} L^{m}}.
\end{align*}
We conclude using \eqref{eqAProp53} and 
$ \| \mathcal{A}_t-\mathcal{A}_s \|_{L^m}  \leq \| R_{s,t} \|_{L^m}  + \| A_{s,t} \|_{L^m} $.

~

We now check that the conditions \ref{item53(1)}, \ref{item53(1')}, \ref{item53(2)} and \ref{item53(3)} actually hold.

\smallskip

Proof of \ref{item53(1)}:  For $u \in [s,t]$, by the tower property of conditional expectation and Fubini's theorem, we have
\begin{align}\label{eqdeltaAst-decomp-SDE-critic}
\mathbb{E}^{s} \delta A_{s,u,t} & =  \int_u^{t} \mathbb{E}^{s}\, \mathbb{E}^u \Big[ b^k(X_0+K_{s} +B_r) - b^k(X_0+K^{h,k}_{s}+ B_r) \nonumber \\ & \quad  - b^k(X_0+K_u + B_r) + b^k(X_0+K^{h,k}_u+ B_r) \Big] \, dr  \nonumber \\ 
 & =\colon\int_u^{t} \mathbb{E}^{s}\, \mathbb{E}^u [ F(B_r,s,u) + \tilde{F}(B_r,s,u) ] \, dr  , 
\end{align}
where 
\begin{align*}
F(\cdot,s,u) & =-b^k(X_0+K^{h,k}_{s} +B_r) + b^k(X_0+K_{s}+ B_r) + b^k(X_0+K_u^{h,k} + B_r) \\ & \quad - b^k(X_0+K_s + K^{h,k}_{u} - K^{h,k}_{s} + \cdot)  , \\
\tilde{F}(\cdot,s,u)& =-b^k(X_0+K_u+\cdot) + b^k(X_0+K_s + K^{h,k}_{u} - K^{h,k}_{s} + \cdot) .
\end{align*}
By Lemma \ref{lemreg-B}$(ii)$, we have for $\lambda\in [0,1]$ that
\begin{align*}
 |\mathbb{E}^u F(B_r,s,u) | \leq C\, \| F(\cdot,s,u) \|_{\mathcal{B}_p^{\tilde\gamma-1-\lambda}} \, (r-u)^{H(\gamma-1-\lambda)} .
\end{align*}
Moreover, by Lemma \ref{lembesov-spaces}$(iii)$, it comes that
\begin{align*}
\| F(\cdot,s,u) \|_{\mathcal{B}_p^{\tilde\gamma-1-\lambda}} 
& \leq C\, \|b^k \|_{\mathcal{B}_p^{\tilde\gamma}} | K^{h,k}_{s}-K^{h,k}_u |\, | K_{s}-K^{h,k}_{s} |^\lambda.
\end{align*}
Hence
\begin{align*}
 | \EE^s \mathbb{E}^u F(B_r,s,u) |
& \leq C\, \|b^k \|_{\mathcal{B}_p^{\tilde\gamma}} \, | K_{s}-K^{h,k}_{s} |^\lambda \, \EE^s | K^{h,k}_{s}-K^{h,k}_u | \, (r-u)^{H(\gamma-1-\lambda)} \\
& \leq C\, \|b^k\|_{\mathcal{B}_p^{\tilde\gamma}}\, | K_{s}-K^{h,k}_{s} |^\lambda\,  [K^{h,k}]_{\mathcal{C}^{\frac{1}{2}+H}_{[S,T]} L^{m,\infty}} \,  (u-s)^{\frac{1}{2}+H}\, (r-u)^{H(\gamma-1-\lambda)}
\end{align*}
and from Jensen's inequality,
\begin{align}\label{eqboundcondF}
 \| \EE^s \mathbb{E}^u F(B_r,s,u) \|_{L^m}
& \leq C\, \|b^k\|_{\mathcal{B}_p^{\tilde\gamma}}\,  \| K_{s}-K^{h,k}_{s} \|_{L^m}^\lambda\,  [K^{h,k}]_{\mathcal{C}^{\frac{1}{2}+H}_{[S,T]} L^{m,\infty}} \,  (u-s)^{\frac{1}{2}+H}\, (r-u)^{H(\gamma-1-\lambda)}.
\end{align}
As for $\tilde{F}$, we have similarly that
\begin{align}\label{eqboundFtilde}
 \| \EE^s \mathbb{E}^u \tilde{F}(B_r,s,u) \|_{L^m}
 &\leq C\, \| \EE^s\| \tilde{F}(\cdot,s,u) \|_{\mathcal{B}_p^{\tilde\gamma-1}} \|_{L^m} \, (r-u)^{H(\gamma-1)} \nonumber\\
 &\leq C\, \|b^k\|_{\mathcal{B}_p^{\tilde\gamma}}\, \| K^{h,k}_{s}-K^{h,k}_u -K_{s} + K_u \|_{L^m}\, (r-u)^{H(\gamma-1)}\nonumber\\
 &\leq C\, \|b^k\|_{\mathcal{B}_p^{\tilde\gamma}}\, \left( [K^{h,k}]_{\mathcal{C}^{\frac{1}{2}+H}_{[S,T]}L^m} + [K]_{\mathcal{C}^{\frac{1}{2}+H}_{[S,T]}L^m} \right) (u-s)^{\frac{1}{2}+H}\, (r-u)^{H(\gamma-1)} .
\end{align}
Choosing $\lambda=0$ and noticing that $H(\gamma-1) =-1/2$, we plug \eqref{eqboundcondF} and \eqref{eqboundFtilde} in \eqref{eqdeltaAst-decomp-SDE-critic} to obtain
\begin{align*}
\|\EE^s  \delta A_{s,u,t} \|_{L^m} 
&\leq C\, \|b^k\|_{\mathcal{B}_p^{\tilde\gamma}}\,  [K^{h,k}]_{\mathcal{C}^{\frac{1}{2}+H}_{[S,T]} L^{m,\infty}} \, (u-s)\, (t-u)^{\frac{1}{2}} \\
&\quad +  C\, \|b^k\|_{\mathcal{B}_p^{\tilde\gamma}}\, \left( [K^{h,k}]_{\mathcal{C}^{\frac{1}{2}+H}_{[S,T]}L^m} + [K]_{\mathcal{C}^{\frac{1}{2}+H}_{[S,T]}L^m} \right) (u-s)^{\frac{1}{2}+H}\, (r-u)^{\frac{1}{2}}\\
&\leq C\, \|b^k\|_{\mathcal{B}_p^{\tilde\gamma}}\, \left(  [K^{h,k}]_{\mathcal{C}^{\frac{1}{2}+H}_{[S,T]}L^{m,\infty}} + [K]_{\mathcal{C}^{\frac{1}{2}+H}_{[S,T]}L^m}\right) (t-s)^{1+H}  .
\end{align*}
Use that $\sup_{(h,k) \in \mathcal{D}} [K^{h,k} ]_{\mathcal{C}^{1/2+H}_{[0,1]} L^{m,\infty}} < +\infty$, $[K]_{\mathcal{C}^{1/2+H}_{[0,1]} L^{m}} <+\infty$ and \eqref{eqconv-in-gamma-} to deduce~\ref{item53(1)}.

\smallskip

Proof of \ref{item53(1')}:  We rely again on the decomposition \eqref{eqdeltaAst-decomp-SDE-critic}.

We now use \eqref{eqboundcondF} with $\lambda=0$ and \eqref{eqboundFtilde}. Since $H(\gamma-2) = -H-1/2>-1$, there is
\begin{equation}\label{eqcondExpAsut}
\begin{split}
\|\mathbb{E}^{s} \delta A_{s,u,t} \|_{L^m} 
&\leq C\, \|b^k \|_{\mathcal{B}_p^{\tilde\gamma}} \, [K^{h,k}]_{\mathcal{C}^{\frac{1}{2}+H}_{[S,T]} L^{m,\infty}}  \| K_{s}-K^{h,k}_{s} \|_{L^{m}}\, (t-s)\\
&\quad + C\, \|b^k \|_{\mathcal{B}_p^{\tilde\gamma}} \, \| K^{h,k}_{u}-K^{h,k}_s  - K_{u}+ K_{s} \|_{L^m}\, (t-s)^{\frac{1}{2}}.
\end{split}
\end{equation}
Here we do not expect $K^{h,k}-K$ to be $1/2$-H\"older continuous uniformly in $h$ and $n$, but only $(1/2-\zeta)$-H\"older, so we need to decompose $\| K^{h,k}_{u}-K^{h,k}_s  - K_{u}+ K_{s} \|_{L^m}$ into several terms. First we introduce the pivot term $K^k_{u}-K^k_{s}$ to get
\begin{align*}
\| K^{h,k}_{u}-K^{h,k}_s  - K_{u}+ K_{s} \|_{L^m} \leq [K-K^k]_{\mathcal{C}^{\frac{1}{2}}_{[S,T]}L^m} (u-s)^{\frac{1}{2}} + \| K^{h,k}_{u}-K^{h,k}_s  - K^k_{u}+ K^k_{s} \|_{L^m}.
\end{align*}
Now observe that from \eqref{eqdef-Khn}, \eqref{eqProp51A-critic} and \eqref{eqR},
\begin{align*}
R_{s,u} = K^k_{u}-K^k_{s} - A_{s,u} - \int_{s}^u b^k(X_0+K^{h,k}_{r} + B_{r})\, dr,
\end{align*}
so that
\begin{align*}
K^{h,k}_{u}-K^{h,k}_s  - K^k_{u}+ K^k_{s} = \int_{s}^u b^k(X_0+K^{h,k}_{r_{h}} + B_{r}) - b^k(X_0+K^{h,k}_{r} + B_{r})\, dr -A_{s,u} - R_{s,u} .
\end{align*}
Hence recalling the definition of $E^{2,h,k}$ from \eqref{eqdefE}, we get
\begin{align*}
\| K^{h,k}_{u}-K^{h,k}_s  - K^k_{u}+ K^k_{s} \|_{L^m} \leq \|E^{2,h,k}_{s,u}\|_{L^m} + \|A_{s,u}\|_{L^m} + \|R_{s,u}\|_{L^m}.
\end{align*}
As in \eqref{eqAst-}, we have
\begin{align}\label{eqAProp53}
\|A_{s,u}\|_{L^m} &\leq C \|b^k\|_{\mathcal{B}^{\tilde\gamma}_{p}}\, \|K_{s}-K^{h,k}_{s}\|_{L^m} \, (u-s)^{1+H(\gamma-1)} \nonumber\\
&= C \|b^k\|_{\mathcal{B}^{\tilde\gamma}_{p}}\, \|K_{s}-K^{h,k}_{s}\|_{L^m} \, (u-s)^{\frac{1}{2}}.
\end{align}
Thus we get
\begin{align*}
\| K^{h,k}_{u}-K^{h,k}_s  - K^k_{u}+ K^k_{s} \|_{L^m} &\leq \left( C \|b^k\|_{\mathcal{B}^{\tilde\gamma}_{p}}\, \|K_{s}-K^{h,k}_{s}\|_{L^m} + [E^{2,h,k}]_{\mathcal{C}^{\frac{1}{2}}_{[S,T]} L^{m}} \right) (u-s)^{\frac{1}{2}} \\
&\quad + [R]_{\mathcal{C}^{\tau}_{[S,T]} L^{m}}\, (u-s)^\tau .
\end{align*}
Plugging the previous inequality in \eqref{eqcondExpAsut}, it comes
\begin{align*}
\|\mathbb{E}^{s} \delta A_{s,u,t} \|_{L^m} 
&\leq C\, \|b^k\|_{\mathcal{B}^{\tilde\gamma}_{p}} \Big( (\|b^k\|_{\mathcal{B}^{\tilde\gamma}_{p}}+ [K^{h,k}]_{\mathcal{C}^{\frac{1}{2}+H}_{[S,T]} L^{m,\infty}})\, \|K_{s}-K^{h,k}_{s}\|_{L^m} + [E^{2,h,k}]_{\mathcal{C}^{\frac{1}{2}}_{[S,T]} L^{m}} \Big) (t-s) \\
&\quad+ \|b^k \|_{\mathcal{B}_p^{\tilde\gamma}} \, [R]_{\mathcal{C}^{\tau}_{[S,T]} L^{m}} \, (t-s)^{\frac{1}{2}+\tau}.
\end{align*}
Use finally that $\sup_{(h,k) \in \mathcal{D}} [K^{h,k}]_{\mathcal{C}^{\frac{1}{2}+H}_{[S,T]} L^{m,\infty}}<\infty$, $[E^{2,h,k}]_{\mathcal{C}^{\frac{1}{2}}_{[S,T]} L^{m}} \leq \epsilon(h,k)$ 
 and \eqref{eqconv-in-gamma-} to deduce~\ref{item53(1')}.

\smallskip

Proof of \ref{item53(2)}: We apply Corollary~\ref{cor4inc} with $\beta=\tilde\gamma$, $\lambda=1$, $\lambda_{1}=1$, $\lambda_{2}=\varepsilon$, $\kappa_{1}=K^{h,k}_{s}$, $\kappa_{2}=K_{s}$, $\kappa_{3}=K^{h,k}_{u}$ and $\kappa_{4}=K_{u}$: this yields
\begin{align*}
\| \delta A_{s,u,t} \|_{L^m} 
&\leq C \|b^k\|_{\mathcal{B}_p^{\tilde\gamma}}\, \|\EE^s[|K^{h,k}_{s}-K^{h,k}_{u}|^m]^{1/m}\|_{L^\infty}^{\varepsilon}\, \|K_{s}-K^{h,k}_{s}\|_{L^m}\, (t-u)^{1+H(\gamma-1-\varepsilon)} \\
    &\quad + C\|b^k\|_{\mathcal{B}_p^{\tilde\gamma}}\, \|K_{s}-K^{h,k}_{s} - K_{u} +K^{h,k}_{u}\|_{L^m}\,  (t-u)^{1+H(\gamma-1)} \\
 & \leq C \|b^k\|_{\mathcal{B}_p^{\tilde\gamma}}\, [K^{h,k}]_{\mathcal{C}^{\frac{1}{2}+H}_{[S,T]}L^{m,\infty}}^{\varepsilon}\, \|K_{s}-K^{h,k}_{s}\|_{L^m}\, (t-u)^{\frac{1}{2}+\frac{\varepsilon}{2}} \\
    &\quad + C\|b^k\|_{\mathcal{B}_p^{\tilde\gamma}}\, [K-K^{h,k}]_{\mathcal{C}^{\frac{1}{2}-\zeta}_{[S,T]}L^m}\,  (t-u)^{1-\zeta} .
\end{align*}
Since $\sup_{k} \|b^k\|_{\mathcal{B}_p^{\tilde\gamma}}$, $\sup_{(h,k) \in \mathcal{D}} [K^{h,k}]_{\mathcal{C}^{1/2+H}_{[S,T]}L^{m,\infty}}$ are finite and $\varepsilon < 1-2 \zeta$, we have obtained \ref{item53(2)}.
 
\smallskip

Proof of \ref{item53(3)}: The proof is similar to point~\ref{item51(3)} of Proposition~\ref{propbound-E1-SDE}.
\end{proof}


\subsection{Further regularisation properties of the discrete-time fBm}\label{subsecfurtherdtimefBm}

The aim here is to obtain Corollary~\ref{cornewbound-E2-general}, which is a quadrature bound between a fractional process and its discrete-time approximation based on regularisation properties from the previous sections.

\begin{lemma}\label{lemnewbound-E2-v0}
Let $\varepsilon \in (0,\frac{1}{2})$ and $m \in [2, \infty)$. There exists a constant $C>0$ such that for any $0 \leq S < T \leq 1$, any $\mathbb{R}^d$-valued $\mathcal{F}_S$-measurable random variable $\psi$, any $f \in \mathcal{C}^1_{b}(\mathbb{R}^d, \mathbb{R}^d) $, any $h>0$ and any
 $(s, t)\in \Delta_{S,T}$, we have
\begin{align*}%
\Big\| \int_s^t  f(\psi + B_r) - f(\psi+ B_{r_h}) \diff r  \Big\|_{L^{m}} \nonumber \leq C\, \|f\|_\infty\,  h^{\frac{1}{2}-\varepsilon}  (t-s)^{\frac{1}{2}+\frac{\varepsilon}{2}}  .
\end{align*}
\end{lemma}

\begin{proof}
We will check the conditions in order to apply Lemma~\ref{lemSSL}. For $(s,t) \in \Delta_{S,T}$, let 
\begin{align*}
A_{s,t} = \mathbb{E}^s \int_s^t  f(\psi + B_r) - f(\psi + B_{r_h}) \diff r  ~~\mbox{and}~~
\mathcal{A}_t = \int_S^t f(\psi + B_r) - f(\psi+ B_{r_h}) \diff r .
\end{align*}
Let $u\in [s,t]$ and notice that $\mathbb{E}^s \delta A_{s,u,t}=0$, so \eqref{sts1} holds with $\Gamma_1=0$. We will prove that \eqref{sts2} holds with $\alpha_2=0$ and
$$ \Gamma_2 =  C \|f\|_\infty   h^{\frac{1}{2}-\varepsilon} .$$
\paragraph{The case $t-s\leq 2h$.} In this case we have
\begin{align}\label{eqAst<h}
|A_{s,t}| \leq \|f\|_\infty (t-s) \leq C \|f\|_\infty h^{\frac{1}{2}-\varepsilon} (t-s)^{\frac{1}{2}+\varepsilon} .
\end{align}

\paragraph{The case $t-s>2h$.} Here we split $A_{s,t}$ in two:
\begin{equation*}
A_{s,t} = \mathbb{E}^s \int_s^{s+2h}  f(\psi + B_r) - f(\psi + B_{r_h}) \diff r + \mathbb{E}^s \int_{s+2h}^t  f(\psi + B_r) - f(\psi + B_{r_h}) \diff r .
\end{equation*}
For the first part, we obtain 
\begin{equation*}
\Big|\mathbb{E}^s \int_s^{s+2h}  f(\psi + B_r) - f(\psi + B_{r_h}) \diff r \Big| \leq 4h\, \|f\|_\infty \leq C\, \|f\|_\infty h^{\frac{1}{2}-\varepsilon} (t-s)^{\frac{1}{2}+\varepsilon}.
\end{equation*}
Denote the second part by
\begin{align*}
J \colon= \int_{s+2h}^t \mathbb{E}^s [ f(\psi+ B_{r}) - f(\psi+ B_{r_h}) ] \diff r .
\end{align*}
From Lemma~\ref{lemreg-B}$(i)$ and \eqref{eqLND}, we have
\begin{align}\label{eqboundJ55}
J & =  \int_{s+2h}^t  \Big( G_{C_{1}(r-s)^{2H}}f(\psi+ \EE^s B_{r}) - G_{C_{1}(r_h-s)^{2H}}f(\psi+ \EE^s B_{r_h}) \Big) \diff r \\
& = \int_{s+2h}^t  \Big( G_{C_{1}(r-s)^{2H}}f(\psi+ \EE^s B_{r}) - G_{C_{1}(r_h-s)^{2H}}f(\psi+ \EE^s B_{r}) \Big) \diff r \nonumber \\ 
& \quad + \int_{s+2h}^t  \left( G_{C_{1}(r_h-s)^{2H}}f(\psi+ \EE^s B_{r}) - G_{C_{1}(r_h-s)^{2H}}f(\psi+ \EE^s B_{r_h}) \right)  \diff r \nonumber\\
& =\colon J_1 + J_2 .\nonumber
\end{align}
For $J_1$, we apply \cite[Proposition 3.7 (ii)]{butkovsky2021approximation} with $\beta=0$, $\delta=1$, $\alpha=0$  to get
\begin{align*}
\| J_1 \|_{L^m} &  \leq C\, \| f \|_\infty \int_{s+2h}^t \big( (r-s)^{2H} - (r_h-s)^{2H} \big) (r_h-s)^{-2H} \diff r .
\end{align*}
Now applying the inequalities $(r-s)^{2H} - (r_h-s)^{2H} \leq C (r-r_{h}) (r_h-s)^{2H-1}$ and $2 (r_h-s) \ge (r-s)$, it comes
\begin{align*}
\| J_1 \|_{L^m} &  \leq C\,  \| f \|_\infty \int_{s+2h}^t  (r-r_h) (r-s)^{2H-1} (r-s)^{-2H} \diff r \\
& \leq C\,  \| f \|_\infty h \int_{s+2h}^t (r-s)^{-1} \diff r \\
& \leq C\, \| f \|_\infty h  \left( |\log(2h)| + |\log(t-s)| \right) .
\end{align*}
Use again that $2h < t-s$ to get
\begin{align*}
\| J_1 \|_{L^m} &  \leq C\, \| f \|_\infty h^{\frac{1}{2}-\varepsilon} (t-s)^{\frac{1}{2}+\frac{\varepsilon}{2}} .
\end{align*}
As for $J_2$, we have
\begin{align*}
\|J_{2}\|_{L^m} \leq \int_{s+2h}^t  \|G_{C_{1}(r_h-s)^{2H}} f \|_{\mathcal{C}^1} \, \| \EE^s B_{r} - \EE^s B_{r_h}  \|_{L^m}  \diff r.
\end{align*}
In view of \cite[Proposition 3.7 (i)]{butkovsky2021approximation} applied with $\beta=1$, $\alpha=0$, \cite[Proposition 3.6 (v)]{butkovsky2021approximation} and using again that $2 (r_h-s) \ge (r-s)$, we get
\begin{align*}
\| J_2 \|_{L^m} & \leq C\, \| f \|_\infty \int_{s+2h}^t   \| \EE^s B_{r} - \EE^s B_{r_h} \|_{L^m} (r_h-s)^{-H} \diff r \\
& \leq C\, \| f \|_\infty \int_{s+2h}^t (r-r_h) (r-s)^{H-1} (r-s)^{-H}  \diff r \\
& \leq C\, \| f \|_\infty\,  h\,  \big( |\log(2h)| + |\log(t-s)| \big) \\
& \leq C\, \| f \|_\infty\,  h^{\frac{1}{2}-\varepsilon}\,  (t-s)^{\frac{1}{2}+\frac{\varepsilon}{2}} .
\end{align*}
Combining the bounds on $J_1$ and $J_2$, we deduce that
\begin{align*}
\| J \|_{L^m} \leq C\, \| f \|_\infty\,  h^{\frac{1}{2}-\varepsilon} \, (t-s)^{\frac{1}{2}+\frac{\varepsilon}{2}} .
\end{align*}
Hence for all $t-s>2h$,
\begin{align}\label{eqAst>h}
\| A_{s,t} \|_{L^m}  \leq C\, \| f \|_\infty\,  h^{\frac{1}{2}-\varepsilon} \, (t-s)^{\frac{1}{2}+\frac{\varepsilon}{2}} .
\end{align}

\medskip

Overall, combining \eqref{eqAst<h} and \eqref{eqAst>h}, we obtain that for all $s \leq t$,
\begin{align*}
\| A_{s,t}\|_{L^{m}} & \leq C  \|f\|_{\infty} h^{\frac{1}{2}-\varepsilon} (t-s)^{\frac{1}{2}+\frac{\varepsilon}{2}} .
\end{align*}
Thus for any $u\in [s,t]$, $\| \delta A_{s,u,t}\|_{L^{m}}  \leq \| A_{s,t}\|_{L^{m}}+\| A_{s,u}\|_{L^{m}}+\| A_{u,t}\|_{L^{m}}\\
\leq C  \|f\|_{\infty} h^{\frac{1}{2}-\varepsilon} (t-s)^{\frac{1}{2}+\frac{\varepsilon}{2}}$.
The power in $(t-s)$ is strictly larger than $1/2$, so \eqref{sts2} holds.

\paragraph{Convergence in probability.}
Finally, for a sequence $(\Pi_n)_{n \in 
\mathbb{N}}$ of partitions of $[S,t]$ with $\Pi_n=\{t_i^n\}_{i=1}^{N_n}$ and mesh size 
converging to zero, we have
\begin{align*}
\Big\| \mathcal{A}_t -  \sum_{i=1}^{N_{n}-1} A_{t_i^n,t_{i+1}^n} \Big\|_{L^1} & \leq \sum_{i=1}^{N_{n}-1} \int_{t_i^n}^{t_{i+1}^n} \mathbb{E}   \Big| f(\psi +B_r) - f(\psi+ B_{r_h})  -\mathbb{E}^{t_i^n}[ f(\psi +B_{r}) + f(\psi- B_{r_h})] \Big|  \diff r\\
& \leq \sum_{i=1}^{N_{n}-1} \int_{t_i^n}^{t_{i+1}^n} \mathbb{E} \Big| f(\psi +B_r) -\mathbb{E}^{t_i^n} f(\psi +B_{r}) \Big| \diff r \\
 &\quad + \sum_{i=1}^{N_{n}-1} \int_{t_i^n}^{t_{i+1}^n} \mathbb{E}\Big| f(\psi +B_{r_h}) -\mathbb{E}^{t_i^n} f(\psi +B_{r_h}) \Big| \diff r  \\
& =\colon I_1 + I_2 .
\end{align*}
In view of Lemma \ref{lemreg-B}$(iii)$, it comes that $I_1  \leq \sum_{i=1}^{N_{n}-1} \int_{t_i^n}^{t_{i+1}^n} \| f \|_{\mathcal{C}^1} (r-t_i^n)^H \diff r  \leq \| f \|_{\mathcal{C}^1} |\Pi_n|^H \, (t-S)$.
As for $I_{2}$, we refer to the paragraph \textbf{Convergence in probability} of the proof of Lemma \ref{lembound-Khn}, where we obtain 
\begin{align*}
I_2 & \leq \sum_{i=1}^{N_{n}-1} \int_{t_i^n}^{t_{i+1}^n} \| f \|_{\mathcal{C}^1} |\Pi_n|^H \diff r ,
\end{align*}
and therefore $\sum_{i=1}^{N_{n}-1} A_{t_{i}^n, t_{i+1}^n}$ converges in probability to $\mathcal{A}_{t}$ as $n \to +\infty$.
We can therefore apply Lemma~\ref{lemSSL} with $\varepsilon_1>0$ and $\varepsilon_2 = \varepsilon/2$ to conclude that $\| \mathcal{A}_t - \mathcal{A}_s \|_{L^m}  \leq \| \mathcal{A}_t - \mathcal{A}_s - A_{s,t} \|_{L^m} + \| A_{s,t} \|_{L^m}  \\
 \leq C\, \|f\|_\infty \, h^{\frac{1}{2}-\varepsilon} \, (t-s)^{\frac{1}{2}+\frac{\varepsilon}{2} } $.
\end{proof}

\begin{proposition}\label{propnewbound-E2}
Let $\varepsilon \in (0,\frac{1}{2})$ and $m \in [2, \infty)$.
There exists a constant $C>0$ such that for any $\mathbb{R}^d$-valued stochastic process $(\psi_{t})_{t\in [0,1]}$ adapted to $\mathbb{F}$, any $f \in \mathcal{C}^1_{b}(\mathbb{R}^d, \mathbb{R}^d) $, any $h \in (0,1)$ and any $(s, t)\in \Delta_{0,1}$, we have
\begin{align}\label{eqssl-on-on-SDE}
\Big\| \int_s^t f(\psi_r + B_r) - f(\psi_{r} + B_{r_h}) \diff r  \Big\|_{L^{m}} & \leq C \Big( \|f\|_\infty h^{\frac{1}{2}-\varepsilon} (t-s)^{\frac{1}{2}+\frac{\varepsilon}{2}}   + \|f \|_{\mathcal{C}^1}  [\psi]_{\mathcal{C}^{1}_{[0,1]} L^{\infty}}   h^{1-\varepsilon} (t-s)^{1+\frac{\varepsilon}{2}} \Big)  .
\end{align}
\end{proposition}

\begin{proof}
 Assume that $[\psi]_{\mathcal{C}^{1}_{[0,1]}L^\infty}<\infty$, otherwise 
\eqref{eqssl-on-on-SDE} trivially holds. 
We will check the conditions in order to apply Lemma~\ref{lemSSL} (with $q=m$). Let  and $0 \leq S < T \leq 1$. For any $(s,t) \in \Delta_{S,T}$, define 
\begin{align}\label{eqProp417}
A_{s,t} = \int_s^t f(\psi_s + B_r) - f(\psi_s+ B_{r_h}) \diff r   ~~\mbox{and}~~
\mathcal{A}_t = \int_S^t f(\psi_r+ B_r) - f(\psi_r+ B_{r_h}) \diff r  .
\end{align}
To show that \eqref{sts1} and \eqref{sts2} hold true with $\varepsilon_1= \varepsilon_2=\varepsilon/2 >0$ and $\alpha_1=\alpha_2=0$, we prove that there exists a constant $C>0$ independent of $s,t,S$ and $T$ such that for $u = (s+t)/2$,
\begin{enumerate}[label=(\roman*)]
\item \label{item54(1)} 
$\|\EE^{{s}} [\delta A_{s,u,t}]\|_{L^m} \leq C\, \|f \|_{\mathcal{C}^1}\, [\psi]_{\mathcal{C}^{1}_{[S,T]} L^{\infty}}\, h^{1-\varepsilon}\, ({t}-{s})^{1 + \frac{\varepsilon}{2}}$\text{;}

\item \label{item54(2)} 
$\| \delta A_{s,u,t}\|_{L^m} \leq C\, \|f\|_{\infty}\, h^{\frac{1}{2}-\varepsilon}\, ({t}-{s})^{\frac{1}{2}+\frac{\varepsilon}{2}}$\text{;}

\item \label{item54(3)} If \ref{item54(1)} and \ref{item54(2)} are satisfied, \eqref{sts3} gives the 
convergence in probability of $\sum_{i=1}^{N_n-1} A_{t^n_i,t^n_{i+1}}$ along any sequence of 
partitions $\Pi_n=\{t_i^n\}_{i=1}^{N_n}$ of $[S,t]$ with mesh converging to $0$. We will prove 
that the limit is the process $\mathcal{A}$ given in \eqref{eqProp417}.
\end{enumerate}

Assume for now that \ref{item54(1)}, \ref{item54(2)} and \ref{item54(3)} hold. Applying Lemma~\ref{lemSSL}, we obtain that
\begin{align*}
   \Big\| \int_s^t f(\psi_r + B_r) - f(\psi_{r} + B_{r_h}) \diff r  \Big\|_{L^{m}} 
   & \leq C \Big( \|f\|_\infty h^{\frac{1}{2}-\varepsilon} (t-s)^{\frac{1}{2}+\frac{\varepsilon}{2}}   + \|f \|_{\mathcal{C}^1}  [\psi]_{\mathcal{C}^{1}_{[S,T]} L^{\infty}}   h^{1-\varepsilon} (t-s)^{1+\frac{\varepsilon}{2}}   \Big) \\
    & \quad +\| A_{s,t}\|_{L^m}.
\end{align*}
We will see in the proof of \ref{item54(2)} that $\|A_{s,t}\|_{L^m} \leq C\, \|f\|_{\infty} h^{\frac{1}{2}-\varepsilon}\, ({t}-{s})^{\frac{1}{2}+\frac{\varepsilon}{2}}$. Then choosing $(s,t)=(S,T)$, we get \eqref{eqssl-on-on-SDE}, using that $[\psi]_{\mathcal{C}^{1}_{[S,T]} L^{\infty}}   \leq [\psi]_{\mathcal{C}^{1}_{[0,1]} L^{\infty}}  $.

We now check that the conditions \ref{item54(1)}, \ref{item54(2)} and \ref{item54(3)} actually hold.

\smallskip

Proof of \ref{item54(1)}: We have
\begin{align*}
\mathbb{E}^s \delta A_{s,u,t} & = \mathbb{E}^s \int_u^t f(\psi_s+ B_r) - f(\psi_s+ B_{r_h}) - f(\psi_u + B_r) + f(\psi_u+ B_{r_h}) \diff r  .
\end{align*}

\paragraph{The case $t-u \leq 2h$.}
In this case, using the Lipschitz norm of $f$, we have
\begin{align*}
| \mathbb{E}^s \delta A_{s,u,t} | & \leq 2 \|f\|_{\mathcal{C}^1} \int_u^t  |\psi_s-\psi_u|  \diff r. 
\end{align*}
Therefore using the inequality $(t-u) (u-s) \leq C\, h^{1-\varepsilon}\, (t-u)^{\varepsilon} (u-s)  \leq C\, (t-s)^{1+\varepsilon} h^{1-\varepsilon}$,
\begin{align*}
\| \mathbb{E}^s \delta A_{s,u,t} \|_{L^m} & \leq \|f\|_{\mathcal{C}^1}  [ \psi ]_{\mathcal{C}^1_{[S,T]} L^m}  (t-u) (u-s)   \leq C \|f\|_{\mathcal{C}^1}  [ \psi ]_{\mathcal{C}^1_{[S,T]} L^\infty}  (t-s)^{1+\varepsilon} h^{1-\varepsilon} .
\end{align*}

\paragraph{The case $t-u > 2h$.} We split the integral between $u$ and $u+2h$ and then between $u+2h$ and $t$ as follows:
\begin{align*}
\EE^s \delta A_{s,u,t} & = \int_u^{u+2h} \EE^s [ f(\psi_s+ B_r) - f(\psi_s+ B_{r_h}) - f(\psi_u + B_r) + f(\psi_u+ B_{r_h}) ] \diff r  \\ 
& \quad + \EE^s \int_{u+2h}^t \mathbb{E}^u [ f(\psi_s+ B_r) - f(\psi_s+ B_{r_h}) - f(\psi_u + B_r) + f(\psi_u+ B_{r_h}) ] \diff r 
\\
& =\colon J_1 + J_2  ,
\end{align*}
using the tower property of conditional expectation for $J_{2}$. 
For $J_1$, we obtain from the case $t-u\leq 2h$ that
\begin{align}\label{eqJ1Prop54}
\| J_{1}\|_{L^m} = \|\mathbb{E}^s \delta A_{s,u,u+2h}\|_{L^m}& \leq  C  \|f\|_{\mathcal{C}^1} [ \psi ]_{\mathcal{C}^1_{[S,T]} L^\infty} (t-s)^{1+\varepsilon} h^{1-\varepsilon} .
\end{align}
As for $J_2$, we use Lemma~\ref{lemreg-B}$(i)$ and \eqref{eqLND} to write
\begin{align*}
J_2 
& = \mathbb{E}^s \int_{u+2h}^t \big( G_{C_{1}(r-u)^{2H}} - G_{C_{1}(r_h-u)^{2H}} \big) \big( f(\psi_s+ \EE^u B_r) - f(\psi_u +\EE^u B_r) \big) \diff r \\ 
&\quad + \mathbb{E}^s \int_{u+2h}^t G_{C_{1}(r_h-u)^{2H}} \Big( f(\psi_s+ \EE^u B_r) - f(\psi_s+ \EE^u B_{r_h}) - f(\psi_u + \EE^u B_r) + f(\psi_u+ \EE^u B_{r_h}) \Big) \diff r \\
& =\colon J_{21} + J_{22} .
\end{align*}
For $J_{21}$, we apply \cite[Proposition 3.7 (ii)]{butkovsky2021approximation} with $\beta=0$, $\delta=1$, $\alpha=0$ and $f\equiv f(\psi_s+ \cdot) - f(\psi_u +\cdot)$ to get
\begin{align*}
\| J_{21} \|_{L^m} & \leq C \| \EE^s \| f(\psi_s + \cdot) - f(\psi_u+\cdot) \|_{\infty} \|_{L^m} \int_{u+2h}^t \big( (r-u)^{2H}-(r_h-u)^{2H} \big) (r_h-u)^{-2H} \diff r .
\end{align*}
Now using that $ \| \EE^s \| f(\psi_s + \cdot) - f(\psi_u+\cdot) \|_{\infty} \|_{L^m}
 \leq \| f \|_{\mathcal{C}^1}  \| \EE^s |\psi_s - \psi_u| \|_{L^m}
 \leq \| f \|_{\mathcal{C}^1}  \| \psi_s - \psi_u \|_{L^\infty}$ and 
applying the inequalities $(r-u)^{2H} - (r_h-u)^{2H} \leq C (r-r_{h}) (r_h-u)^{2H-1}$ and $2 (r_h-u) \ge (r-u)$, it comes
\begin{align*}
\| J_{21} \|_{L^m} & \leq C\, \| f \|_{\mathcal{C}^1} [ \psi ]_{\mathcal{C}^1_{[S,T]} L^\infty} |u-s| \int_{u+2h}^t (r-r_h) (r-u)^{2H-1} (r-u)^{-2H} \diff r \\
& \leq C\, \| f \|_{\mathcal{C}^1} [ \psi ]_{\mathcal{C}^1_{[S,T]} L^\infty}\, h \, (t-s)\,  \big( |\log(2h)| + |\log(t-u)| \big) .
\end{align*}
Since $t-u = (t-s)/2 > 2h$, one has
\begin{align}\label{eqJ21Prop54}
\| J_{21} \|_{L^m}  & \leq C\, \| f \|_{\mathcal{C}^1} [ \psi ]_{\mathcal{C}^1_{[S,T]} L^\infty} \, h^{1-\varepsilon} (t-s)^{1+\frac{\varepsilon}{2}} .
\end{align}
As for $J_{22}$, observe that
\begin{align*}
&\Big|G_{C_{1}(r_h-u)^{2H}} \Big( f(\psi_s+ \EE^u B_r) - f(\psi_s+ \EE^u B_{r_h}) - f(\psi_u + \EE^u B_r) + f(\psi_u+ \EE^u B_{r_h}) \Big)\Big| \\
&\quad \leq \|G_{C_{1}(r_h-u)^{2H}} f(\psi_{s}+\cdot) - G_{C_{1}(r_h-u)^{2H}}f(\psi_{u}+\cdot)\|_{\mathcal{C}^1} \, | \EE^u B_r - \EE^u B_{r_h} | \\
&\quad \leq C\, \| f(\psi_s + \cdot) - f(\psi_u+\cdot) \|_{\infty}\, (r_{h}-u)^{-H} \, | \EE^u B_r - \EE^u B_{r_h} |\\
&\quad \leq C\, \| f \|_{\mathcal{C}^1} \, (r_{h}-u)^{-H} \, | \psi_s - \psi_u |\, | \EE^u B_r - \EE^u B_{r_h} |,
\end{align*}
where we used \cite[Proposition 3.7 (i)]{butkovsky2021approximation} with $\beta=1$ and $\alpha=0$ in the penultimate inequality. Now in view of the previous inequality, using consecutively Jensen's inequality, \cite[Proposition 3.6 (v)]{butkovsky2021approximation}, that $2(r_{h}-u)\geq r-u$ and that $t-u = (t-s)/2 > 2h$, it comes
\begin{align}\label{eqJ22Prop54}
\| J_{22} \|_{L^m} & \leq C\, \| f \|_{\mathcal{C}^1} \int_{u+2h}^t \| \EE^s[   | \psi_s - \psi_u |\, | \EE^u (B_r - B_{r_h})|] \|_{L^m} (r_{h}-u)^{-H} \diff r \nonumber\\
& \leq C\, \| f \|_{\mathcal{C}^1}\, \|\psi_{s}-\psi_{u}\|_{L^\infty} \int_{u+2h}^t \|  \EE^u (B_r - B_{r_h}) \|_{L^m} (r_{h}-u)^{-H} \diff r \nonumber\\
& \leq C\, \| f \|_{\mathcal{C}^1}  [ \psi ]_{\mathcal{C}^1_{[S,T]} L^\infty} \, (u-s) \int_{u+2h}^t (r-r_h) (r-u)^{H-1} (r_{h}-u)^{-H} \diff r \nonumber \\
& \leq C\, \| f \|_{\mathcal{C}^1}  [ \psi ]_{\mathcal{C}^1_{[S,T]} L^\infty} \, h\, (t-s) \, \big( | \log(2h) | + | \log(t-u)| \big)\nonumber \\
& \leq C\, \| f \|_{\mathcal{C}^1} [ \psi ]_{\mathcal{C}^1_{[S,T]} L^\infty} \, h^{1-\varepsilon}\, (t-s)^{1+\frac{\varepsilon}{2}} .
\end{align}
In view of the inequalities \eqref{eqJ1Prop54}, \eqref{eqJ21Prop54} and \eqref{eqJ22Prop54}, we have finally
\begin{align*}%
\| \mathbb{E}^s \delta A_{s,u,t} \|_{L^{m}} & \leq C\, \| f \|_{\mathcal{C}^1} [ \psi ]_{\mathcal{C}^1_{[S,T]} L^\infty} h^{1-\varepsilon} (t-s)^{1+\frac{\varepsilon}{2}} .
\end{align*}

\medskip

Proof of \ref{item54(2)}: 
We write
\begin{equation*}
\| \delta A_{{s},u,{t}}\|_{L^m} \leq \| A_{{s},{t}}\|_{L^m} + \| A_{{s},u}\|_{L^m} + \| A_{u,{t}}\|_{L^m}
\end{equation*}
and we apply Lemma~\ref{lemnewbound-E2-v0} for each term in the right-hand side of the previous inequality, respectively for $\psi = \psi_{s}$, $\psi_{s}$ again and $\psi_{u}$. We thus have $\| A_{s,t} \|_{L^{m}} \leq  C \, \|f\|_\infty\, h^{\frac{1}{2}-\varepsilon} \, (t-s)^{\frac{1}{2}+\frac{\varepsilon}{2} }$.
Combining similar inequalities on $A_{{s},u}$ and $A_{u,{t}}$ yields $\| \delta A_{s,u,t}\|_{L^m} \leq C\, \|f\|_{\infty}\, h^{\frac{1}{2}-\varepsilon}\, ({t}-{s})^{\frac{1}{2}+\frac{\varepsilon}{2}}$.

\medskip
Proof of \ref{item54(3)}: The proof is similar to point~\ref{item51(3)} of Proposition~\ref{propbound-E1-SDE}.
\end{proof}

\begin{corollary}\label{cornewbound-E2-general}
Let $\varepsilon \in (0,\frac{1}{2})$ and $m \in [2, \infty)$. There exists a constant $C>0$ such that for any $\mathbb{R}^d$-valued $\mathbb{F}$-adapted process $(\psi_{t})_{t\in [0,1]}$, any $f \in \mathcal{C}^1_{b}(\mathbb{R}^d, \mathbb{R}^d) $, any $h \in (0,1)$, and any $(s, t)\in \Delta_{0,1}$, 
\begin{align*}%
\Big\| \int_s^t  f(\psi_r + B_r) - f(\psi_{r_h} + B_{r_h}) \diff r  \Big\|_{L^{m}} & \leq C \left( \|f\|_\infty \, h^{\frac{1}{2}-\varepsilon} (t-s)^{\frac{1}{2}+\frac{\varepsilon}{2}}  +   \|f \|_{\mathcal{C}^1}  [\psi]_{\mathcal{C}^{1}_{[0,1]} L^\infty}  \, h^{1-\varepsilon} (t-s) \right) .
\end{align*}
\end{corollary}

\begin{proof}
Introducing the pivot term $f(\psi_{r} + B_{r_h})$, we have
\begin{align*}
\Big\| &\int_s^t f(\psi_r + B_r) - f(\psi_{r_h} + B_{r_h}) \diff r  \Big\|_{L^{m}}  \\
&\leq  \Big\| \int_s^t f(\psi_{r} + B_{r_h}) - f(\psi_{r_h} + B_{r_h}) \diff r  \Big\|_{L^{m}} + \Big\| \int_s^t  f(\psi_r + B_r) - f(\psi_{r} + B_{r_h})\diff r  \Big\|_{L^{m}}\\
&=\colon J_1 + J_2 .
\end{align*}
We first bound $J_1$ using the $\mathcal{C}^1$ norm of $f$: $J_1 \leq \|f\|_{\mathcal{C}^1} \int_s^t \| \psi_r - \psi_{r_h} \|_{L^m} \diff r \leq \|f\|_{\mathcal{C}^1} [ \psi ]_{\mathcal{C}^1_{[0,1]} L^\infty}\, h\, (t-s)$.
Then $J_2$ is bounded by Proposition \ref{propnewbound-E2}. Combining the two bounds, we get the desired result.
\end{proof}

\subsection{H\"older regularity of $E^{2,h,k}$}\label{subsecreg-E2hn}
The result below is a quadrature estimate that can be compared to \cite[Lemma 4.2]{butkovsky2021approximation}. We avoid here the exponential dependence on the sup-norm of the drift coming from the Girsanov approach, which is crucial to get a good rate of convergence of the scheme when the drift is unbounded or distributional. But we have instead a linear dependence on the $\mathcal{C}^1$-norm of the drift.

\begin{corollary}\label{cornewbound-E2}
Recall that the process $K^{h,k}$ was defined in \eqref{eqdef-Khn}. 
Let $\varepsilon \in (0,\frac{1}{2})$ and $m \in [2, \infty)$. 
There exists a constant $C>0$ such that for any $(s , t) \in \Delta_{0,1}$, any $h\in (0,1)$ and any $k\in \N$, we have
\begin{align*}%
 \Big\| \int_s^t  b^k(X_0+K^{h,k}_r + B_r) - b^k(X_0 +K_{r_h}^{h,k}+ B_{r_h})  \diff r  \Big\|_{L^{m}} \leq  C \Big( \|b^k\|_\infty h^{\frac{1}{2}-\varepsilon} +  \|b^k\|_{\mathcal{C}^1} \|b^k\|_\infty  h^{1-\varepsilon} \Big) (t-s)^{\frac{1}{2}} .
\end{align*}
\end{corollary}

\begin{proof}
Define the process $\psi_t = X_0 + K^{h,k}_t,\, t\in [0,1]$. 
Since $\psi$ is $\mathbb{F}$-adapted and $b^k \in \mathcal{C}^1_{b}(\mathbb{R}^d, \R^d)$, we apply Corollary \ref{cornewbound-E2-general} to get
\begin{align*}
& \Big\| \int_s^t  b^k(X_0+K^{h,k}_r + B_r) - b^k(X_0+K_{r_h}^{h,k}+ B_{r_h})  \diff r  \Big\|_{L^{m}} \\
&\quad \leq  C \Big( \|b^k\|_\infty h^{\frac{1}{2}-\varepsilon} (t-s)^{\frac{\varepsilon}{2}} +  \|b^k\|_{\mathcal{C}^1} [\psi]_{\mathcal{C}^1_{[0,1]}L^\infty}  h^{1-\varepsilon} (t-s)^{\frac{1}{2}}\Big) (t-s)^{\frac{1}{2}} \\
&\quad \leq  C \Big( \|b^k\|_\infty h^{\frac{1}{2}-\varepsilon}+  \|b^k\|_{\mathcal{C}^1} [\psi]_{\mathcal{C}^1_{[0,1]}L^\infty}  h^{1-\varepsilon}\Big) (t-s)^{\frac{1}{2}} .
\end{align*}
It remains to prove an upper bound on $[\psi]_{\mathcal{C}^1_{[0,1]} L^\infty}$. For $0 \leq u \leq v \leq 1$, we have
\begin{align*}
|\psi_v -\psi_u| = |\int_u^v b^k(X^{h,k}_{r_h}) \diff r | \leq |v-u| \|b^k\|_\infty  .
\end{align*}
Hence $[\psi]_{\mathcal{C}^1_{[0,1]} L^\infty} \leq \| b^k \|_\infty$.
\end{proof}

\section{Examples and simulations}\label{secsimulations}

In this section, we discuss examples of SDEs of the form \eqref{eqSDE} that can be treated by Theorem~\ref{thmmain-SDE}. For the rest of this section, we denote by $\kappa$ either the identity function or a smooth cut-off, i.e $\kappa$ is smooth on some compact interval and vanishes otherwise.

\subsection{Skew fractional Brownian motion}

The skew Brownian motion has been introduced in \cite{harrison1981skew} as the scaling limit of a random walk with biased probability transition upon touching $0$. The study of this process reveals several equivalent constructions \cite{Lejay}: one the one hand it can be interpreted as a one-dimensional stochastic process that behaves like a Brownian motion with a certain diffusion coefficient above the $x$-axis, and with another diffusion coefficient below the $x$-axis; on the other hand it is the solution of an SDE which involves its local time \cite{LeGall}. The latter yields the equation
$dX_{t} = \alpha\, dL^X_{t} + dW_{t}$, for $\alpha\in [-1,1]$, where $L^X$ is the local time at $0$ of the solution. Formally we can write $dL^X_{t} =  \delta_{0}(X_{t})\, dt$. More generally in $\R^d$, although this is not the only possible approach (see e.g. \cite{Banos,Sole} for alternative definitions), we call skew fractional Brownian motion the solution to \eqref{eqSDE} when the drift is $\alpha\delta_{0}$, $\alpha\in \R^d$, that is
\begin{align}\label{eqskewfbm}
\diff X_t = \alpha \delta_0(X_t) \diff t + \diff B_t .
\end{align}
Since $\delta_{0} \in \mathcal{B}_p^{-d+d/p} \hookrightarrow \mathcal{B}_{\infty}^{-d}$ for any $p<\infty$, Corollary~\ref{corSP} gives strong existence and uniqueness under \eqref{eqcond-gamma-p-H}, which reads here $H \leq \frac{1}{2(d+1)}$. For the same values of $H$, the tamed Euler scheme converges by Corollary~\ref{corbn-choice}. 
\begin{remark}
Since $H \leq 1/(2(d+1)) < 1/d$, we know from \cite[Theorem 7.1]{Xiao} that the fBm visits any state $x$ (and in particular $0$) infinitely may times with positive probability. So the Equation~\eqref{eqskewfbm} is not simply reduced to $X=B$.\\
Instead of putting a Dirac measure in dimension $d >1$, one can also define a skew fBm on some set $S$ of dimension $d-1$ by considering a measure $\mu$ supported on $S$ (for example a Hausdorff measure). As before, we know from \cite[Theorem 7.1]{Xiao} that for $H < 1/d$, the fBm visits $S$ infinitely many times with positive probability. So the equation $d X_t = \mu(X_t) d t + d B_t$ is not reduced to $X=B$. Since signed measures also belong to $\mathcal{B}_p^{-d/d/p} \in \mathcal{B}_\infty^{-d}$ \cite[Proposition 2.39]{bahouri2011fourier}, we have again strong existence and uniqueness for $H \leq 1/(2(d+1))$.
\end{remark}
As an alternative construction of the skew fBm, we also propose to replace the local time by its approximation $b(x) = \frac{\alpha}{2\varepsilon} \mathds{1}_{(-\varepsilon, \varepsilon)}(x)$, $\varepsilon>0$.
Now we have $b$ bounded and so $b \in \mathcal{B}_\infty^0$, therefore one can take $H < 1/2$ and consider the SDE
\begin{align*}%
\diff X_t = \frac{\alpha}{(2\varepsilon)^d} \mathds{1}_{(-\varepsilon, \varepsilon)^d}(X) \diff t + \diff B_t  .
\end{align*}

In the Markovian case and dimension $d=1$, the skew Brownian motion is reflected on the $x$-axis when $\alpha=\pm 1$. Unlike the skew Brownian motion, the skew fBm (for $H\neq 1/2$) is not reflected for any value of $\alpha$, since $X-B$ is more regular than $B$ (see Theorem~\ref{thWP}). To construct reflected processes, a classical approach is to proceed by penalization, see e.g. \cite{LionsSznitman} in the Brownian case, and \cite{RTT} for rough differential equations. This consists in choosing a drift of the form $b_{\varepsilon}(x)=\frac{(x)_{-}}{\varepsilon}$ and letting $\varepsilon$ tend to $0$. Note that this approach also works for stochastic partial differential equations (SPDEs), see for instance  \cite{nualartpardoux,zambotti2003integration,Ludovic}. 
If we consider more specifically the stochastic heat equation, the solution in time observed at a fixed point in space behaves qualitatively like a fractional SDE with Hurst parameter $H=1/4$. Hence it is interesting to consider the following one-dimensional SDE: 
\begin{align}\label{eqpenalised}
\diff X_t^\varepsilon =  \frac{(X_t^\varepsilon)_{-}}{\varepsilon} \kappa(X_t^\varepsilon) \diff t + \diff B_t .
\end{align}
In \cite{RTT}, $\kappa$ was essentially the identity mapping and the distance between $X^\varepsilon$ and $X^0$ was quantified, with $X^0$ a reflected process. But then the drift is not in some Besov space. So in order to approximate \eqref{eqpenalised} numerically, we assume that $\kappa$ is a smooth cut-off to ensure that the drift is in some Besov space (e.g. $\mathcal{B}^1_\infty$), however it is no longer clear that $X^\varepsilon$ converges to a reflected process. We leave the question of numerical approximation of reflected fractional processes for future research.

\subsection{Applications in finance}

Some models of mathematical finance involve irregular drifts. 

First, consider a dividend paying firm, whose capital evolution can be modelled by the following one-dimensional SDE:
\begin{align*}%
\diff X_{t} = (r - \mathds{1}_{X_{t}\leq q}) \diff t + \sigma  \diff B_{t} ,
\end{align*}
swith an interest rate $r$, the volatility of the market $\sigma$ and some threshold $q$, see e.g \cite{MR1466852}. 
\begin{remark}
An extension of the previous SDE to dimension $d$ can be done by considering a threshold of the form $\mathds{1}_{x \in D}$ where $D$ is some domain in $\R^d$ or $\mathds{1}_{f(x) \leq q}$ where $f \colon\R^d \mapsto \R$.
\end{remark}
Numerical methods for bounded drifts with Brownian noise exist in the literature, see e.g. \cite{dareiotis2021quantifying,JourdainMenozzi}. When $B$ is a fractional Brownian motion with $H<1/2$, \cite{butkovsky2021approximation} provides a rate of convergence for the strong error (and Theorem~\ref{thmmain-SDE} provides the same rate of convergence).

\smallskip

Then, we propose a class of models which can be related heuristically to the rough Heston model introduced in \cite{ElEuchRosenbaum}. Recently, it was observed empirically that the volatility in some high-frequency financial markets has a very rough behaviour, in the sense that its trajectories have a very small H\"older exponent, close to $0.1$. 
Formally, the volatility component in the rough Heston model is described by a square root diffusion coefficient and a very rough driving noise. It would read
\begin{equation}\label{eqroughHeston}
dV_{t} = \kappa(V_{t}) \, dt + \sqrt{V_{t}}\, dB_{t},
\end{equation}
if we could make sense of this equation, the difficulty being both to define a stochastic integral when $H$ is small, and to ensure the positivity of the solution. Note that it is possible to define properly a rough Heston model, by means of Volterra equations, see \cite{ElEuchRosenbaum}. While there are some quantitative numerical approximation results for rough models (e.g. for the rough Bergomi model \cite{Gassiat,FrizEtAl}), the Euler scheme for the rough Heston model is only known to converge without rate~\cite{RTY}. 

However, we keep discussing \eqref{eqroughHeston} at a formal level, and consider the
Lamperti transform $L(x) = \sqrt{x}$. If a first order chain rule could hold for the solution of \eqref{eqroughHeston}, then as long as $V$ stays nonnegative, it would come that
\begin{align*}
L(V_{t}) = L(V_{0}) + \int_{0}^t \frac{\kappa(V_{s})}{2\sqrt{V_{s}}} \, ds + \frac{1}{2}B_{t} ,
\end{align*}
which for $\widetilde{V}_{t} \colon= L(V_{t}) = \sqrt{V_{t}}$ also reads
\begin{align}\label{eqVtilde}
\widetilde{V}_{t} = \widetilde{V}_{0} + \frac{1}{2} \int_{0}^t \frac{1}{\widetilde{V}_{s}} \kappa( \widetilde{V}_{s}^2) \, ds + \frac{1}{2} B_{t} .
\end{align}
Now we can make sense of the Equation \eqref{eqVtilde} for a slightly less singular drift, namely with drift $b(x) = \frac{\kappa(x^2)}{2 |x|^{1-\varepsilon}}$, since for a bump function $\kappa$ and for small $\varepsilon>0$, $b\in \mathcal{B}^{0}_{1} \in \mathcal{B}^{-1+d/p}_{p}$ (see \cite[Prop. 2.21]{bahouri2011fourier}). Hence Theorem~\ref{thmmain-SDE} can be applied whenever $H \leq 1/4$ and in view of Corollary~\ref{corbn-choice}, this yields a strong error of order $(1/4)^-$ (excluding the limit case).

\subsection{Fractional Bessel processes (in dimension $1$)}
 
Bessel processes \cite[Chapter XI]{revuz2013continuous} play an important role in probability theory and financial mathematics. 
As a generalization and motivated by the discussion in the previous subsection, we consider solutions to the following one-dimensional SDE:
\begin{align}\label{eqbessel}
\diff X_t = \frac{\kappa(X_{t})}{|X_t|^\alpha} \diff t + \diff B_t ,
\end{align}
for some $\alpha > 0$ and $H \in (0,1)$. When $H=1/2$, $\alpha=1$ and $\kappa$ is the identity function, we know that the solution always stays positive \cite[Chapter XI, Section 1]{revuz2013continuous}. By computations similar to \cite[Prop. 2.21]{bahouri2011fourier}, the drift $b(x)= \kappa(x) \, |x|^{-\alpha}$ belongs to $\mathcal{B}_\infty^{-\alpha}$ for $\alpha \in (0,1)$. In this case, \eqref{eqcond-gamma-p-H} reads $H < \frac{1}{2(1+\alpha)}$ and the rate of convergence of the tamed Euler scheme is close to $\frac{1}{2(1+\alpha)}$.

\subsection{Other examples in higher dimension}
A way to extend processes \eqref{eqbessel} to dimension $2$ could be the following:
\begin{align}\label{eqbessel2D}
\diff X_t^{i} = \frac{\kappa(X_{t})}{|X_t|^{\alpha}} \diff t + \diff B_t^i , \ i=1,2,
\end{align}
where $B^1$ and $B^2$ are two independent fBms and $\alpha>0$. By \cite[Proposition 2.21]{bahouri2011fourier}, one can prove that $x \mapsto b(x)= \frac{\kappa(x)}{| x|^{\alpha}}$ belongs to $\mathcal{B}^{-\alpha}_\infty$ for $\alpha \in (0,2)$. Therefore, the condition on $H$ becomes $H< \frac{1}{2(1+\alpha)}$. \\
Notice that the SDE \eqref{eqbessel2D} presents a singularity only at the point $(0,0)$. To create a singularity on both the $x$ and $y$-axes, one could also look at the following SDE
\begin{align*}%
\diff X_t^{i} = \frac{1}{(|X_t^1| \wedge |X_t^2|)^{\alpha}} \diff t + \diff B_t^i , \ i=1,2.
\end{align*}

\smallskip

Another example to consider in higher dimension is an SDE with discontinuous drift. For instance, let the drift be an indicator function of some domain $D$ as in \eqref{eqind2D}:
\begin{align}\label{eqind2D}
\diff X_t= \mathds{1}_{D}^{(d)}(X_t) \diff t + \diff B_t,
\end{align}
where $\mathds{1}_{D}^{(d)}$ denotes the vector-valued indicator function with identical entries $\mathds{1}_{D}$ on each component. 
We have $\mathds{1}_{D}^{(d)} \in \mathcal{B}_\infty^0$, and thus one can take $H < 1/2$.

\subsection{Simulations}\label{subsecsim}

In dimension $1$, 
we will simulate two SDEs. First the skew fractional Brownian motion \eqref{eqskewfbm} with $\alpha=1$. Then we simulate the SDE with bounded measurable drift $\mathds{1}_{\R_{+}} \in \mathcal{B}_\infty^0$, \emph{i.e.}
\begin{align}\label{eqsimubounded}
\diff X_t = \mathds{1}_{X_{t}>0} \diff t + \diff B_t .
\end{align}
We check that the rate of convergence obtained in Corollary~\ref{corbn-choice} holds numerically. The drifts are approximated by convolution with the Gaussian kernel, that is $b^k (x) = G_{1/k} b (x)$ (recall that $b^k$ satisfies the assumptions of Corollary \ref{corbn-choice} by Lemma \ref{lemreg-S}), and we fix the initial condition to $X_0=0$. For the skew fBm, this corresponds to 
$$b^k(x) = \sqrt{\frac{k}{2 \pi}} e^{-\frac{k x^2}{2}} ,$$
and for \eqref{eqsimubounded} this yields
$$b^k(x) = \sqrt{\frac{k}{2 \pi}} \int_0^x e^{-\frac{k y^2}{2}} \diff y .$$
As in Corollary~\ref{corbn-choice}, we fix the parameter $k$ of the mollifier in the tamed Euler scheme as $k=\lfloor h^{-\frac{1}{1-\gamma}}\rfloor$. The solutions of \eqref{eqskewfbm} and \eqref{eqsimubounded} do not have an explicit expression so we do not have an exact reference value. Instead, we first make a costly computation with very small time-step $h=2^{-7}\, 10^{-4}$ that will serve as reference value. 
In a second step, we compute the tamed Euler scheme for $h\in\{2^{-1}\, 10^{-4},\ 2^{-2}\, 10^{-4},\ 2^{-3}\, 10^{-4},\ 2^{-4}\, 10^{-4} \}$ and compare it to the reference value with the same noise and $h=2^{-7}\, 10^{-4}$. The result is averaged over $N=50000$ realisations of the noise to get an estimate of the strong error.

~

In dimension $2$, we simulate the $2$-dimensional SDE \eqref{eqind2D} with $X_0=0$ and with $D$ the quadrant defined by $D = \{x=(x_{1},x_{2})\colon~ x_{1} \ge 0, x_{2} \ge 0 \}$. The drift $b$ is approximated again by the convolution with the Gaussian kernel 
\begin{align*}
b^k(x) = G_{\frac{1}{k}} b (x) = \frac{k}{2 \pi } \int_{\mathbb{R}^2} e^{-\frac{k}{2}|x-y|^2} \mathds{1}_{D}(y)  \diff y.
\end{align*}

\smallskip

Recall that according to Corollary~\ref{corbn-choice}, the theoretical order of convergence is almost $1/2$ when the drift is bounded and almost $1/4$ when the drift is a Dirac distribution. We plot the logarithmic strong error with respect to the time-step $h$ for several values of the Hurst parameter, in Figure \ref{figsimbounded} for the Equations \eqref{eqsimubounded} and \eqref{eqind2D}, and in Figure \ref{figsimDirac} for the Equation \eqref{eqskewfbm}. We conclude that the empirical order of convergence is consistent with the theoretical one.

\begin{figure}[h!]
\centering
\includegraphics[scale=0.35]{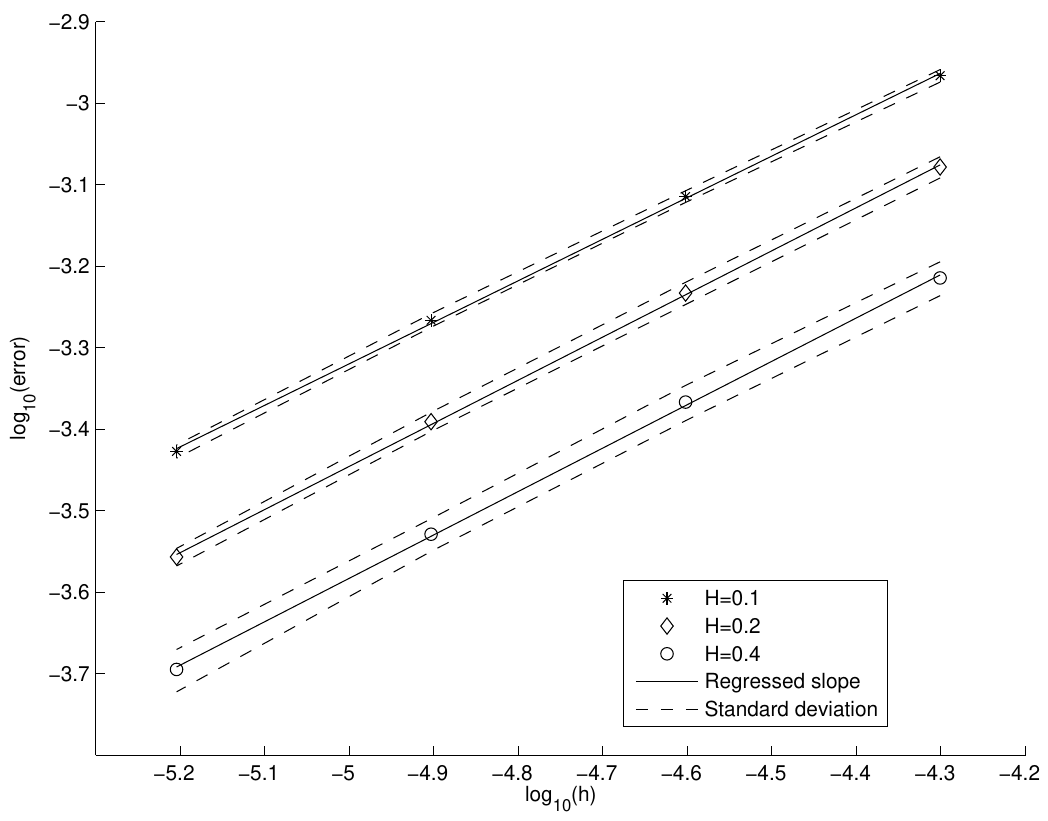} \,
\includegraphics[scale=0.35]{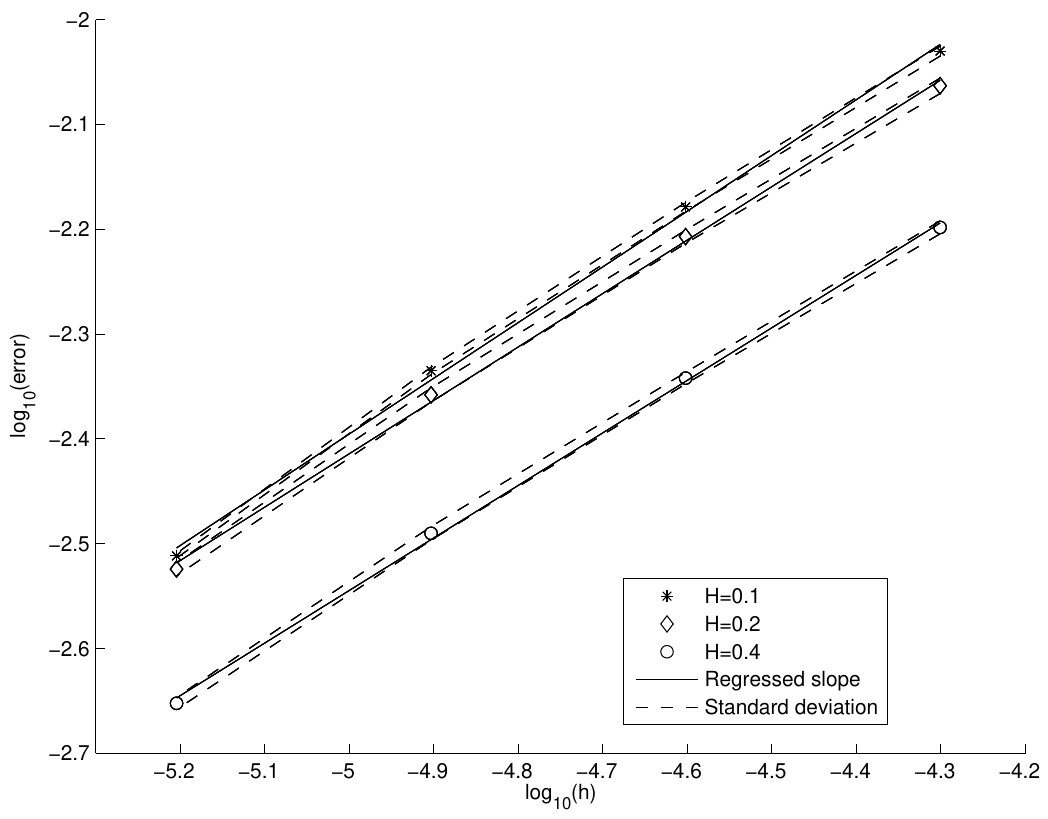}
\caption[Plot of the logarithm of the strong error for bounded drift]{Plot of the logarithm of the strong error ($y$-axis) against $h$ ($x$-axis) for a bounded drift. Left: Equation~\eqref{eqsimubounded} ($d=1$), the calculated slopes of the linear regression are approximately $0.51, 0.53, 0.53$.  - Right: Equation \eqref{eqind2D} $(d=2)$, the calculated slopes are approximately $0.53, 0.51, 0.50$. The standard deviation are plotted in dashed lines. For different values of $H<1/2$, and in both dimension $1$ and $2$, we observe that the numerical order of convergence is close to the theoretical order $1/2$.
}
\label{figsimbounded}
\end{figure}

\begin{figure}[h!]
\centering
\includegraphics[scale=0.35]{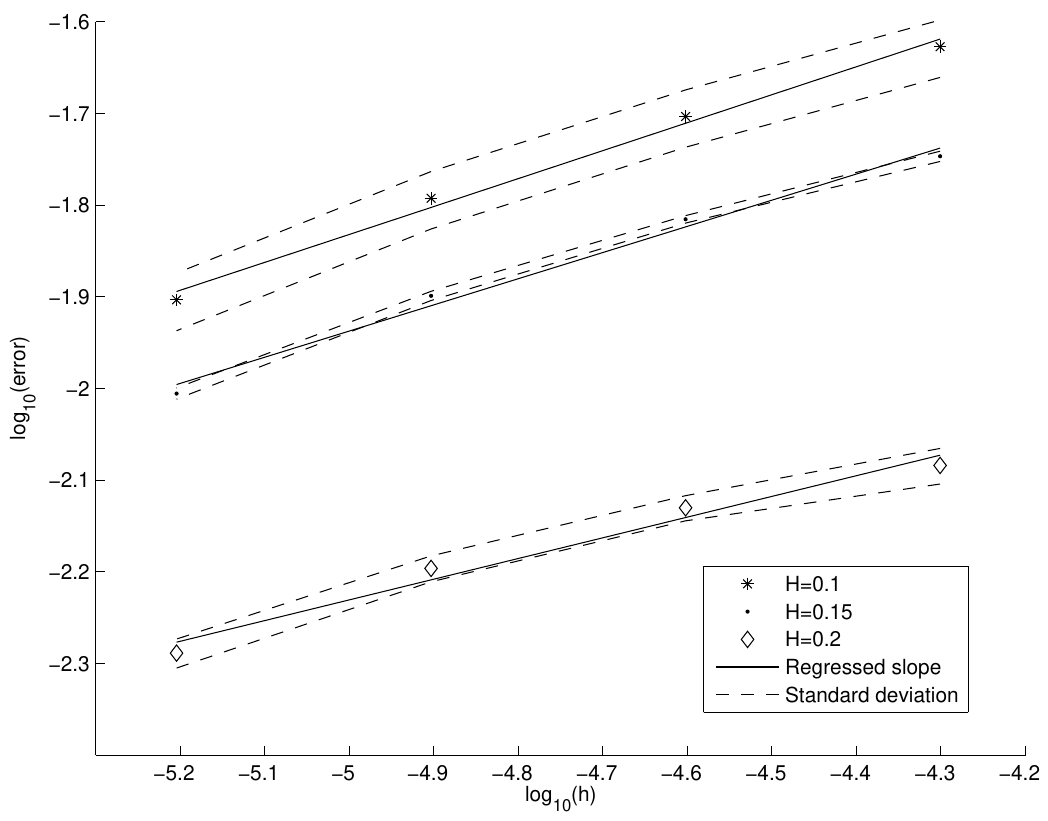}
\caption[Plot of the logarithm of the strong error for a Dirac drift]
{Plot of the logarithm of the strong error ($y$-axis) against $h$ ($x$-axis) for a Dirac drift in dimension $d=1$, namely Equation~\eqref{eqskewfbm}. For several values of $H<1/2$,  the calculated slopes of the linear regression are, from top to bottom, approximately $0.30, 0.29, 0.23$. The standard deviations are plotted in dashed lines. We observe that the numerical order of convergence is a bit far from the theoretical order $1/4$, hence further investigation (considering a larger number of realisations or smaller time-steps) is required to achieve a better rate. In particular, the difficulty comes from the strong singularity of the Dirac drift and a less precise simulation of the fBm as $H$ gets smaller, which could explain why the rate is further from $1/4$ when $H$ is smaller.
}
\label{figsimDirac}
\end{figure}

\paragraph{Acknowledgments.} The authors thank the anonymous referees for their comments which have led us to improve the presentation of this paper.

\begin{appendices}

\section{Regularisation properties of the fractional Brownian motion}\label{app:sec:regfBm}
We start by recalling an extension of the stochastic sewing Lemma \cite{le2020stochastic} with singular weights that was established in \cite{athreya2020well}. It is useful for the main estimates of Section~\ref{secstochastic-sewing}. 

For $\alpha \in[0,1)$ and $(s, t) \in \Delta_{S, T}$ we define $
\nu_{S, T}^{(\alpha)}(s, t)  := \int_{s}^{t}(r-S)^{-\alpha} d r$,
which satisfies
\begin{align}\label{eqbounds-nu}
\nu_{S, T}^{(\alpha)}(s, t) \leq C\, (t-s)^{1-\alpha} .
\end{align}

\begin{lemma}[\cite{athreya2020well}] \label{lemSSL}
Let $0\leq S<T$, $m \in [2, \infty)$ and $q \in [m,\infty]$. Let $(\Omega,\mathcal{F},\mathbb{F},\mathbb{P})$ be a filtered probability space. Let $A : \Delta_{S,T} \rightarrow L^m$ be such that $A_{s,t}$ is $\mathcal{F}_t$-measurable for any $(s,t) \in \Delta_{S,T}$. Assume that there exist constants $\Gamma_1,\Gamma_2\geq 0$, $\alpha_1 \in [0,1)$, $\alpha_2 \in [0, \frac{1}{2})$ and $\varepsilon_1,\varepsilon_2>0$ such that for any $(s,t) \in \Delta_{S,T}$ and $u = (s+t)/2$,
\begin{align}
    \|\EE^s[\delta A_{s,u,t}]\|_{L^q}&\leq \Gamma_1 \, (u-S)^{-\alpha_1} (t-s)^{1+\varepsilon_1},\label{sts1}\\
   \| \big( \EE^S | \delta A_{s,u,t} |^{m} \big)^{\frac{1}{m}} \|_{L^q} &\leq \Gamma_2 \, (u-S)^{-\alpha_2} (t-s)^{\frac{1}{2}+\varepsilon_2}. \label{sts2}
\end{align}
Then there exists a process $(\mathcal{A}_t)_{t\in [S,T]}$ such that, for any $t \in [S,T]$ 
and any sequence of partitions $\Pi_n=\{t_i^n\}_{i=0}^{N_n}$ of $[S,t]$ with mesh size going 
to zero, we have
\begin{align} \label{sts3}
    \mathcal{A}_t=\lim_{k\rightarrow \infty}\sum_{i=0}^{N_n-1}A_{t_i^n,t_{i+1}^n} \text{ in probability.} 
\end{align}

Moreover, there exists a constant $C=C(\varepsilon_1,\varepsilon_2,m, \alpha_1, \alpha_2)$ independent of $S,T$ such that for every $(s,t) \in \Delta_{S,T}$ we have
\begin{align*}
    \| \big( \EE^S | \mathcal{A}_t-\mathcal{A}_s-A_{s,t} |^m \big)^{\frac{1}{m}}  \|_{L^q} \leq C\, \Gamma_1 \nu_{S, T}^{(\alpha_{1}) }(s, t) (t-s)^{ \varepsilon_1} + C\, \Gamma_2 \Big( \nu_{S, T}^{(2 \alpha_{2})}(s, t)  \Big)^{\frac{1}{2}} (t-s)^{ \varepsilon_2},
\end{align*}
and
\begin{align*}
    \|\EE^S[\mathcal{A}_t-\mathcal{A}_s-A_{s,t}]\|_{L^{q}}\leq C\, \Gamma_1\, \nu_{S,T}^{(\alpha_1)}(s,t)\, (t-s)^{ \varepsilon_1}.
\end{align*}
\end{lemma}
\begin{remark}\label{rmkcritical-sewing}
\begin{itemize}
\item In this paper, the stochastic sewing Lemma is applied for only two possible values of $q$, that is $q=\infty$ or $q=m$, and for the latter case we have $\| \big( \EE^S |  \cdot  |^m \big)^{\frac{1}{m}} \|_{L^m} = \| \cdot \|_{L^m}$. 
\item A critical-exponent version of the stochastic sewing Lemma, introduced in \cite[Theorem 4.5]{athreya2020well} is used in Proposition \ref{propbound-E1-SDE-critic}. Under the same notations and assumptions as Lemma \ref{lemSSL} (with $q=m$ and $\alpha_{1}=\alpha_{2}=0$), assuming moreover that there exist $\Gamma_3, \Gamma_4, \varepsilon_4 >0$ such that
\begin{align}\label{sts4}
\left\|\EE^s\left[\delta A_{s, u, t}\right]\right\|_{L^m} \leq \Gamma_3|t-s|+\Gamma_4|t-s|^{1+\varepsilon_4},
\end{align}
we get that for $(s,t) \in \Delta_{S,T}$,
\begin{align}\label{sts5}
\left\|\mathcal{A}_t-\mathcal{A}_s-A_{s, t}\right\|_{L^m} \leq C \Gamma_3\left(1+\left|\log \frac{\Gamma_1 T^{\varepsilon_1}}{\Gamma_3}\right|\right)(t-s)+C \Gamma_2(t-s)^{\frac{1}{2}+\varepsilon_2}+C \Gamma_4(t-s)^{1+\varepsilon_4}.
\end{align}
\end{itemize}
\end{remark}

In the rest of this section, we gather the proofs of Lemma~\ref{lem1streg} and Proposition~\ref{propregfBm}.

\subsection{ Proof of Lemma~\ref{lem1streg}}\label{app1streg}

We will apply Lemma~\ref{lemSSL} for $S\leq s \leq t \leq T$,
\begin{align*}
    \mathcal{A}_t \colon = \int_S^t f(B_r,\Xi) \, dr  ~~\text{and}~~ A_{s,t} \colon= \EE^s\left[\int_s^t f(B_r,\Xi) \, dr\right].
\end{align*}
Notice that we have $\EE^s[\delta A_{s,u,t}]=0$, so \eqref{sts1} trivially holds.
In order to establish \eqref{sts2}, we will show that for some $\varepsilon_{2}>0$,
\begin{align} \label{(4.8)-critic}
    \|\delta A_{s,u,t}\|_{L^q}\leq \Gamma_2 \, (t-s)^{\frac{1}{2}+\varepsilon_2} (u-S)^{-\frac{dH}{p}}.
\end{align} 
For $u=(s+t)/2$ we have by the triangle inequality, Jensen's inequality for 
conditional expectation and Lemma~\ref{lemreg-B}$(iv)$ (recall that $q \leq p$) that 
\begin{align*}
    \|\delta A_{s,u,t}\|_{L^q}&\leq \left\|\EE^s\left[\int_u^t f(B_r,\Xi) \, dr\right]\right\|_{L^q} + \left\|\EE^u\left[\int_u^t  f(B_r,\Xi) \, dr\right]\right\|_{L^q}\\
    &\leq \int_u^t\left(\|\EE^s f(B_r,\Xi) \|_{L^q}+\|\EE^u f(B_r,\Xi) \|_{L^q}\right) dr\\
    &\leq 2 \int_u^t \|\EE^u f(B_r,\Xi) \|_{L^q} \, dr \\
    &\leq C\int_u^t \|\| f(\cdot,\Xi) \|_{\mathcal{B}_p^{\beta}}\|_{{L^q}} (r-u)^{H\beta} (u-S)^{-\frac{d}{2p}} (r-S)^{d\frac{1-2H}{2p}} \, dr\\ &\leq C \, \| \| f(\cdot,\Xi) \|_{\mathcal{B}_p^{\beta}}\|_{L^q}\,  (t-u)^{1+H\beta} (u-S)^{-\frac{dH}{p}},
\end{align*}
where we used $r-S \leq 2(u-S)$ for the last inequality. Hence, we 
have \eqref{(4.8)-critic} for $\varepsilon_2=1/2+H \beta >0$.

Let $t\in [S,T]$. Let $(\Pi_n)_{n \in \mathbb{N}}$ be a 
sequence of partitions of $[S,t]$ with mesh size converging to zero. For each $n$, denote $\Pi_n=\{t_i^n\}_{i=1}^{N_n}$. By 
Lemma~\ref{lemreg-B}$(iii)$ we have that
\begin{align*}
    \left\|\mathcal{A}_t-\sum_{i=1}^{N_{n}-1} A_{t^n_i,t^n_{i+1}} \right\|_{L^1}&\leq\sum_{i=1}^{N_{n}-1} \int_{t^n_i}^{t^n_{i+1}}\|f(B_r,\Xi)-\EE^{t_i^n} f(B_r,\Xi) \|_{L^1} dr\\
    &\leq C \, \|\|f(\cdot,\Xi) \|_{\mathcal{C}^1}\|_{L^2}\, (t-S)\, |\Pi_n|^H \longrightarrow 0.
\end{align*}
Hence \eqref{sts3} holds true.

Applying Lemma~\ref{lemSSL}, we get
\begin{align}\label{eqlem34-ssl-goal}
   \| \big( \EE^S | \mathcal{A}_t-\mathcal{A}_s|^m \big)^{\frac{1}{m}}  \|_{L^q} &\leq \|A_{s,t}\|_{L^q} + C \, \| \| f(\cdot,\Xi) \|_{\mathcal{B}_p^{\beta}}\|_{L^q}\, \Big( \nu_{S,T}^{(\frac{2dH}{p})}(s,t)\Big)^{\frac{1}{2}} (t-s)^{\frac{1}{2} + H\beta}.
\end{align}
To bound $\|A_{s,t}\|_{L^q}$, notice that
\begin{align}\label{eqboundAst}
   \|A_{s,t}\|_{L^q} &=\Big\|\EE^s \int_s^t f(B_r,\Xi) \, dr \Big\|_{L^q} \leq \int_s^t \|\EE^s f(B_r,\Xi) \|_{L^q} dr .
\end{align}

Hence to obtain \eqref{eqregulINT}, use Lemma~\ref{lemreg-B}$(ii)$ and recall that $1+H(\beta-\frac{d}{p}) >0$ to get that
\begin{align*}
   \|A_{s,t}\|_{L^q}  &\leq C  \int_s^t \| \| f(\cdot,\Xi) \|_{\mathcal{B}_p^{\beta}}\|_{L^q}\, (r-s)^{H(\beta-\frac{d}{p})} dr \\
    &\leq C \, \| \| f(\cdot,\Xi) \|_{\mathcal{B}_p^{\beta}}\|_{L^q}\,  (t-s)^{1+H(\beta-\frac{d}{p})} .
\end{align*}
Plugging the previous inequality in \eqref{eqlem34-ssl-goal} with \eqref{eqbounds-nu} yields \eqref{eqregulINT}.

\subsection{ Proof of Proposition~\ref{propregfBm}}\label{appregfBm}

Let $(S,T)\in \Delta_{0,1}$. For $(s,t) \in \Delta_{S,T}$, let
\begin{align} \label{eqA}
     \mathcal{A}_{t} \colon= \int_S^{t} f(B_r+\psi_r) \, dr  \quad \text{and}\quad A_{s,t} \colon = \int_{s}^{t} f(B_r+\psi_{s}) \, dr.
\end{align}

\paragraph{Proof of \ref{item3.5(a)}.}

Assume that $[\psi]_{\mathcal{C}^\tau_{[S,T]}L^{m,q}}<\infty$, otherwise 
\eqref{eq3.5a} trivially holds. 
In the last part of this proof, we will check that the conditions in order to apply Lemma~\ref{lemSSL} are verified. Namely, we will show that \eqref{sts1} and \eqref{sts2} hold true with $\varepsilon_1=H(\beta-1)+\tau>0$, $\alpha_1=0$ and $\varepsilon_2=1/2+H\beta >0$, $\alpha_2=0$, so that there exists a constant $C>0$ independent of $S,T,s,t$ such that
\begin{enumerate}[label=(\roman*$_{a}$)]
\item \label{en:(1a)} $\|\EE^{s} [\delta A_{s,u,t}]\|_{L^q}\leq C\, \|f\|_{\mathcal{B}_\infty^\beta}\, [\psi]_{\mathcal{C}^\tau_{[S,T]}L^{m,q}}\, (t-s)^{1+H(\beta-1)+\tau}$\text{;} %

\item \label{en:(2a)} $\Big\| \big( \EE^S |\delta A_{s,u,t} |^m \big)^{\frac{1}{m}} \Big\|_{L^q }  \leq C\, \| f \|_{\mathcal{B}_\infty^\beta}  (t-s)^{1+H\beta}$\text{;} 

\item \label{en:(3a)} If \ref{en:(1a)} and \ref{en:(2a)} are satisfied, \eqref{sts3} gives the 
convergence in probability of $\sum_{i=1}^{N_n-1} A_{t^n_i,t^n_{i+1}}$ along any sequence of 
partitions $\Pi_n=\{t_i^n\}_{i=1}^{N_n}$ of $[S,t]$ with mesh converging to $0$. We will prove 
that the limit is the process $\mathcal{A}$ given in \eqref{eqA}.
\end{enumerate}

Assume for now that \ref{en:(1a)}, \ref{en:(2a)} and \ref{en:(3a)} hold. Applying Lemma~\ref{lemSSL} and recalling \eqref{eqbounds-nu}, we obtain that
\begin{align*}
   \Big\| \Big( \EE^S \Big | \int_{s}^{t} f(B_r+\psi_r) \, dr \Big|^m \Big)^{\frac{1}{m}} \Big\|_{L^q} 
   &\leq C \, \|f\|_{\mathcal{B}_\infty^\beta}\, [\psi]_{\mathcal{C}^\tau_{[S,T]}L^{m,q}}\, (t-s)^{1+H(\beta-1)+\tau} \\ 
   & \quad + C\, \| f \|_{\mathcal{B}_\infty^\beta}  (t-s)^{1+H\beta} + \big\| \big( \EE^S | A_{s,t} |^m \big)^{\frac{1}{m}} \big\|_{L^q}.
\end{align*}

To bound $\big\| \big( \EE^S | A_{s,t} |^m \big)^{\frac{1}{m}} \big\|_{L^q}$, we apply Lemma~\ref{lem1streg} with $p=\infty$ to $\Xi = \psi_s$.
As $f$ is smooth and bounded, the first assumption of Lemma~\ref{lem1streg} is verified. By Lemma~\ref{lembesov-spaces}$(i)$, $ \|f(\cdot + \psi_s) \|_{\mathcal{B}^{\beta}_{\infty}} \leq \| f\|_{\mathcal{B}^{\beta}_{\infty}}$, hence the second assumption of Lemma~\ref{lem1streg} is verified. It follows by Lemma~\ref{lem1streg} that 
\begin{align}\label{eqAst}
\big\| \big( \EE^S | A_{s,t} |^m \big)^{\frac{1}{m}} \big\|_{L^q}  &\leq C\,  \| \|f(\psi_s+\cdot) \|_{\mathcal{B}_\infty^{\beta}} \|_{L^q} \,  (t-s)^{1+H\beta}\ \nonumber\\
&\leq C\, \|f\|_{\mathcal{B}^{\beta}_\infty} \, (t-s)^{1+H\beta}.
\end{align}
Then, 
we get \eqref{eq3.5a}.

\paragraph{Proof of \ref{item3.5(b)}.}
Assume that $[\psi]_{\mathcal{C}^{1/2+H}_{[S,T]}L^m}<\infty$, otherwise 
\eqref{eq3.5b} trivially holds. In the last part of this proof, 
we will check that the conditions in order to apply the stochastic sewing Lemma with critical exponent \cite[Theorem 4.5]{athreya2020well} are verified.
 Namely, we will show that  for some $\varepsilon \in (0,1)$ small enough (specified later), \eqref{sts1}, \eqref{sts2} and \eqref{sts4} hold true with $\varepsilon_1=H >0$, $\alpha_1=0$, $\varepsilon_2=\varepsilon/2>0$, $\alpha_2=0$ and $\Gamma_{4}=0$,
 so that there exists a constant $C>0$ independent of $s,t,S$ and $T$ such that
\begin{enumerate}[label=(\roman*$_{b}$)]
\item\label{en:(1b)fbm-app}  $\|\EE^{s} [\delta A_{s,u,t}]\|_{L^m}\leq C\, \| f \|_{\mathcal{B}_p^{\beta+1}}\, [\psi]_{\mathcal{C}^{\frac{1}{2}+H}_{[S,T]}L^m}\, (t-s)^{1+H}$\text{;}

\myitem{(i$^\prime_{b}$)}\label{en:(1'b)fbm-app} $\| \EE^s [\delta A_{s,u,t}]\|_{L^m}\leq C\, \|f\|_{\mathcal{B}_p^{\beta}}\, [\psi]_{\mathcal{C}^{\frac{1}{2}+H}_{[S,T]}L^m}\, (t-s)$ \text{;}

\item\label{en:(2b)fbm-app} $\| \delta A_{s,u,t} \|_{L^m} \leq C\, \| f \|_{\mathcal{B}_p^{\beta}}  \Big( 1+  [\psi]_{\mathcal{C}^{\frac{1}{2}+H}_{[S,T]} L^m} \Big) (t-s)^{\frac{1}{2}+ \frac{\varepsilon}{2}}$\text{;} 

\item\label{en:(3b)fbm-app}  If \ref{en:(1b)fbm-app} and \ref{en:(2b)fbm-app} are satisfied, \eqref{sts3} gives the 
convergence in probability of $\sum_{i=1}^{N_n-1} A_{t^n_i,t^n_{i+1}}$ along any sequence of 
partitions $\Pi_n=\{t_i^n\}_{i=1}^{N_n}$ of $[S,t]$ with mesh converging to $0$. We will prove 
that the limit is the process $\mathcal{A}$ given in \eqref{eqA}.
\end{enumerate}
Assume for now that \ref{en:(1b)fbm-app}, \ref{en:(1'b)fbm-app}, \ref{en:(2b)fbm-app} and \ref{en:(3b)fbm-app} hold. Applying \cite[Theorem 4.5]{athreya2020well}, we obtain that
\begin{align*}
    \Big\| \int_{s}^{t} f(B_r+\psi_r) \, dr \Big\|_{L^m} &\leq C\, \|f\|_{\mathcal{B}_p^{\beta}}\, [\psi]_{\mathcal{C}^{\frac{1}{2}+H}_{[S,T]}L^m} \Big(1+\left| \log\frac{\| f \|_{\mathcal{B}_p^{\beta+1}} t^{\varepsilon_1}}{\| f \|_{\mathcal{B}_p^{\beta}}} \right| \Big)\, (t-s)  \\ 
    & \quad  + C\, \| f \|_{\mathcal{B}_p^{\beta}}  \Big( 1+  [\psi]_{\mathcal{C}^{\frac{1}{2}+H}_{[S,T]} L^m} \Big) (t-s)^{\frac{1}{2} + \frac{\varepsilon}{2}} \\  
    & \quad + \| A_{s,t}\|_{L^m} .
\end{align*}
To bound $\| A_{s,t} \|_{L^m}$, we use again Lemma~\ref{lem1streg} with $\beta-\frac{d}{p}=-\frac{1}{2H}$ and proceed as for \eqref{eqAst}  to get $\|A_{s,t} \|_{L^m}\leq C \|f\|_{\mathcal{B}^{\beta}_p} (t-s)^{\frac{1}{2}}$. Hence 
 we get \eqref{eq3.5b}.

~

We now check that the conditions \ref{en:(1a)}, \ref{en:(2a)}, \ref{en:(3a)}, \ref{en:(1b)fbm-app}, \ref{en:(1'b)fbm-app}, \ref{en:(2b)fbm-app} and \ref{en:(3b)fbm-app} actually hold.

\smallskip

\paragraph{Proof of \ref{en:(1a)}, \ref{en:(1b)fbm-app} and \ref{en:(1'b)fbm-app}.} 
Let $p\in [m,+\infty]$ which will be fixed at the end of this paragraph.
For 
$(s,t) \in \Delta_{S,T}$, 
we have
    $\delta A_{s,u,t}= \int_u^{t} f(B_{r}+\psi_{s})-f(B_{r}+\psi_u) \, dr$. 
Hence, by the tower property of conditional expectation and Fubini's 
Theorem, we get 
\begin{align*}
    |\EE^{s} \delta A_{s,u,t}|&=
    \Big|\EE^{s} \int_u^{t} \EE^u [f(B_{r}+\psi_{s})-f(B_{r}+\psi_u)]
    \, dr \Big|.
\end{align*}
Now using Lemma~\ref{lemreg-B}$(ii)$ with the $\mathcal{F}_{u}$-measurable variable $\Xi=(\psi_{s},\psi_{u})$
 and using again Fubini's Theorem, we obtain that for $\lambda\in [0,1]$, %
\begin{align}\label{eqLqBound(1a)}
    \Big\|\EE^{s} \int_u^{t} \EE^u [f(B_{r}+\psi_{s})-f(B_{r}+\psi_u)]
    \, dr \Big\|_{L^q}
    &\leq \int_u^{t} \| \mathbb{E}^{s} \|
    f(\cdot+\psi_{s})-f(\cdot+\psi_u)
    \|_{\mathcal{B}_p^{\beta-\lambda}}  \|_{L^q} \, (r-u)^{H(\beta-\lambda-\frac{d}{p})}  \, dr \nonumber \\
    &\leq C \|f\|_{\mathcal{B}_p^{\beta-\lambda+1}} \, \| \mathbb{E}^{s} |\psi_u-\psi_{s}|\|_{L^q} \int_u^{t}  (r-u)^{H(\beta-\lambda-\frac{d}{p})} \, dr.
\end{align}
By the conditional Jensen inequality and \eqref{eqdefbracket} (recall that $m\leq q$), we have
\begin{align}\label{eqconditionalIncPsi}
\left\|\mathbb{E}^{s} \left|\psi_u-\psi_{s} \right|\right\|_{L^q} \leq [\psi]_{\mathcal{C}_{[s, t]}^\tau L^{m, q}}\, (u-s)^\tau .
\end{align}

In the sub-critical case, choosing $\lambda=1$ and $p=\infty$ in \eqref{eqLqBound(1a)}, we get \ref{en:(1a)}.

In the limit case, let $\tau=1/2+H$ and $p<\infty$. For $q=m$, we get from \eqref{eqLqBound(1a)} and \eqref{eqconditionalIncPsi} that
\begin{align*} 
\| \EE^{s} \delta A_{s,u,t} \|_{L^m} &  \leq C\, \|f\|_{\mathcal{B}_p^{\beta-\lambda+1}} \, [\psi]_{\mathcal{C}^\tau_{[S,T]} L^m} \, (t-s)^{1+H(\beta-\lambda-\frac{d}{p})+\tau} .
\end{align*}
Choosing $\lambda=1$ in the previous inequality, we get \ref{en:(1'b)fbm-app}. While choosing $\lambda=0$ yields \ref{en:(1b)fbm-app}.

\paragraph{Proof of \ref{en:(2a)}.} We write
\begin{align*}
\left\| \big( \EE^S |\delta A_{s,u,t} |^m \big)^{\frac{1}{m}} \right\|_{L^q }  \leq \left\| \big( \EE^S |\delta A_{s,t} |^m \big)^{\frac{1}{m}} \right\|_{L^q} 
+ \left\| \big( \EE^S |\delta A_{s,u} |^m \big)^{\frac{1}{m}} \right\|_{L^q } + \left\| \big( \EE^S |\delta A_{u,t} |^m \big)^{\frac{1}{m}} \right\|_{L^q }  ,
\end{align*}
Recall that we already obtained a bound on $\| ( \EE^S |\delta A_{s,t} |^m )^{1/m} \|_{L^q} $ in \eqref{eqAst}. We obtain similar bounds for $\| ( \EE^S |\delta A_{s,u} |^m )^{1/m} \|_{L^q}$ and $\| (\EE^S |\delta A_{u,t} |^m )^{1/m} \|_{L^q} $, which yields
\begin{align*}
\left\| \big( \EE^S |\delta A_{s,u,t}|^m \big)^{\frac{1}{m}} \right\|_{L^q} 
&\leq C\, \| f \|_{\mathcal{B}_\infty^\beta} \Big( (t-s)^{1+H\beta}  +(u-s)^{1+H\beta}  +(t-u)^{1+H\beta}   \Big) \\
& \leq C\, \| f \|_{\mathcal{B}_\infty^\beta}  (t-s)^{1+H\beta} .
\end{align*}

\paragraph{Proof of \ref{en:(2b)fbm-app}.} We choose $\varepsilon$ such that $\beta-\varepsilon > -1/2H$ and $\beta-\varepsilon-d/p > -1/H$. We apply now Lemma \ref{lem1streg} with $\beta\equiv\beta-\varepsilon$ and $\Xi = (\psi_{s}, \psi_u)$. As $f$ is smooth and bounded, the first assumption of Lemma~\ref{lem1streg} is verified. 
By Lemma~\ref{lembesov-spaces}$(i)$, $ \|f(\cdot + \psi_{s})-f(\cdot + \psi_{u}) \|_{\mathcal{B}^{\beta-\varepsilon}_{p}} \leq 2\| f\|_{\mathcal{B}^{\beta-\varepsilon}_{p}}$, hence the second assumption of Lemma~\ref{lem1streg} is verified. It follows by Lemma~\ref{lem1streg} and Lemma~\ref{lembesov-spaces}$(ii)$ that 
\begin{align*}
\| \delta A_{s,u,t} \|_{L^m}  & \leq C\, \| \| f(\cdot + \psi_{s})-(\cdot + \psi_{u}) \|_{\mathcal{B}_p^{\beta-\varepsilon}} \|_{L^m} (t-u)^{1+H(\beta-\varepsilon-\frac{d}{p})} \\
& \leq C\, \| f \|_{\mathcal{B}_p^{\beta}} \, \| |\psi_{s}-\psi_u |^{\varepsilon} \|_{L^m} (t-u)^{1+H(\beta-\varepsilon-\frac{d}{p})} .
\end{align*}
Hence by Jensen's inequality,
\begin{align*}
\| \delta A_{s,u,t} \|_{L^m}  & \leq C\, \| f \|_{\mathcal{B}_p^{\beta}} \| \psi_{s}-\psi_u \|_{L^m}^\varepsilon (t-u)^{1+H(\beta-\varepsilon-\frac{d}{p})} \\
& \leq  C\, \| f \|_{\mathcal{B}_p^{\beta}}  [\psi]_{\mathcal{C}^{\frac{1}{2}+H}_{[S,T]} L^m}^\varepsilon (t-s)^{1+H(\beta-\frac{d}{p}) + \frac{\varepsilon}{2}} \\
& \leq C\, \| f \|_{\mathcal{B}_p^{\beta}}  \Big( 1+  [\psi]_{\mathcal{C}^{\frac{1}{2}+H}_{[S,T]} L^m} \Big) (t-s)^{1+H(\beta-\frac{d}{p}) + \frac{\varepsilon}{2}} . 
\end{align*}

\paragraph{Proof of \ref{en:(3a)} and \ref{en:(3b)fbm-app}.} For a sequence $(\Pi_n)_{n \in 
\mathbb{N}}$ of partitions of $[S,t]$ with $\Pi_n=\{t_i^n\}_{i=1}^{N_n}$ and mesh size 
converging to zero, we have
\begin{align*}
    \Big\|\mathcal{A}_{t}-\sum_{i=1}^{N_{n}-1} A_{t_i^n,t_{i+1}^n}\Big\|_{L^m}
    &\leq \sum_{i=1}^{N_{n}-1} \int_{t_i^n}^{t_{i+1}^n} \| f(B_r + \psi_r)-f(B_r+\psi_{t_i^n})\|_{L^m} \, dr\\
    &\leq \sum_{i=1}^{N_{n}-1} \int_{t_i^n}^{t_{i+1}^n} \|f\|_{\mathcal{C}^1}\|\psi_r-\psi_{t_i^n}\|_{L^m} \, dr\\
    &\leq C\,  \|f\|_{\mathcal{C}^1} \,  |\Pi_n|^{\tau \wedge (\frac{1}{2}+H)}\,  [\psi]_{\mathcal{C}^{\tau \wedge (\frac{1}{2}+H)}_{[s,t]}L^m} \, \,
    \underset{n \rightarrow \infty}{\longrightarrow} 0.
\end{align*}

\section{Proof of the critical Gr\"onwall-type lemma: Lemma \ref{lemrate-critical}}\label{app:sec:gronwall1}

Assume without any loss of generality that $C_1, C_2 >0$. Let $\delta \in (0,\frac{1}{2} e^{-C_2})$. There exists $a > 1$ such that $e^{-C_2 \frac{a\log a}{a-1}} =  e^{-C_2}-\delta$. Let $\varepsilon \equiv \varepsilon(C_2,\delta) \in (0,1)$ small enough such that $e^{-C_2 \frac{a\log a}{(a-1)(1-\varepsilon)}} \geq e^{-C_2}-2 \delta$. Denote also $\alpha := 1- e^{-C_2 \frac{a\log a}{(a-1)(1-\varepsilon)}}$. Now for $\eta \in (0,1)$ and $f\in \mathcal{R}(\eta,\ell,C_{1},C_{2})$, define the following increasing sequence: $t_0=0$ and for $k \in \N$, 
\begin{equation*}
t_{k+1} = 
\begin{cases}
\inf \{t>t_{k}\colon ~  \eta+ \|f_t\| \ge a^{k+1}\, \eta \} \wedge 1 & \text{ if } t_{k}<1,\\
1 & \text{ if } t_{k}=1 ,
\end{cases}
\end{equation*}
with the convention that $\inf \emptyset = +\infty$. In view of \eqref{eqboundIncf} and of the boundedness of $f$, the mapping $t\mapsto \|f_t\|$ is continuous. In particular, and by definition of the sequence $(t_k)$, we deduce that for any $k$,
\begin{align*}
\|f\|_{L^\infty_{[0,t_{k}]}E} \leq a^k\eta - \eta \leq a^k \eta.
\end{align*}
Let 
\begin{equation}\label{eqdefN}
N = \left\lfloor - \alpha \frac{\log(\eta)}{\log(a)} \right\rfloor -1 ,
\end{equation}
and let $\bar{\eta}_{0} \equiv \bar{\eta}_{0}(C_{2},\delta)$ be such that for $\eta<\bar{\eta}_{0}$, we have $N\geq 1$.
We shall prove the following statement:
\begin{align}\label{eqstatementEpsBar}
\mbox{There exists $\bar{\eta} \equiv \bar{\eta}(C_1,C_{2},\ell, \delta)$ such that for any } \eta<\bar{\eta} \mbox{ and } f\in \mathcal{R}(\eta,\ell,C_{1},C_{2}), ~ t_{N+1}=1.
\end{align}
Observe that if \eqref{eqstatementEpsBar} holds true, then for $\eta<\bar{\eta}$ and $f\in \mathcal{R}(\eta,\ell,C_{1},C_{2})$, we have
\begin{align*}
 \| f \|_{L^\infty_{[0,1]} E} \leq  a^{N+1} \eta 
\leq  \eta^{1-\alpha} \leq \eta^{(e^{-C_2}-2\delta)} 
\end{align*}
and the lemma is proven.

Let us now prove the statement \eqref{eqstatementEpsBar}.  Fix $\eta < \bar{\eta}_{0}$ and $f\in \mathcal{R}(\eta,\ell,C_{1},C_{2})$. 
Let $N_{0} = \inf\left\{ k\in \N\colon~ t_{k+1} = 1 \right\}$.
We aim to prove that $N_{0}\leq N$, so that we will have indeed $t_{N+1} = 1$. First, if $N_{0} = 0$, we have obviously that $N\geq N_{0}$. Assume now that $N_{0}\geq 1$.
 For any $k \leq N_{0}-1$, we have $\eta+\|f_{t_{k}}\| = a^k\, \eta$ and $\eta+\|f_{t_{k+1}}\| = a^{k+1}\, \eta$, which implies that $ \|f_{t_{k+1}} - f_{t_{k}} \|\geq (a^{k+1}-a^k) \eta$. Consider two cases:

\begin{enumerate}[label=(\arabic*), labelwidth=!, labelindent=\parindent]

\item If $t_{k+1}-t_k \leq \ell$, then one can apply \eqref{eqboundIncf}, using $\| f_{t_{k+1}} \| = \| f \|_{L^\infty_{[t_k,t_{k+1}]} E}$, to get
\begin{align*}
(a^{k+1}-a^k) \eta &\leq C_1 \, a^{k+1} \eta \, (t_{k+1}-t_{k})^{\frac{1}{2}} + C_2 \, a^{k+1}\, \eta \, |\log (a^{k+1} \, \eta)| \, (t_{k+1}-t_{k}).
\end{align*} 

\item If $t_{k+1}-t_{k} > \ell$, then we split the interval  $[t_k, t_{k+1}]$ into at most $ \lfloor \frac{1}{\ell} \rfloor$ intervals of size $\ell$, that we denote $[\beta^{j}_k, \beta^{j+1}_k]$. We can apply \eqref{eqboundIncf} over each such interval to get that
\begin{align*}
(a^{k+1}-a^k) \eta & \leq C_1 \, \sum_{j=0}^{\lfloor \frac{1}{\ell} \rfloor} (\|f \|_{L^\infty_{[\beta^{j}_k,\beta^{j+1}_k]}E} + \eta) \,  (\beta^{j+1}_k-\beta^{j}_k)^{\frac{1}{2}} \\ 
&\quad + C_2 \sum_{j=0}^{\lfloor \frac{1}{\ell} \rfloor} \, (\|f \|_{L^\infty_{[\beta^{j}_k,\beta^{j+1}_k]}E} + \eta) \, |\log (\|f \|_{L^\infty_{[\beta^{j}_k,\beta^{j+1}_k]}E} + \eta) | \,(\beta^{j+1}_k-\beta^{j}_k).
\end{align*} 
By definition of the sequence $(t_k)_{k\in \N}$, we know that $\eta + \|f \|_{L^\infty_{[\beta^{j}_k,\beta^{j+1}_k]}E} \leq a^{k+1} \eta$. Moreover for $k\leq N$, there is $a^{k+1} \eta \leq a^{N+1} \eta \leq \eta^{e^{-C_2}-\delta}$. Therefore, for $\bar{\eta}_1 \equiv \bar{\eta}_1 (C_{2},\delta) \leq \bar{\eta}_{0}$ small enough, we have that $a^{k+1} \eta \leq e^{-1}$ for any $\eta < \bar{\eta}_1$ and $k\leq N$. Since the mapping $x\mapsto x\, |\log x|$ is nondecreasing on the interval $[0,e^{-1}]$, we get that 
$$ 
 (\|f \|_{L^\infty_{[\beta^{j}_k,\beta^{j+1}_k]}E}+\eta) \, |\log (\|f \|_{L^\infty_{[\beta^{j}_k,\beta^{j+1}_k]}E}+\eta) | \leq a^{k+1} \eta\, | \log ( a^{k+1} \eta) |.
$$
Then, by Cauchy-Schwarz on the first term, we write
\begin{align*}
(a^{k+1}-a^k) \eta & \leq C_1 \,  a^{k+1} \eta \,  \sqrt{\frac{1}{\ell}} (t_{k+1}-t_k)^{\frac{1}{2}} + C_2 \, a^{k+1}\, \eta \, |\log (a^{k+1} \eta) | \,(t_{k+1}-t_k).
\end{align*}

\end{enumerate}
Hence in both cases ($t_{k+1}-t_k \leq \ell$ or $t_{k+1}-t_k > \ell$), for any $\eta < \bar{\eta}_1$, there is
\begin{align*}
1 &\leq \frac{C_1}{\sqrt{\ell}} \frac{a}{a-1}\,  (t_{k+1}-t_{k})^{\frac{1}{2}} + C_2 \frac{a}{a-1} \,  |\log (a^{k+1} \eta) | \, (t_{k+1}-t_{k}) .
\end{align*}
Notice that the polynomial $C_2  \frac{a}{a-1} \, |\log (a^{k+1} \eta) | \, X^2 +  \frac{C_1}{\sqrt{\ell}} \frac{a}{a-1} \, X -1$ has only one non-negative root. Thus
\begin{align*}
(t_{k+1}-t_{k})^{\frac{1}{2}} \geq \frac{-\frac{a}{a-1} \frac{C_1}{\sqrt{\ell}} + \sqrt{(\frac{a}{a-1})^2 \frac{C_1^2}{\ell} + 4 \frac{C_2 a}{a-1} |\log (a^{k+1} \eta) | }}{2 \frac{C_2 a}{a-1} |\log (a^{k+1} \eta) |  } .
\end{align*}
Then we have
\begin{align*}
(t_{k+1}-t_{k}) \geq \frac{2\left( \frac{C_1 a}{\sqrt{\ell}(a-1)}\right)^2  +  4 \frac{C_2 a}{a-1} |\log (a^{k+1} \eta) | -\frac{2 C_1 a}{\sqrt{\ell}(a-1)} \sqrt{\left(\frac{C_1 a}{\sqrt{\ell}(a-1)}\right)^2 + 4 \frac{C_2 a}{a-1} |\log (a^{k+1} \eta) | }}{ \left( \frac{2 C_2 a}{a-1}\right)^2 |\log (a^{k+1} \eta) |^2  } .
\end{align*}
Using the notation $C_a = \frac{C_2 a}{a-1}$ and the inequality $\sqrt{a+b} \leq \sqrt{a}+\sqrt{b}$, we get that for any $\eta < \bar{\eta}_1$,
\begin{align}
(t_{k+1}-t_{k}) & \geq \frac{C_1^2}{2 C_2^2 \ell |\log (a^{k+1} \eta) |^2 }  + \frac{1}{C_a |\log (a^{k+1} \eta) |} - \frac{C_1^2}{2 C_2^2 \ell |\log (a^{k+1} \eta) |^2 } - \frac{ C_1}{ (\frac{a}{a-1})^{\frac{1}{2}} C_2^{\frac{3}{2}} \sqrt{\ell}\, |\log (a^{k+1} \eta) |^{\frac{3}{2}} } \nonumber \\
& \geq \frac{1}{C_a |\log (a^{k+1} \eta) |}  -  \frac{C_1}{C_2^{\frac{3}{2}} \sqrt{\ell}\,  |\log (a^{k+1} \eta) |^{\frac{3}{2}} }  \label{equpperb-tk}.
\end{align}

Now we will show that for $N$ defined in \eqref{eqdefN}, the sum from $0$ to $N$ of the right-hand side of \eqref{equpperb-tk} is larger than $1$, which implies that $N_0 \leq N$ since $\sum_{k=0}^{N_0} (t_{k+1}-t_k) = 1$. Let us start with the second term in the above inequality. Notice that for $k \leq N$, we always have $|\log (a^{k+1} \eta) | = |\log(\eta)| - (k+1) \log(a) $. Thus we get
\begin{align*}
\sum_{k=0}^{N-1} \frac{1}{|\log (a^{k+1} \eta) |^{\frac{3}{2}} } & \leq \int_0^{N} \frac{1}{ |\log (a^{x+1} \eta) |^{\frac{3}{2}} } \diff x \\
& = \frac{2}{\log(a)} \left( -\frac{1}{\sqrt{|\log (a \eta) | }} + \frac{1}{\sqrt{|\log (a^{N+1} \eta) |}}  \right).
\end{align*}
We have $|\log (a^{N+1} \eta) | = |\log(\eta)| - \lfloor \alpha \frac{|\log(\eta)|}{\log(a)} \rfloor \log(a) \ge (1-\alpha) | \log(\eta) | = e^{-C_2 \frac{a\log a}{(a-1)(1-\varepsilon)} }  | \log(\eta) |$. So we have $ |\log (a^{N+1} \eta) | \rightarrow \infty$ as $\eta \rightarrow 0$ and therefore,
\begin{align*}
 \lim_{\eta \rightarrow 0} \sum_{k=0}^{N} \frac{1}{ |\log (a^{k+1} \eta) |^{\frac{3}{2}} } &\leq  \lim_{\eta \rightarrow 0} \frac{2}{\log(a)} \Big( -\frac{1}{\sqrt{|\log (a \eta) | }} + \frac{1}{\sqrt{|\log (a^{N+1} \eta) |}}  \Big) + \frac{1}{|\log (a^{N+1} \eta) |^{\frac{3}{2}} }\\
& =0 .
\end{align*}
On the other hand,
\begin{align*}
\sum_{k=0}^{N} \frac{1}{C_a |\log (a^{k+1} \eta) |} &  \ge \frac{1}{C_a} \int_{-1}^{N-1} \frac{1}{|\log (a^{x+1} \eta) |} \diff x =  \frac{1}{\log(a) C_a}  \log \left( \frac{|\log ( \eta) |}{|\log (a^{N} \eta) |}    \right) \\
& =  \frac{1}{\log(a) C_a}  \log \left( \frac{|\log(\eta)|}{(1-\alpha) | \log(\eta)| + \alpha | \log(\eta)|- N \log(a) } \right)  .
\end{align*}
We have $N+1 = \lfloor \frac{\alpha | \log(\eta) |}{\log(a)} \rfloor \ge \frac{\alpha | \log(\eta) |}{\log(a)} - 1$, thus $N \log(a) + 2 \log(a) \ge \alpha | \log(\eta) |$. Hence,
\begin{align*}
& \ge  \frac{1}{\log(a) C_a} \log \left( \frac{|\log(\eta)|}{(1-\alpha) | \log(\eta)| + 2\log(a) } \right)  .
\end{align*}
The right hand side converges to $ \log(1/(1-\alpha))/ (C_a \log(a) )$ as $\eta$ goes to 0. Hence, going back to \eqref{equpperb-tk} and summing over $k \in \llbracket 0, N-1 \rrbracket$, we know that there exists $\bar{\eta} \equiv \bar{\eta} (C_1,C_2, \ell, \delta) \leq \bar{\eta}_{1}$ such that for $\eta <  \bar{\eta}$ , we have
\begin{align*}
\sum_{k=0}^{N} (t_{k+1}-t_k) \geq \frac{1}{C_a \log(a)} \log\left( \frac{1}{1-\alpha} \right) (1-\varepsilon) = 1 = \sum_{k=0}^{N_0} (t_{k+1}-t_k)  .
\end{align*}
It follows that $N_0 \leq N$ and thus $t_{N+1}=1$. Hence \eqref{eqstatementEpsBar} is true and we conclude that for $\eta <  \bar{\eta}$, we have 
 $\| f \|_{L^\infty_{[0,1]} E} \leq  \eta^{e^{-C_2}-2 \delta}$ .

\end{appendices}

\end{document}